\documentclass[11pt]{amsart}
\usepackage {amsmath, amssymb,   epsfig, a4wide, enumerate, psfrag, mathrsfs}
\usepackage[utf8]{inputenc}
\usepackage{graphicx,color}
\usepackage[cal=dutchcal]{mathalfa}
\date{\today}
\keywords{}
\author{Romain Dujardin}
\address{Sorbonne Universit\'es, Laboratoire de probabilit\'es, statistique et mod\'elisation, UMR 8001,  4 place Jussieu, 75005 Paris, France}
\email{romain.dujardin@upmc.fr}

\author{Charles Favre}
\address{CNRS - Centre de Math\'ematiques Laurent Schwartz, 
\'Ecole Polytechnique, 
91128 Palaiseau Cedex, France}
\email{charles.favre@polytechnique.edu}

\thanks{The first author is supported by the Institut Universitaire de France.
Both authors are supported by the ANR project LAMBDA,  ANR-13-BS01-0002. }

\title{Degenerations of $\mathrm{SL}(2, \mathbb{C})$ representations and Lyapunov exponents}

\newcommand{\cc}{\mathbb{C}}

\newcommand{\re}{\mathbb{R}}
\newcommand{\dd}{\mathbb{D}}

\newcommand{\nn}{\mathbb{N}}
\newcommand{\pp}{\mathbb{P}}
\newcommand{\hh}{\mathbb{H}}
\newcommand{\prob}{\mathsf{P}}
\newcommand{\phyb}{{\mathbb{P}^1_{\mathrm{hyb}}}}
\newcommand{\Rep}{\mathrm{rep}}
\newcommand{\Att}{\mathrm{att}}

\newcommand{\mm}{\mathbb{M}}
\newcommand{\cA}{\mathcal{A}}
\newcommand{\cD}{\mathcal{D}}
\newcommand{\cH}{\mathcal{H}}
\newcommand{\sm}{\mathcal{m}}
\newcommand{\cdt}{\C(\!(t)\!)}

\def\A{{\mathbb A}}
\def\C{{\mathbb C}}
\def\R{{\mathbb R}}
\def\Z{{\mathbb Z}}
\def\H{{\mathbb H}}
\def\Q{{\mathbb Q}}
\def\B{\mathbb{B}}
\def\D{\mathbb{D}}

\def\N{{\mathbb N}}

\def\cO{\mathcal{O}}

\def\cD{\mathcal{D}}

\def\sP{\mathsf{P}}
\def\sE{\mathsf{E}}

\def\an{\textup{an}}

\newcommand{\e}{\varepsilon}
\newcommand{\cv}{\rightarrow}
\newcommand{\cvf}{\rightharpoonup}

\newcommand{\om}{\Omega}
\newcommand{\set}[1]{\left\{#1\right\}}
\newcommand{\norm}[1]{\left\Vert#1\right\Vert}
\newcommand{\abs}[1]{\left\vert#1\right\vert}
\newcommand{\cd}{{\cc^2}}

\newcommand{\pu}{{\mathbb{P}^1}}

\newcommand{\PP}{\mathbb{P}}

\newcommand{\berk}{{\mathbb{P}^{1, \mathrm{an}}_k}}
\newcommand{\berkl}{\mathbb P^{1,\an}_{\Laurent}}
\newcommand{\Laurent}{\mathbb{L}}

\newcommand{\rest}[1]{ \arrowvert_{#1}}
\newcommand{\m}{{\bf M}}
\newcommand{\unsur}[1]{\frac{1}{#1}}

\newcommand{\cst}{ {C}^{st}}
\newcommand{\lrpar}[1]{\left(#1\right)}

 \newcommand{\la}{\lambda}

\newcommand{\att}{\mathrm{att}}
\newcommand{\rep}{\mathrm{rep}}
\newcommand{\na}{\mathsf{na}}

\newcommand{\SL}{\mathrm{SL}}
\newcommand{\SLC}{\mathrm{SL}(2,\mathbb C)}
\newcommand{\SLk}{\mathrm{SL}(2,k)}

\newcommand{\PGLk}{\mathrm{PGL}(2, k)}
\newcommand{\PGL}{\mathrm{PGL}}
  
\newcommand{\inv}{^{-1}}

\newcounter{note}% a new counter for use in margin notes
 
\newcommand{\note}[1]{% a simple margin note
	\refstepcounter{note}% step counter
	\mbox{\textsuperscript{\thenote}}% the number (superscript) in text
	\marginpar{\footnotesize \mbox{\textsuperscript{\thenote}}{\color{red}\sf  {\bf To do: }}#1}% the note
}

\DeclareMathOperator{\supp}{Supp}
\DeclareMathOperator{\hyb}{hyb}
\DeclareMathOperator{\Gal}{Gal}

\DeclareMathOperator{\conv}{Conv}

\DeclareMathOperator{\tr}{tr}
\DeclareMathOperator{\id}{id}
\DeclareMathOperator{\red}{red}

\DeclareMathOperator{\length}{length}

\DeclareMathOperator{\ord}{ord}
\DeclareMathOperator{\Lm}{Lim}
\DeclareMathOperator{\sph}{sph}
\DeclareMathOperator{\diam}{diam}

\newcommand{\diamant}{\medskip \begin{center}$\diamond$\end{center}\medskip}

\theoremstyle{plain}
\newtheorem{thm}{Theorem}[section]
\newtheorem{prop}[thm]{Proposition}
\newtheorem*{cor*}{Corollary}
\newtheorem{cor}[thm]{Corollary}
\newtheorem{lem}[thm]{Lemma}

\newtheorem*{conj*}{Conjecture}
\newtheorem{mthm}{Theorem}

\theoremstyle{definition}
 
\newtheorem*{ex*}{Example}
\newtheorem{defn}[thm]{Definition}

\theoremstyle{remark}
\newtheorem{rem}[thm]{Remark}

\numberwithin{equation}{section}

\begin{document}

\begin{abstract}
We study the asymptotic behavior of the  Lyapunov exponent in a meromorphic family of 
random products of matrices in  $\mathrm{SL}(2, \mathbb C)$, as the parameter converges to  a pole. We show that the 
blow-up of the Lyapunov exponent is governed by 
 a quantity which can be interpreted as the non-Archimedean Lyapunov exponent of the family. We also describe the limit 
 of the corresponding family of stationary measures on $\mathbb{P}^1(\mathbb{C})$.
\end{abstract}

 \maketitle

\tableofcontents

\section*{Introduction}

Let $G$ be a finitely generated group, endowed with a probability measure $\sm$, satisfying 
the following two conditions:
 \begin{itemize}
 \item[(A1)]  $\supp(\sm)$ generates 
 $G$;
  \item[(A2)] $\displaystyle \int \length(g) d\sm(g)<\infty$. 
 \end{itemize}
In a few occasions we shall also require the following stronger moment condition:
 \begin{itemize}
 \item[(A2${}^+$)] there exists $\delta >0$ such that $\displaystyle \int (\length(g))^{1+\delta} d\sm(g)<\infty$. 
 \end{itemize}
In  (A2) and (A2${}^+$), $\length(\cdot)$ denotes the word-length relative to some fixed, unspecified, finite symmetric set of generators of $G$. It 
depends of course of the choice of generators but the moment conditions do not. 

\medskip

 If $\rho:G \to \SLC$ is any representation, 
   the random walk on $G$ induced by $\sm$ gives rise  through $\rho$ 
   to  a random product of matrices in $\SLC$. If $\norm{\cdot}$ denotes any matrix norm 
    on $\SLC$,  then under the moment condition (A2) we can define the \emph{Lyapunov exponent} 
  $\chi = \chi(G, \sm, \rho)$ by the formula
  \begin{equation}\tag{1}\label{def:lyapunov exponent intro} 
  \chi  = \lim_{n\to\infty} \unsur{n} \int \log\norm{\rho(g_n\cdots g_1)} d\sm(g_1) \cdots d\sm(g_n) =  
\lim_{n\to\infty} \unsur{n} \int \log\norm{\rho(g)} d\sm^n(g),
\end{equation}
where $\sm^n$ is the image of $\sm^{\otimes n}$ under the $n$-fold product map
$(g_1, \ldots , g_n)\mapsto g_n\cdots g_1$. Observe that the limit 
exists because if we choose the matrix norm to be submultiplicative then  the sequence of integrals is subadditive. 
The Lyapunov exponent  is  the most basic 
dynamical invariant associated to the random product of matrices, 
and its properties have been the object of intense research 
since the seminal work of Furstenberg \cite{furstenberg} in the 1960's. 

It is quite customary that $(\rho, \sm)$ depends on certain parameters, in which case  
 the dependence of $\chi$ as a function 
of the parameters becomes an interesting problem.  
One famous instance of this problem, motivated by statistical physics, is the study of discrete Schrödinger operators, which involves 
random products of matrices of the form 
$$\begin{pmatrix} E-v & -1 \\ 1 & 0\end{pmatrix},$$ where $v$ is a real random variable and $E$ (the energy) is a  real or complex parameter 
(see e.g. \cite{bougerol lacroix}). 

\medskip

We are interested in the situation where the representation depends holomorphically on a complex parameter $t$, that is, we consider  
 a  family of representations $(\rho_t)$ such that for every $g\in G$, $t\mapsto \rho_t(g)$ is holomorphic. Then the Lyapunov exponent 
  defines a function $\chi(t) = \chi(G, \sm, \rho_t)$ on the parameter space.   
Recall that a representation $\rho: G \to \SLC$ is said    {\em non-elementary}\footnote{Another common terminology is ``strongly irreducible and proximal".} if there exist
two elements $g_1, g_2 \in G$ such that $\rho(g_1)$ and $\rho(g_2)$ are both hyperbolic
 (i.e. their eigenvalues have modulus $\neq 1$)
 and have no common eigenvectors. A celebrated result due to    Furstenberg \cite{furstenberg}  (see also \cite{furstenberg kifer})
asserts that if $t_0$ is such that 
 $\rho_{t_0}$ is non-elementary, then  under the assumptions (A1-2), 
 $\chi(t)$ is positive and continuous at $t_0$.   H\"older continuity can also be derived
  under stronger moment conditions (see \cite{lepage holder}). 
 And it was recently proved by Bocker and Viana \cite{bocker viana} that if $\sm$ is finitely supported then 
$\chi$ is continuous at  elementary representations as well. 

It is a classical observation that $t\mapsto \chi(t)$ is  subharmonic. In \cite{kleinbif, bers1} Deroin and the first named author have 
studied the complex analytic 
properties of the Lyapunov exponent function in relation with the classical bifurcation/stability theory of Kleinian groups  
 (designed by   Bers, Maskit, Sullivan, etc.) and established that the harmonicity of $\chi$ over some domain is equivalent to the structural stability 
 of the corresponding family of representations.

\medskip

Our purpose in this paper is to study the asymptotic properties of $\chi(t)$  
 in non-compact family of representations with 
``algebraic behavior'' at infinity.
To be specific, 
we consider a holomorphic family $(\rho_t)_{t\in \dd^*}$
of representations of $G$ into $\SLC$, parameterized by the punctured unit disk, and
such that  $t\mapsto  \rho_t(g)$ extends meromorphically through the origin for every $g$. 
An obvious but crucial observation is that this data 
 is equivalent to that of a single representation 
with values in $\SL(2, \mm)$ where  $\mm$ is 
the ring of holomorphic functions on the punctured unit disk with meromorphic extension through  the origin.

We will show that 
the behavior of $\chi(t)$ at $t\to 0$ 
is controlled by a quantity which can be interpreted as a non-Archimedean Lyapunov exponent associated to the family $(\rho_t)_{t\in \dd^*}$. 
To make sense of this statement, observe first that $\mm$ may be viewed as 
a subring of the field of formal Laurent series\footnote{Beware that  in some of the  
references cited in our bibliography,  $\Laurent$ denotes the  completion of the algebraic closure of $\cdt$.}
 $\Laurent:=\cdt$, which is a complete metrized field
 when endowed with  the $t$-adic norm $\abs{f}_\na = \exp(-\ord_{t=0}(f))$. 
 A representation $\rho : G \to \SL(2,\mm)$ thus canonically yields a representation $\rho_\na : G \to \SL(2,\Laurent)$
and exactly as in  \eqref{def:lyapunov exponent intro} we   define  $\chi_\na$ to 
be the Lyapunov exponent  of the representation
$\rho_\na$, where   the matrix norm $\norm{\cdot}$ is  now associated to the $t$-adic 
absolute value on $\Laurent$. 

We are now in position to state our main result. 

  \begin{mthm}\label{Thm:lyapunov na}
  Let $(G, \sm)$ be a finitely generated group endowed with a measure satisfying (A1) and (A2${}^+$), and 
 let $\rho: G\to \SL(2, \mm)$ be any representation. 
  Then 
 \begin{equation} \tag{2}\label{eq:lyapunov na}
   \unsur{{\log\abs{t}\inv}}\chi(t)    \longrightarrow \chi_\na \text{ as } t\cv 0.
  \end{equation}
    \end{mthm}

To illustrate the result in a simple case, consider a random product of  Schrödinger matrices of the form 
\begin{equation}\tag{3}\label{eq:schrodinger matrix}
\begin{pmatrix}  {\unsur{t}(E-v)}  & -1 \\ 1 & 0\end{pmatrix},
\end{equation} 
 where $v$ is a bounded random variable, $E$ is fixed, and $t\to 0$. 
Then it is easy to show in this case that $\chi(t)\sim \log\abs{t}\inv$ as $t\to 0$. One reason for this ease is that the pole structure of 
the $n$-fold product of such matrices is explicit and easy to describe: such a product will be of the form 
$\unsur{t^n}\lrpar{\begin{smallmatrix} O(1)  & O(t) \\ O(t) & O(t^2) \end{smallmatrix}}$. 
Avron, Craig and Simon \cite{avron craig simon} gave in
  this situation  a  refined asymptotics of $\chi(t)$ at the  order $o(1)$ (and a conjectural asymptotics at the order $O(t^{2})$). 
  
For a general random product of matrices with meromorphic coefficients, the poles can add up or cancel in a rather subtle way, and the 
non-Archimedean formalism  allows to deal efficiently with this algebra. 
The idea of using non-Archimedean representations to describe the degenerations 
of $\SLC$ representations is now classical and 
 was pioneered by Culler and Shalen \cite{culler shalen}. One main input of the present work is
  the incorporation of this technique  into the theory of random matrix products. 
  
  \medskip

 Let us explain the strategy of the proof of Theorem \ref{Thm:lyapunov na}. 
 Since $\Laurent$ is a metrized field it makes sense to talk about hyperbolic elements in $\SL(2,\Laurent)$ so we can define
  natural notions of elementary and non-elementary subgroups. 
We refer to  \S\ref{sec:subgroups} for a thorough discussion of these concepts. 
        
The proof of Theorem \ref{Thm:lyapunov na} splits into two quite different parts according to the elementary or non-elementary nature of  $\rho_\na$. 
The easier case is when $\rho_\na$ is elementary: then 
either the family is holomorphic at the origin (after conjugation by a suitable meromorphic family of 
M\"obius transformations and possibly  taking a branched 2-cover of the base)
or $\rho_t$ is elementary for all $t\in \D^*$. In the former case, $\chi_\na = 0$ and   it follows from   Furstenberg's theory 
that $\chi(t)  = O(1)$ so we are done. 
In the latter case, up to meromorphic conjugacy, the image of 
$\rho_t$ in $\SLC$
is either upper triangular  
or lies in an index $2$ extension of the diagonal subgroup. 
The result then follows from a careful application of the law of large numbers.
The details of the arguments are explained in \S \ref{sec:elem}. 
Note that this is the only place where we use the stronger integrability condition (A2${}^+$). 

\medskip

Let us now assume   that $\rho_\na$ is non-elementary. 
A first observation is that   $\rho_t$ is then   non-elementary for small enough  $t$ 
 (see Lemma~\ref{lem:specialization NE}; by rescaling  we may  assume that this holds for $\abs{t}<1$) 
so   we may apply Furstenberg's theory to analyze the Lyapunov exponent $\chi(t)$. The main step of the proof of the continuity of 
$\chi$ in the classical setting is the study of the Markov chain on $\pu$ induced by the image measure $\rho_*\sm$ in $\SL(2,\C)$. More specifically, 
the measure $\mu_t = (\rho_t)_* \sm$ acts by convolution on the set of probability measures on the Riemann sphere by  
$\nu \mapsto \mu_t \ast \nu$, where $\mu_t \ast \nu = \int \gamma_*\nu \; d\mu_t(\gamma)$.
A  \emph{stationary measure} is by definition a fixed point of this action. 
 A fundamental result is that when $\rho_t$ is non-elementary,
 there is a \emph{unique} stationary probability measure $\nu_t$. Furthermore,   
  the Lyapunov exponent $\chi(t)$ is positive and 
can be expressed by an explicit formula involving $\nu_t$:
\begin{equation}\label{eq:fur}\tag{4}
 \chi(t) = \int \log \frac{\| \gamma \cdot v\|}{\|v \|} \, d\mu_t(\gamma)\, d\nu_t(v)~. 
\end{equation}
The continuity of $\chi$ at non-elementary representations immediately follows: since $\nu_t$ is unique it varies continuously with $t$, and
\eqref{eq:fur}  implies the result. 

\medskip

The first step of the proof consists in extending these results to random products in $\SL(2, \Laurent)$. 
We actually work over an arbitrary complete metrized field $k$ and show 
  in \S\ref{sec:furstenberg} how to generalize the above results 
 to any non-elementary representation $\rho: G \to \SLk$. Of particular interest to
 us is the fact that the representation $\rho_\na$ admits a unique stationary measure $\nu_\na$ which lives on the Berkovich analytification 
 $\PP^{1,\an}_{\Laurent}$ of the projective line and for which a non-Archimedean analog of~\eqref{eq:fur} holds. 
 In particular we obtain the positivity of the non-Archimedean Lyapunov exponent. 
 
Let us point out that this positivity  also follows
 from the recent work of Maher and Tiozzo~\cite{maher tiozzo} on random walks on   groups of isometries
of non-proper Gromov hyperbolic spaces. Maher and Tiozzo also discuss   stationary measures, however, they  work in a 
 compactification which is a priori hard to relate to the Berkovich space. 

\medskip
 
From this point,  two different paths lead to  the main theorem. 
Both of them deal with the asymptotic properties of the stationary measure $\nu_t$ and can be seen as ways to imitate
 the Furstenberg argument for the continuity of the Lyapunov exponent. 

\medskip

The first method belongs to complex geometry and  is described in  \S\ref{sec:non elem}. It 
relies on a correspondence between certain finite subsets of $\PP^{1,\an}_{\Laurent}$ and  {\em  models} of 
$\D \times \PP^1_\C $. Here by   model we mean a complex 
surface $Y$   endowed with a birational
 map $\pi: Y \to \D \times  \PP^1_\C$ which is a biholomorphism over $\D^*\times \PP^1_\C$.

For $t\neq 0$ we denote by $\nu_{Y_t}$ the pull-back to $\nu_t$ on $Y$, which should be understood as ``the measure $\nu_t$ viewed on the model $Y$''. 
We obtain the following result.
 
 \begin{mthm}\label{Thm:atomic}
 Let $(G, \sm)$ be a finitely generated group endowed with a measure satisfying (A1) and let
 $\rho: G\to \SL(2, \mm)$ be a non-elementary representation. 

Then for every model $Y\cv \dd\times \pu$, the canonical family of stationary probability measures 
 $\nu_{Y_t}$ converges as $t\to 0$ to a purely atomic  measure $\nu_{Y}$. 
 
 Furthermore if  $\nu_\na$ denotes the unique stationary probability measure on $\PP^{1,\an}_{\Laurent}$, then
  $\nu_Y = (\mathrm{res}_Y)_*\nu_\na$ is the residual measure of $\nu_\na$ on $Y$.
 \end{mthm} 
 
 Here the residue map $\mathrm{res}_Y$ is a canonical anti-continuous map from $\PP^{1,\an}_{\Laurent}$ to the $\C$-scheme
 $\pi^{-1}(\{ 0 \} \times \PP^1_\C)$. In particular the push-forward of a \emph{non-atomic} measure on $\PP^{1,\an}_{\Laurent}$ is 
 an  \emph{atomic} measure on $Y$ . We refer to \S \ref{subs:reduction} for a detailed discussion on this map.  
 The asymptotics~\eqref{eq:lyapunov na} of the Lyapunov exponent in Theorem \ref{Thm:lyapunov na}
  then follows from an   analysis of \eqref{eq:fur} as $t\cv 0$ in a carefully chosen family of   models.
  
 \medskip
 
The second  approach to Theorem~\ref{Thm:lyapunov na} relies on 
 the notion of {\em hybrid (Berkovich) space}. This is a topological space which allows to give a precise meaning to the 
 intuitive idea that $\nu_t$ ``converges'' to $\nu_\na$ as $t\to 0$, and derive the asymptotics~\eqref{eq:lyapunov na}
directly from this weak convergence.
 
 This space  was first constructed by Berkovich in~\cite{berko}. It has been recently realized by Boucksom and Jonsson 
 in~\cite{boucksom jonsson} and 
 by the second named author in~\cite{degeneration} that it is well-adapted to 
the description of the limiting behavior of families of measures such as the $(\nu_t)_{t\in \dd^*}$. 
Concretely, 
 $\phyb$ is a compact topological space endowed with  a continuous surjective 
 map\footnote{The choice of the value  $1/e$ for the radius   is convenient, of course any other would do.} 
 $p_{\hyb}: \phyb \to \overline \dd_{1/e}$ 
 such that $p_{\hyb}^{-1}(0)$ can be identified with
the non-Archimedean analytic space $\PP^{1,\an}_{\Laurent}$ while 
 $p_{\hyb}$ is a trivial topological fibration over  $\overline \dd_{1/e}^*$ with  $\PP^1$ fibers. 
More precisely there exists a canonical homeomorphism 
$\psi\colon \overline\dd_{1/e}^* \times \PP^1 \rightarrow  p_{\hyb}^{-1}(\overline\dd_{1/e}^*) $
such that $p_{\hyb} \circ \psi$ is the projection onto the first factor. 
Likewise we denote by  $\psi_\na$  the canonical identification $\PP^{1,\an}_{\Laurent} \to p_{\hyb}\inv(\set{0})$.

 A key point in the construction of $\phyb$ is that its topology is designed so that 
  for every $g\in G$ the function
$$
(t,v) \longmapsto \frac1{\log|t|^{-1}} \, \log \frac{\|\rho_t(g)\cdot v\|}{\| v \|}
$$
extends \emph{continuously} to the hybrid space for $t=0$. Theorem~\ref{Thm:lyapunov na} hence follows immediately from: 
  \begin{mthm}\label{Thm:hybrid}
 Let $(G, \sm)$ be a finitely generated group endowed with a measure satisfying (A1)
  and 
 $\rho: G\to \SL(2, \mm)$ be a non-elementary representation. 
 
Then in $\phyb$, we have that 
$ (\psi_{t})_*(\nu_t) \longrightarrow   (\psi_\na)_*\nu_{\na}$ as $t\cv 0$ in  the weak topology of measures, where
$\nu_t$ is the  unique stationary probability measure under $\mu_t$, and $\psi_t(\cdot) =\psi( t, \cdot)$.
 \end{mthm} 

We discuss the construction of the hybrid space and prove Theorem~\ref{Thm:hybrid} in \S\ref{sec:hybrid}.

\diamant

As for  \cite{kleinbif, bers1}, this work was prompted by analogous results in the context of iteration of rational mappings, in 
accordance  with the  celebrated \emph{Sullivan dictionary}. Consider a holomorphic family of rational maps $(f_t)_{t\in \D^*}$
of degree $d\ge2$ that extends meromorphically through $0$, and denote by $\mu_{f_t}$ their measure of maximal entropy. 
In this context, the analogue of \ref{Thm:lyapunov na} is a formula for the blow-up of the Lyapunov exponent of $\mu_{f_t}$
which follows from the work of DeMarco~\cite{demarco03,demarco16} (see also \cite{degeneration} for generalizations to higher dimension),  
and Theorem \ref{Thm:atomic} was proven by DeMarco and Faber \cite{demarco faber} . 
It is particularly interesting to note that proving the convergence of the measures $\mu_{f_t}$ relies on pluripotential theory and 
the interpretation of $\mu_{f_t}$ as the Monge-Amp\`ere measure of a suitable metrization on an ample line bundle, 
whereas in our case it follows from the uniqueness of the stationary measure.

 \medskip
 
Our work raises several natural open questions. 
\begin{enumerate}

\item 
Is it possible to estimate the  error term $\chi(t)  - \chi_\na \log\abs{t}^{-1}$? The answer is easy  when $\chi(t)$ is harmonic 
in a punctured neighborhood of the  origin, in which case one obtains an expansion of the form 
$$\chi(t) =  \chi_\na \log\abs{t}^{-1} + \cst +o(1)$$ (this situation happens e.g. in \eqref{eq:schrodinger matrix}). 
In the general case, continuity holds under appropriate moment assumptions if $\rho_\na$ is elementary 
(see \S\ref{subs:continuity elementary}).  However in the 
nonelementary case  our method seems to produce  errors of magnitude
 $\e\log\abs{t}^{-1}$ (see \S \ref{subs:asymptotics}) so new ideas have to be developed.
 
 \item Can    our results be extended to higher dimensions? 
Random matrix products in arbitrary dimension over an Archimedean or a local field are well-understood. 
Some of our arguments should carry over to the study of the extremal Lyapunov exponents, even if we use   
 the hyperbolic structure of $\PP^{1, \an}_{\Laurent}$ at some places. Note however that even 
   the Oseledets theorem does not  seem to have received much attention over arbitrary metrized fields.
  \end{enumerate}

\diamant

The plan of the paper is as follows. In section \ref{sec:prelim} we recall some basics on  
 Berkovich theory. In section \ref{sec:subgroups} we classify subgroups of $\PGL(2,k)$ 
 for an arbitrary complete metrized field $k$. In particular we define the 
  notion of non-elementary subgroup and classify elementary ones.  
Part of this material follows from the classical theory of group acting on trees. In section \ref{sec:furstenberg} we develop the  
non-Archimedean Furstenberg theory. The 
complex geometric proof  to  Theorem \ref{Thm:lyapunov na} and 
is given in sections \ref{sec:non elem} (in the non-elementary case, including \ref{Thm:atomic}) 
and \ref{sec:elem} (for the elementary case). 
The hybrid formalism and Theorem \ref{Thm:hybrid} are  explained in
\S \ref{sec:hybrid}.

\subsection*{Acknowledgement}
We are grateful to Bertrand Deroin for many useful conversations.

\section{The Berkovich projective line}\label{sec:prelim}

In this section, we collect some basic facts on the Berkovich analytification of the projective line  
 over a complete non-trivially metrized  field $(k, |\cdot|)$. Observe that we do not assume $k$ to be algebraically closed (since we apply these results to $k = \cdt $ later on)
 which leads to a few subtleties. 
  The reader is referred to \cite{Ber,temkin} for a general discussion on Berkovich spaces, and to \cite{baker rumely,jonsson} for 
  a detailed description of the Berkovich projective line. 
  
\subsection{Analytification of the projective line}
We denote by $\PP^1_k$ the projective line over a field $k$, viewed as an algebraic variety, endowed with its Zariski topology, 
and by $\PP^1(k)$ its set of $k$-points which is in bijection with $k \cup \{ \infty\}$. 
When $k = \C$, we often simply denote by $\pu = \pu(\C)$ the Riemann sphere, that is, the complex projective line with its 
 usual structure of compact complex manifold\footnote{Note that formally $\pu$ can be viewed as
  the analytification $\PP^{1,\an}_\C$ of the variety $\PP^1_\C$.}.
 
In the remainder of this section, we suppose that  $(k,|\cdot|)$ is a complete metrized {\em non-Archimedean} field.  
We also assume that the norm on $k$ is non-trivial, so in particular $k$ is infinite. 
We denote by $\PP^{1,\an}_k$ the Berkovich analytification of $\PP^1_k$ which is   a compact topological 
 space endowed with a structural sheaf of analytic functions. Only its topological structure will be used in this paper,
 and we refer the interested reader to~\cite{berko} for the description of the structural sheaf.
 
 The Berkovich space  $\PP^{1,\an}_k$ is     defined as follows.  The analytification of the affine line $\A^{1,\an}_k$ is  the space of all multiplicative semi-norms on $k[Z]$ whose restriction to $k$ coincides with $|\cdot|$,  
endowed with the topology of   pointwise convergence. 
 
Given a point $x\in \A^{1,\an}_k$ and a polynomial $P\in k[Z]$, the value of the semi-norm defined by $x$ on $P$ is usually denoted by $\mathopen{|}P\mathclose{|}_x \in \R_+$. It is also customary to denote it by   $|P(x)|$, the reason for this notation
 should be   clear from the classification of semi-norms below.  The Gau{\ss} norm $\sum_i a_i Z^i \mapsto \max |a_i|$ defines a point denoted by $x_{\mathrm{g}}$, and referred 
to as the Gau{\ss} point. 

The Berkovich projective line can be  defined as a topological space to be 
 the one-point compactification of $\A^{1,\an}_k$ so that we write  $\PP^{1,\an}_k = \A^{1,\an}_k \cup \{\infty\}$.
More formally it is obtained by gluing  two copies of $\A^{1,\an}_k$  in a standard way
 using the transition map $z\mapsto z^{-1}$ on the punctured affine line $(\A^1)^{*,\an}_k$.

A \emph{rigid} point in $\PP^{1,\an}_k$ is a point defined by a multiplicative semi-norm having a non-trivial kernel. For any point $z$ lying in a finite extension of $k$, the semi-norm $\abs{\cdot}_z$ defined by 
 $P \mapsto |P(z)|$ is a rigid point in $\A^{1,\an}_k$. The induced map
  yields a canonical bijection between closed points of the $k$-scheme $\PP^1_k$ and 
rigid points in $\PP^{1,\an}_k$.  In particular $\pp^1(k)$ naturally embeds as a set of rigid points in $\PP^{1,\an}_k$, and in the following we simply view $\pp^1(k)$ as a subset of $\PP^{1,\an}_k$.

The Berkovich projective line  $\PP^{1,\an}_k$ is an $\R$-tree in the sense that it is uniquely pathwise connected, 
see~\cite[\S 2]{jonsson} for  precise definitions. In particular for any pair of points $(x,y)\in  \PP^{1,\an}_k$ there is a well-defined
segment $[x,y]$.  Recall that the convex hull of a subset $F$ in an $\R$-tree is the smallest 
connected set $\conv (F)$ which contains $F$, that is
  the union of all segments $[x,y]$ with $x, y \in F$. 
Any point in  $\PP^{1,\an}_k$ admit a well defined projection  to a closed convex subset. 

 In this paper, by {\em measure} on  $\PP^{1,\an}_k$ we mean a 
Radon measure, that 
is a Borel measure which is internally regular, or equivalently a bounded 
linear functional on the vector space of continuous functions on $\PP^{1,\an}_k$
endowed with the sup norm.

Using  the  tree structure, one can show the following result
(see~\cite[Lemma 7.15]{valtree}).

\begin{lem}\label{lem:support}
The support of any   measure in $\PP^{1,\an}_k$ is compact and metrizable.
\end{lem}

\subsection{Balls in the projective line and semi-norms}\label{subsec:balls}
We still assume that the norm on $k$ is non-Archimedean and non-trivial. 
Closed and open balls in $k$ (of radius $R \in \R_+$)
 \[\overline{B}(z_0,R)= \{ z\in k, \, | z -z_0| \le R\} \text{ and } 
 {B}(z_0,R)= \{ z\in k, \, | z -z_0| < R\},
 \] 
  are defined as usual. By definition, a ball in $\PP^1(k)$ is either a ball in $k$ of the complement of a ball in $k$. 
  
Any closed (or open) ball $B\subsetneq\PP^1(k)$ 
determines a point $x_B\in \PP^{1,\an}_k$. When $B$ or its complement is a singleton $\{ z\}$,
 this point $x_B $ is the rigid point
attached to $z$. 
When $B$ is a (open or closed) ball of finite radius in $k$, 
we let $x_B$ be the point in $\A^{1,\an}_k$ corresponding to the semi-norm $|P(x_B)| := \sup_B |P|$.  
Otherwise the complement of $B$ is a ball of finite radius in $k$, and we set  $|P(x_B)| := \sup_{k\setminus B} |P|$.
Observe that a  closed ball $x_B$ is rigid iff  its  diameter is zero, and that the Gau{\ss} point is equal to 
$x_{\overline{B}(0,1)}$.

\begin{rem}
When $|k^*|$ is dense in $\R_+$, we have $x_{\overline{B}(z_0,R)}=x_{B(z_0,R)}$ for all $z_0\in k$ and $R\in \R_+$. 
Otherwise $k$ is discretely valued,  $|k^*| = r^\Z$ for some $r>1$, and we have $x_{B(z_0,r^n)}= x_{\overline{B}(z_0,r^{n-1})}$. 
\end{rem}

\subsection{The spherical metric}
Let $(k,|\cdot|)$ be any non-Archimedean complete metrized field. 
We can endow
$\PP^1(k)$ with the spherical metric:
\begin{equation}\label{eq:dsph}
d_{\sph} ([z_0:z_1], [w_0:w_1]) = \frac{|z_0w_1 - z_1w_0|}{\max\{ |z_0|, |z_1|\} \, \max\{ |w_0|, |w_1|\}}~.
\end{equation}
and its spherical diameter is equal to $1$.  
For any $z\in \PP^1(k)$ and $r\le 1$ we define closed and open {\em spherical} balls
\[
\overline{B}^{\sph}(z,r) = \{ d_{\sph} (\cdot , z) \le r\} \text{ and } 
B^{\sph}(z,r) = \{ d_{\sph} (\cdot , z) < r\}~. 
\]
Observe that for all $r\ge1$, $\overline{B}^{\sph}(z,r) = \PP^1(k)$. 
A spherical ball is either $\PP^1(k)$ or a ball in $\PP^1(k)$ in the sense of the previous section. 
Conversely a ball in $\PP^1(k)$ is either a spherical ball or the complement of a spherical ball.

\subsection{The hyperbolic space ($k$ algebraically closed)} \label{subs:hyperbolic alg closed}
Let $(k,|\cdot|)$ be any \emph{algebraically closed} non-Archimedean complete metrized field, and 
let us describe the geometry of $\PP^{1,\an}_k$ under this assumption. First observe that rigid points are in bijection with balls of zero diameter. 
Following the  Berkovich terminology we say that points corresponding to balls of diameter $\diam(B) \in |k^*|$ (resp. $\diam(B) \notin |k^*|$)
are of type $2$ (resp. of type $3$). 
Rigid points are said to be of type $1$. 

More generally, points in $\PP^{1,\an}_k$ are in bijection with (equivalence classes of) 
decreasing sequences of balls, see~\cite[Theorem 1.2]{baker rumely}. 
When $k$ is spherically complete, that is, every decreasing 
intersection of balls is non-empty, 
then $\PP^{1,\an}_k$ consists  of type $1$, $2$ or $3$ points. 
In general, there may exist a fourth type of points, associated with decreasing sequences of balls 
with empty intersection. 
All these  types of points (whenever non-empty) yield   dense subsets of $\PP^{1,\an}_k$. 

Types of points relate with the tree structure as follows.
Type $1$ and $4$ points are precisely the ones at which the $\R$-tree $\PP^{1,\an}_k$ has only one branch, that is, they are
 endpoints of the tree. Type $2$ points
are branching points (i.e. $\PP^{1,\an}_k$ has at least three branches at these points)
 and  type $3$ points are regular points (i.e. $\PP^{1,\an}_k$ has exactly two branches at these points).

The hyperbolic space $\H_k$ is by definition the complement of the set of rigid points in $\PP^{1,\an}_k$, i.e. $\H_k = \PP^{1,\an}_k \setminus \PP^{1}(k)$. 
It is a proper subtree of 
$\PP^{1,\an}_k \setminus \PP^1(k)$  which contains 
no rigid point  and  is neither open nor closed.

To describe the structure of $\H_k$, for any $r\in \R_+$ introduce  the semi-norm  $x_r\in \A^{1,\an}_k$ defined by
\[
|P(x_r)| = \max \left\{ |a_n| r^n , \, a_n \neq 0 \right\}\]
where $P(Z) = \sum a_n Z^n$. In particular
 $x_r = x_{\bar B(0,r)}$.  
 Let $\H_k^\circ$ be the orbit under $\PGL(2,k)$ of
the ray $\{ x_r, \, r\in \R_+^*\}$. Then $\H_k^\circ$ is a  dense subtree of $\H_k$, and 
 $\H_k \setminus \H_k^\circ$ coincides with the set of type 4 points in $\PP^{1,\an}_k$. 

By \cite[Prop. 2.30]{baker rumely}, one can endow $\H_k^\circ$ with a unique $\PGL(2,k)$-invariant
metric such that 
\[ d_{\H} (x_{r_1}, x_{r_2}) = \log \left|\frac{r_1}{r_2}\right|\]
for any $r_1\ge r_2 >0$.
Observe that 
$$d_{\H} (x_{B_1}, x_{B_2}) = \log \left( \frac{\diam(B_2)}{\diam(B_1)}\right)~,$$
for any closed two balls $B_1 \subset B_2 \subset k$ (the diameter is  relative to the metric induced by the norm $|\cdot|$). 

A proof of the next result can be found in \cite[Prop. 2.29]{baker rumely}.
\begin{lem}
The metric defined above on $\H_k^\circ$ extends  to a distance on $\H_k$, which makes $(\H_k, d_{\H})$   a complete 
metric $\R$-tree upon which $\PGLk$ acts by isometries.
\end{lem}

Recall that by   metric $\R$-tree
 we   mean  that 
  for any pair of distinct points $x,y$ there exists a unique isometric
embedding $\gamma : [0,d_{\H}(x,y)] \to \H_k$ such that $\gamma(0) = x$, $\gamma(1) =y$, 
and $d_{\H}(\gamma(t), \gamma(t') ) = |t -t'|$.

\subsection{Field extensions}\label{subs:field extensions}

Let $K/k$ be any complete field extension. The inclusion $k[Z] \subset K[Z]$ yields by restriction a canonical surjective and continuous map 
$\pi_{K/k}\colon \PP^{1,\an}_K \to \PP^{1,\an}_k$, and the Galois group $\Gal(K/k)$ acts continuously on $\PP^{1,\an}_K$. 

Let  $\bar{k}^a$ be the completion of an algebraic closure of $k$.  Then  $\PP^{1,\an}_{k}$ is homeomorphic to the quotient of 
$\PP^{1,\an}_{\bar{k}^a}$ by $\Gal(\bar{k}^a/k)$, see~\cite[Corollary 1.3.6]{berko}. The group $\Gal(\bar{k}^a/k)$ 
preserves the types of points in  $\PP^{1,\an}_{\bar{k}^a}$, so that we may define the type of a point $x \in  \PP^{1,\an}_k$ as the type of any of its
preimage by $\pi_{\bar{k}^a/k}$ in $\PP^{1,\an}_{\bar{k}^a}$. Note that since the field extension $\bar{k}^a/k$ is not algebraic in general, it may happen 
that some type 1 points in $\PP^{1,\an}_k$ are not rigid (this phenomenon occurs
 when $k$ is the field of Laurent series over any field). 

By \cite[Prop. 2.15]{baker rumely}, the natural action of $\PGL(2, k)$ on $\PP^{1,\an}_{\bar{k}^a}$  
preserves the types of points so the same holds for its action on $\PP^{1,\an}_k$. 

\begin{prop}\label{prop:orbit gauss}
The following assertions are equivalent. 
\begin{enumerate}
\item
The point $x\in  \PP^{1,\an}_k$ belongs to the orbit of the Gau{\ss} point under the action of $\PGL(2,k)$ (in particular it is of type 2). 
\item
There exists $z\in k$ and $r\in |k^*|$ such that $x= x_{\bar{B}(z,r)}$.
\end{enumerate}
\end{prop}

\begin{proof}
Pick $z\in k$ and assume $r\in \abs{k^*}$ that there exists 
$y\in k^*$ with $r= |y|$. The image of $\bar{B}(0,1)$ by the affine map $Z \mapsto yZ + z$ is equal to $\bar{B}(z,r)$
hence (2)$\Rightarrow$(1). Conversely any element in $\PGL(2,k)$ can be decomposed  as a product of affine maps and the inversion $\Phi(Z) := 1/Z$. Thus we conclude that  (1)$\Rightarrow$(2) by observing that 
$\Phi(x_{\bar{B}(z,r)})= x_{\bar{B}(z^{-1},r/|z|^2)}$ if $r \le |z|$, $\Phi(x_{\bar{B}(0,r)})= x_{\bar{B}(0,r^{-1})}$ and $A(x_{\bar{B}(z,r)}) = x_{\bar{B}(az+b, |a|r)}$ if $A(Z) = aZ +b$.
\end{proof}

\begin{prop}\label{prop:projection}
Suppose $x,y, z$ belongs to the orbit of $x_{\mathrm{g}}$ under $\PGL(2,k)$.
Then the projection of $z$ on $[x,y]$ also belongs to the orbit of $x_{\mathrm{g}}$ under $\PGL(2,k)$.
\end{prop}

\begin{proof}
We can normalize the situation so that $x = x_{\mathrm{g}} = x_{\bar B(0,1)}$ and $y = x_{\bar B(0,r)}$ for some $1<r\in |k^*|$. 
Let $z = x_{\bar B(a, r')}$. If $\bar B(a, r')$ is disjoint from $\bar B(0,1)$ or contains it then the projection is the Gauss point and we are done. Otherwise $\bar B(a, r') \subset \bar B(0,1)$ with   $\abs{a}\leq 1$ and $r'<1$. If 
$\bar B(a, r') \subset \bar B(0,r)$ the projection equals $\bar B(0,r)$  and again we are done. The remaining case is when $\abs{a} >r $, in which case the projection is  $\bar B(0,\abs{a})$ and we conclude by Proposition \ref{prop:orbit gauss}.
\end{proof}

 \subsection{The hyperbolic space ($k$ arbitrary)} \label{subs:hyperbolic arbitrary}

The Galois group  $\Gal(\bar{k}^a/k)$ acts on $\bar{k}^a$ by isometries, so   the diameter of balls is $\Gal(\bar{k}^a/k)$-invariant. 
As a consequence the action of  $\Gal(\bar{k}^a/k)$  on $(\H_{\bar{k}^a},d_\H)$ is also isometric. 
Let $\tilde{\H}_k\subset\H_{\bar{k}^a}$ be the set of fixed points
of this action.  
 \begin{lem} 
 The set of fixed points of the action of $\Gal(\bar{k}^a/k)$ on $\PP^{1,\an}_{\bar{k}^a}$ is $\overline{\conv(\PP^1(k))}$. 
 \end{lem}
 
 \begin{proof}
 Denoting by $\mathcal{F}$ this set of fixed points, it is clear that $\mathcal{F}$ contains $\PP^1(k)$. Since the Galois action 
 preserves the tree structure we infer that $\mathcal{F}\supset \mathrm{Conv}(\PP^1(k))$ and since it is weakly continuous we finally deduce  that $\mathcal{F}$ contains  $\overline{\mathrm{Conv}(\PP^1(k))}$. 

Suppose by contradiction that there exists a point $x\in \mathcal{F}$ which does not belong to $\overline{\mathrm{Conv}(\PP^1(k))}$. Since type $2$ points are dense on any ray
in the tree $\PP^{1,\an}_{\bar{k}^a}$, we may suppose $x =x_{\bar{B}(z,r)}$ for some $z\in \bar{k}^a$ and $r \in |k^*|^{\Q}$. It is enough to show that $\bar{B}(z,r)$ contains a point of $k$. Indeed in this case we get that 
 $x \in\mathrm{Conv}(\PP^1(k))$ which contradicts our assumption.

To show this, first note that 
  algebraic points over $k$ are dense in $\bar{k}^a$, hence
we may assume that $z$ is algebraic over $k$. Let $P$ be its minimal polynomial, and suppose
 first that its degree $d$ is prime to the  characteristic of $k$. 
The point $x$ is fixed by $\Gal(\bar{k}^a/k)$ hence so does $\bar{B}(z,r)$, so 
 this ball contains all the roots $z=z_1, \ldots, z_d$ of $P$ (repeated according to their multiplicity if needed).
In particular letting $z^* = \frac1{d} (z_1+ \cdots + z_d)$ we have that   $|z^* - z|   \le r$ and $z^*\in k$ and we are done. 
 
When $d$ is not prime to the   characteristic, we modify this argument as follows. Fix $a\in k^*$ such that 
$   |a| \cdot |z|^{d+1} \ll \abs{P(x)}$ (recall that $\abs{P(x)} = \sup_B\abs{P}$)
and consider the polynomial $\widetilde{P}(X)= a X^{d+1} + P(X)$. By definition of 
$P$ we have $|\tilde{P}(z)| = |a| \cdot |z|^{d+1}$, and on the other hand $|\tilde{P}(x)| \ge |P(x)| $. 
This classically implies the existence of a root of $\widetilde{P}$ in $\bar{B}(z,r)$, so $\widetilde P$ is a polynomial in 
$k[X]$ of degree $d+1$ with a root in $\bar{B}(z,r)$ and  we can apply 
the previous argument to $\widetilde P$. 
 \end{proof}
By the previous lemma we have that  in $\PP^{1,\an}_{\bar{k}^a}$,

  $\tilde{\H}_k:= \overline{\conv(\PP^1(k))}\setminus \PP^1(k)$. In particular $(\tilde{\H}_{k},d_\H)$ is a 
  complete metric $\R$-tree by \S \ref{subs:hyperbolic alg closed}.

Since $\PP^{1,\an}_{k}$ is homeomorphic to the quotient of 
$\PP^{1,\an}_{\bar{k}^a}$ by $\Gal(\bar{k}^a/k)$, the restriction  map $\pi_{\bar{k}^a/k}$ induces a homeomorphism from $\tilde{\H}_k$ onto its image, which we denote by $\H_k$ and call the hyperbolic space over $k$
\footnote{When $k$ is a $p$-adic field, $\H_k$  is \emph{not} the Drinfeld upper half-plane which is equal as a set to $\PP^{1,\an}_k \setminus \PP^1(k)$.
}. We endow it with the metric $d_\H$ making $\pi_{\bar{k}^a/k}\colon (\tilde{\H}_k,d_\H)\to (\H_k,d_\H)$ an isometry. 

The following proposition summarizes the properties of $\H_k$ obtained so far. 

\begin{prop}
The hyperbolic space  $\H_k$  is the closure of the convex hull of $\PP^1(k)$ in $\PP^{1,\an}_{k}$ 
  from which $\PP^1(k)$ is removed, that is $\H_k:= \overline{\conv(\PP^1(k))}\setminus \PP^1(k)\subset \PP^{1,\an}_k$. 
Endowed with the metric $d_\H$, it  is a  complete metric $\R$-tree upon which $\PGLk$ acts by isometries. 
\end{prop}
  
Let $K/k$ be any complete field extension, then there is a canonical $\PGL(2,k)$-equivariant  continuous map $\sigma_{K/k} : \overline{\conv( \PP^1(k))} \to \PP^{1,\an}_K$
such that $\pi_{K/k} \circ \sigma_{K/k} = \id$ and which sends the point $x_{\bar{B}(0,r)}\in\PP^{1,\an}_k$ to the  corresponding point $x_{\bar{B}(0,r)}\in \PP^{1,\an}_K$ for all $r\in \R_+$. 
A detailed discussion of this map can be found in~\cite{poineau}. 
The map $\sigma_{K/k}$ is injective hence induces a homeomorphism from $\overline{\conv( \PP^1(k))}$ in $ \PP^{1,\an}_k$
onto its image which is  the closure of the convex hull of $ \PP^1(k))$ in $ \PP^{1,\an}_K$.

\begin{prop}\label{prop:quadratic}
For any pair  $(x, y)$ of points in  $ \H_k$ lying in the orbit of $x_{\mathrm{g}}$ under $\PGL(2,k)$, 
there exists a quadratic extension $K/k$ 
and  $g\in \PGL(2,K)$ such that $g\cdot x_{\mathrm{g}}\in \PP^{1,\an}_K$ is the middle point of the segment $[\sigma_{K/k}(x),\sigma_{K/k}(y)]$.
\end{prop}

\begin{proof}
Applying a suitable M\"obius transformation, we may assume that
$x= x_{\mathrm{g}} =x_{\bar{B}(0,1)}$ and $y = x_{\bar{B}(0,r)}$ with $r\in |k^*|$.
Fix $z\in k^*$ such that $|z| = r$, and pick any square root $z'$ of $z$.
The middle point of  the segment   $[x,y]:= \{x_{B(0,t)}, \, t \in [1,r]\}$ is the type 2 point $x_{\bar B(0, \sqrt{r})}$
which lies in the orbit of $x_{\mathrm{g}}$ by $\PGL(2,k(z'))$. The assertion is proved with $K=k(z')$.
 \end{proof}

\subsection{Balls and simple domains in $\PP^{1, \an}_k$}

For any  $a\in k$ and any $r\geq 0$, set
\begin{equation}\label{eq:balls}
\bar{B}^\an(a, r) = \set{x\in \A^{1, \an}_k, \ \abs{Z-a}_x \leq r}\text{ or } 
  B^\an(a, r) = \set{x\in \A^{1, \an}_k, \ \abs{Z-a}_x < r}~.
  \end{equation}
 When no confusion can arise, we  drop the ``an'' subscript. 
  
 A closed  ball in the Berkovich projective line is 
a set of the form   $\bar B^\an(a, r)$ or the complement of a set of the form $B^\an(a, r) $ in $\PP^{1, \an}_k$. 
 One defines similarly open balls. Observe that any ball in $\PP^{1, \an}_k$ not containing $\infty$ is of the form  \eqref{eq:balls}.
  A closed (resp. open)  ball is a closed (resp. open) subset of  $\PP^{1, \an}_k$ with one single boundary point.
  When the ball is $\bar B^\an(a, r) $ or 
  $B^\an(a, r)$, or their complements, this boundary point is  $x_{\bar{B}(a,r)}$. 
  
When $k$ is algebraically closed, a closed (resp. open) ball in $\PP^{1, \an}_k$ is the convex hull of a closed (resp. open) ball in $\PP^1(k)$.

Following the terminology of~\cite{baker rumely},  by a {\em simple domain} we mean 
 any open set $U\subset \berk$ whose boundary is a finite set of type 2 points. Open balls
are simple domains, and it follows from~\cite[Thm 4.2.1]{berko}
that simple domains form a basis for the topology of $\berk$.

The next result will play an important role in our approach to Theorem~\ref{Thm:lyapunov na}. 
 \begin{prop}\label{prop:subtree Hk}
 Let  $\nu$ be any probability measure on  $\berk$ having no atom. 

 Then for every $\e>0$ there exists a finite set $S$ of  type 2 points 
 such that every connected component $U$ of
$\berk\setminus S$ satisfies $\nu(U)<\e$. 
\end{prop}

\begin{proof}
  Pick any $\e>0$. 
Since $\nu$ is a Radon measure, any point $x$ is included in a  simple domain $U_x$ such that $\nu(U_x) \le \e$.  
By compactness, we may cover the support of $\nu$ by finitely many of these  domains $U_{x_1}, \ldots, U_{x_n}$. 
Let $S$ be the union of all boundary points of $U_{x_i}$ for $i =1 , \ldots ,n$. Any connected component $U$ of $\berk\setminus S$ 
intersects one of the open sets say $U_{x_1}$. Since $U \cap \partial U_{x_1} = U\cap S = \emptyset$, it follows that $U\subset U_{x_1}$
and $\nu(U) \le \e$ as claimed.
\end{proof}

\section{Subgroups of $\PGLk$}\label{sec:subgroups}

In this section $(k, \abs{\cdot})$ is an arbitrary  non-trivially  valued field 
 that is  complete  and non-Archimedean and by $k^\circ$ its valuation ring.  
 The case of most interest to us is 
 $\Laurent:=\cdt$
  which is a complete metrized field
 when endowed with  the $t$-adic norm $\abs{f}_\na = \exp(-\ord_{t=0}(f))$.

We consider a subgroup $\Gamma\leq \PGLk = \mathrm{Aut}(\pp^1_k)$ 
and study the geometric properties of its action on the projective and Berkovich  spaces. 
Much   of this material is a reformulation in our context of classical results on groups acting on trees
(see e.g. \cite{culler morgan, kapovich,otal}, and also \cite{yang wang} for related material).

Over the complex numbers the corresponding results are well-known (see \cite{beardon}).

\subsection{Basics} As in the complex setting, there is a morphism $\SLk\to\PGLk$, 
defined by associating a M\"obius transformation to a 2-by-2 matrix by the usual formula
 $$ \left( \begin{matrix} a & b \\ c & d \end{matrix}\right)  \longmapsto \lrpar{z\mapsto \frac{az+b}{cz+d}},$$
whose kernel is $\set{\pm \id}$.  

Beware that in general this morphism is not surjective. It is so when the field $k$ is algebraically closed, but   not for instance 
in the case $k = \cdt$. The trouble is that for a general Möbius transformation 
$\frac{az+b}{cz+d}$, the determinant $ad-bc$ need not be a square in $k$. 
Thus, if we denote by $\e$ the generator of the Galois group of the quadratic field extension 
$\C(\!(t^{1/2})\!) / \cdt$ (i.e. $\e \cdot t^{1/2} = - t^{1/2}$), 
we have a surjective morphism from the subset of matrices $M \in \SL(2, \C(\!(t^{1/2})\!))$ 
for which $ \e \cdot M = \pm M$   onto $\PGL(2,\cdt)$  whose kernel is again $\set{\pm \id}$. 
In other words, after a base change we can always lift a meromorphic family of Möbius transformations to 
a family of matrices in $\SL(2, \mm)$.  The same phenomenon happens for triangular matrices over $k$ and the affine group
$\mathrm{Aff}_k$.

 Working with matrices is often more convenient for calculations, and when no confusion can arise, 
 we simply identify  $\gamma\in \SLk$  with the corresponding M\"obius transformation, denoted by
by  $z\mapsto \gamma(z)$.

 \medskip
 
 An element of $\PGLk$ induces  an automorphism  of  $\berk$ preserving  
$\pu(k)$ and  $\hh_k$. Recall that it preserves  the types of points 
and acts by isometries on $(\hh_k,d_\H)$. 
\smallskip

 For $A = \lrpar{\begin{smallmatrix} a & b \\ c & d \end{smallmatrix}} \in \SLk$ we denote by 
 $\norm{A} = \max\lrpar{\abs{a}, \abs{b}, \abs{c}, \abs{d}}$, which  
 in the ultrametric case  is the matrix norm associated to  the sup norm on $ k^2$. 
 
 When $\gamma \in \PGL(2,k)$ as explained above there exists  a quadratic extension $K/k$ 
 and $A \in \SL(2,K)$ inducing $\gamma$ on $\PP^1(k)$, and
we set  $\| \gamma \|  := \| A\|$. 
This is well-defined since $K$ carries a unique complete norm whose restriction to $k$ is $|\cdot|$ and 
 $A$ is defined up to multiplication by $\pm \id$.  Likewise, we define $\abs{\tr(\gamma)} := \abs{\tr(A)}$. 
 
 \begin{prop}
 For all $\gamma$ and $\gamma' $ in  $\SLk$, we have that: 
  \begin{enumerate}[{(i)}]
  \item $\norm{\gamma}\geq 1$,
 \item 
 $\norm{\gamma \gamma'}\leq \norm{\gamma}\cdot\norm{\gamma'}$,
 \item $\norm{\gamma} = \norm{\gamma\inv}$,
 \item if furthermore $\gamma\in \SL(2,k^\circ)$, then $\gamma$ induces an isometry of $(\PP^1(k),d_{\sph})$.
 \end{enumerate}
 \end{prop}
 
The proof is left to the reader (note that 
{\it (i)} follows from the  ultrametric property of the absolute value).

\subsection{Classification of elements in $\PGL(2,k)$}
\begin{prop}\label{prop:claselt}
Let $\gamma\in \PGL(2,k)$, 
 $\gamma\neq  \id$. Then  exactly  one of the following   holds:
\begin{itemize}
\item   $\mathopen{|}\tr(\gamma)| >1$: then $\gamma$ is diagonalizable over $k$ and  has one attracting (resp. repelling) 
fixed point $x_\att \in \PP^1(k)$ (resp.  $x_\rep \in \PP^1(k)$). Furthermore  for every   $x\neq x_\rep$ in $\PP^{1,\an}_k$, 
  the sequence  $\gamma^n\cdot x $ converges to $ x_\att$ when $n\to\infty$.
\item $\mathopen{|}\tr(\gamma)| \leq  1$: then $\gamma$ admits a fixed point in $\hh_k$ and more precisely: 
\begin{itemize}
\item   $\tr^2(\gamma)  =4$: then  $\gamma$ is not diagonalizable, and is conjugate in $\PGL(2,k)$
 to $z\mapsto z+1$; thus it fixes a segment 
 $[x,y] \in \PP^{1,\an}_k$ where $x$ (resp. $y$) is  a type 2  (resp. type 1) point  belonging to 
  the $\PGL(2,k)$ orbit of the Gau\ss\ point. 
\item  $\tr^2(\gamma) \neq 4$: then 
 in some at most quadratic extension $K/k$ the matrix $\gamma$ is diagonalizable;  in addition it is
 conjugate to an element in $\PGL(2,K^\circ)$ and  fixes a type $2$ point in $\H_k$.
\end{itemize}
\end{itemize}
\end{prop}

In accordance with the terminology of group actions on trees, 
when $\mathopen{|}\tr(\gamma)| >1$ we say that $\gamma$ is \emph{hyperbolic}, otherwise it is said \emph{elliptic}. 
When required we can be more precise: 
if $\gamma$ is elliptic and $\gamma\neq \id$, 
 we say that $\gamma$ is \emph{parabolic} when $\tr^2(\gamma)  =4$  and \emph{strictly elliptic} otherwise.

\medskip

One cannot say much more on the action of $\gamma$ on $\PP^1_k$ in the 
elliptic case. It depends heavily on the residue field $\tilde{k}$. When the characteristic of $\tilde{k}$ is $p>0$, then
the closure of the subgroup generated by $\gamma$ is isomorphic to $\Z_p$ and $\gamma^{p^n} \to \id$ when $n\to\infty$.

  \medskip
  
Following standard terminology, we say that $\gamma$ has {\em good reduction} if it fixes the Gau{\ss} point (i.e. belongs to $\PGL(2,k^\circ)$),
and {\em potential good reduction} over $K/k$ if it is conjugate in $\SL(2,K)$ to a map having good reduction. 
With notation as in \S\ref{subs:hyperbolic arbitrary} this is equivalent to saying that $\gamma$ fixes a  type 2 point  $x\in \PP^{1,\an}_k$ such that $\sigma_{K/k}(x)$ lies in the 
$PGL(2,K)$-orbit of the Gau\ss\ point.

\begin{proof} The diagonalizability of $\gamma$ depends on the roots of its characteristic polynomial. 
  If $\mathrm{char}(k)\neq 2$,  this  can be read off the discriminant  $\tr^2(\gamma)-4$. 

Suppose that $\mathopen|\tr(\gamma)\mathclose| >1$. 

Then $\gamma$ 
is diagonalizable in a quadratic extension $K$ of $k$ 
and its eigenvalues have respective norms norm larger and smaller than 1.
 This implies that there exists a global attracting point  $\PP^1(K)$ 
 which in particular attracts all elements of $k$. Hence by completeness this fixed point belongs to $\pu(k)$. Applying the same reasoning  to the inverse, we conclude that  $\gamma$ is diagonalizable over $k$.

Suppose now that $\mathrm{char}(k)\neq 2$ and $\tr^2(\gamma) =4$. Then (up to sign) 1 is an eigenvalue of multiplicity 2, hence
  since $\gamma$ is not the identity  it is conjugate in $\PGLk$ to 
$\lrpar{\begin{smallmatrix} 1& 1 \\ 0& 1\end{smallmatrix}}$  and $\gamma(z)$ is conjugate to a translation.    
If $\mathrm{char}(k)=2$ and 
$\tr(\gamma) = 0$, then the characteristic polynomial is $X^2+1 = (X-1)^2$ and the same discussion   applies. 

Suppose finally that $\mathopen|\tr(\gamma)\mathclose|\leq 1$ and $\tr(\gamma) \neq 2$. Then 
$\gamma$ is diagonalizable over an extension $K$ of $k$ of degree at most 2, and both eigenvalues belong to $K^\circ$. 
If $K=k$, then the existence of the announced fixed point is clear. Otherwise consider the geodesic in $\PP^{1,\an}_K$ joining the two fixed points. The Galois group $K/k$ acts on this geodesic and permutes these two points. Thus it admits a fixed point which is of type $2$ and lies in $\H_k$.  
\end{proof}

Here is a noteworthy consequence of this classification.  

\begin{cor}
For $\gamma\in \PGLk$, $\norm{\gamma^n} \to \infty$ as $n\cv\infty$ if and only if  $\gamma$ is hyperbolic. 
\end{cor}

The next result is an  analogue of the Cartan  KAK decomposition in the non-Archimedean setting. It will play an important role in the following. 

\begin{prop}\label{prop:KAK}
Any element $\gamma\in \SL(2,k)$ can be decomposed as a product 
$\gamma= m\cdot a \cdot n$ with $m,n \in \SL(2,k^\circ)$ and $a = \mathrm{diag} (\lambda, \lambda^{-1})$ with $\lambda\in k$, $|\lambda|\ge 1$.
Furthermore  $\norm{\gamma}  = \norm{a} = \abs{\lambda}$.
\end{prop}

\begin{proof}
Let $x_{\mathrm{g}}$ be the Gau{\ss} point.
Pick an element $m\in \SL(2,k^\circ)$ such that $m^{-1} \gamma\cdot x_{\mathrm{g}}$ belongs to the segment $[x_{\mathrm{g}}, \infty]$. Likewise, 
  choose $n\in \SL(2,k^\circ)$ such that $n^{-1} \gamma^{-1} \cdot x_{\mathrm{g}}$ belongs to the segment $[x_{\mathrm{g}}, 0]$.
Then $\gamma' = m^{-1} \gamma n$  either fixes $x_{\mathrm{g}}$ or maps $x_{\mathrm{g}}$ into $(x_{\mathrm{g}}, \infty)$ and its inverse into $(x_{\mathrm{g}},0)$. 

In the former case $\gamma'$ belongs to $\SL(2,k^\circ)$ hence $\gamma$ too and we can choose $a = \id$, $m = \gamma$, $n=\id$.

In the latter case, $\gamma'$ is hyperbolic with two fixed points $|c^+| >1$ and $|c^-|<1$. We claim that in this case we can 
 conjugate it
by an element in $\SL(2,k^\circ)$ so that it becomes diagonal. Indeed we first conjugate by $\lrpar{ \begin{smallmatrix}
1 & -c^- \\ 0 & 1\end{smallmatrix}}$ (i.e. by
the translation $z\mapsto z-c^-$),  which belongs to $\SL(2,k^\circ)$, 
to send $c^-$ to 0. This maps $c^+$ to $ \tilde c^+ = c^+ - c^-$ which has the same norm. Then  we use the element 
$\lrpar{ \begin{smallmatrix}
1 & 0\\ -1/\tilde c^+ & 1\end{smallmatrix}}\in \SL(2,k^\circ)$ to send $ \tilde c^+$ to $\infty$, and we are done.

To prove the  identity on $\norm{\gamma}$ simply observe that 
$\norm{\gamma}  = \norm{man} \le \norm{a}$ and $ \norm{a} = \norm{m^{-1} \gamma n^{-1}} \le \norm{\gamma}$.
\end{proof}

\subsection{Norms of elements in $\SLk$}
 
\begin{lem}
For $\gamma\in \PGLk$,  one has  the identity $d_{\H_k}( x_{\mathrm{g}}, \gamma\cdot x_{\mathrm{g}}) = \log \norm{\gamma}$.
\end{lem}
 
\begin{proof}
Any element in $\PGL(2,k^\circ)$ has norm $1$ and also fixes the Gau{\ss} point so the formula is clear in this case. 
In the general case  we use the KAK decomposition and write $\gamma = man$. 
Then  $d_{\H_k}( x_{\mathrm{g}}, \gamma\cdot x_{\mathrm{g}})  =  d_{\H_k}( x_{\mathrm{g}}, a \cdot x_{\mathrm{g}}) = \norm{a}=\norm{\gamma}$. \end{proof}

A similar argument shows: 

\begin{lem}\label{lem:lipschitz}
For $\gamma\in \PGLk$ and $x, y\in \PP^1(k)$,  then  $$\norm{\gamma}^{-2} d_{\sph}(x, y)\leq d_{\sph}(\gamma x, \gamma y)\leq 
\norm{\gamma}^{2} d_\mathrm{sph}(x, y)$$ and this bound is  optimal.  
\end{lem}

The following geometric consequence of Proposition \ref{prop:KAK} will be very useful. 

\begin{prop}\label{prop:B att rep}
For every $\gamma\in \PGLk$, there exist two closed balls $B_\att(\gamma)$ and $B_\rep(\gamma)$  in $\PP^{1,\an}_k$ of spherical radius
$\norm{\gamma}\inv$ and such that $\gamma(\PP^{1,\an}_k\setminus B_\rep(\gamma))\subset B_\att(\gamma)$. 
\end{prop}

\begin{proof}
This is a straightforward consequence of Proposition~\ref{prop:KAK}. Indeed, with notation as in 
Proposition \ref{prop:KAK}, the result is obvious for $a$, with $B_\rep(a) = \overline{B}^{an}(0, \abs{\la}\inv)$ and 
$B_\att(a) = \overline{B}^{an}(\infty, \abs{\la}\inv) = \pp^{1,\an}_k\setminus B(0, \abs{\la})$. 

In the general case, writing $\gamma =man$, it is enough to put
 $B_{\rm rep}(\gamma) = n\inv (\overline{B}^{an}(0, \abs{\lambda}\inv))$ and 
  $B_{\rm att}(\gamma) = m  (\overline{B}^{an}(\infty, \abs{\lambda}\inv))$.
\end{proof}

 For  $\gamma  = \lrpar{\begin{smallmatrix} a & b \\ c & d \end{smallmatrix}} \in \SLk$ and $v\in \pp^1(k)$, let 
\begin{equation}
\label{eq:sigma1}
\sigma(\gamma,v) = \log\frac{\norm{\gamma V}}{\norm{V}} =\log \lrpar{\max (\abs{ax+by}, \abs{cx+dy}) }\, ,
\end{equation}
 where $V\in k^2$  is any representative of $v$, and
$(x,y)\in k^2$ is a representative of $v$  with $\max
(\abs{x}, \abs{y})=1$.
The  second equality 
shows that $\sigma(\gamma, \cdot)$ 
extends continuously to $\berk$, and we have the cocycle relation 
\begin{equation}\label{eq:cocycle}
\sigma(\gamma_1 \gamma_2,v) = \sigma(\gamma_1, \gamma_2 \cdot v)  + \sigma(\gamma_2, v),
\end{equation}
for all $\gamma_1, \gamma_2 \in \SL(2,k)$ and for any  $v\in \PP^{1,\an}_k$. 

\begin{lem}\label{lem:expand vector 1}
For any $\gamma\in \SL(2,k)$ we have that 
 $$- \log \|\gamma\| \le \sigma(\gamma,v) \le \log\|\gamma\|,$$ 
 and if furthermore
 $v\notin B_\rep(\gamma)$,  then $\sigma(\gamma,v) = \log\|\gamma\|$.
\end{lem}
\begin{proof}
The first assertion is obvious from \eqref{eq:sigma1}. For the second one, 
observe first  that $\sigma(\gamma,v) = 0$ for all $v$ when $\gamma\in \SL(2,k^\circ)$.  Then, writing 
$ \gamma = m a n$ with  as in Proposition~\ref{prop:KAK} we are reduced 
to the case of    $a = \mathrm{diag}(\lambda, \lambda^{-1})$ and the result follows easily.
\end{proof}
\begin{lem}\label{lem:expand vector 2}
For any \emph{hyperbolic} element $\gamma\in \SL(2,k)$ the   balls $B_\att(\gamma)$ and $B_\rep(\gamma)$ are disjoint 
and we have 
\begin{equation}\label{eq:expand}
\|\gamma\| = \frac{\mathopen|\tr(\gamma)\mathclose|}{\min \{ \delta, 1\}}
\end{equation}
where $\delta = d_{\sph} (B_\att(\gamma), B_\rep(\gamma)) =\sup_{x\in B_\att(\gamma)} \inf_{y\in B_\rep(\gamma)} d_{\sph}(x,y)$.
\end{lem}
\begin{proof}
If $\gamma$ is hyperbolic the disjointness 
of the two balls $B_\att(\gamma)$ and $B_\rep(\gamma)$  follows from their   construction and the ultrametric property implies that
$\delta = d_{\sph} (x,y)$ for any pair $(x,y)\in B_\att(\gamma) \times B_\rep(\gamma)$.

We follow the reasoning of \cite[Lem. 2.1]{kleinbif}. 
Conjugate $\gamma$ by some element in $\SL(2,k^\circ)$ 
 to send the attracting fixed point to $\infty$. This does not affect neither $\norm{\gamma}$ nor $\tr(\gamma)$, and after this 
 conjugacy we have $$\gamma = \begin{pmatrix} a & b \\ 0 & 1/a \end{pmatrix}$$ so that  as a M\"obius transformation
 $\gamma(z) = a^2z+ ab$
  for some $\abs{a}>1$. The repelling fixed point is $ab/(1-a^2)$, and its distance to $\infty$ is equal to 
  $\delta = \min \{ 1, |(1-a^2)/ab|\} = \min \{ 1, {\abs{a}}/{\abs{b}}\}$.
  We then have
  $$
  \| \gamma \| = \max\{ |a|, |b|\} = \frac{|a|}{\min \{ 1, {\abs{a}}/{\abs{b}}\}} = \frac{|\tr(\gamma)|}{\min\{ 1, \delta\}}~,
  $$
  as was to be shown.
\end{proof}

Let us point out a kind of converse to the previous lemma.

\begin{lem}\label{lem:distinct balls}
Let $\gamma \in \PGL(2,k)$ be such that there exist two disjoint balls $B_\mathrm{a}$ and $B_\mathrm{r}$ of radius $<1$ 
such that 
$\gamma(B_\mathrm{r}^c)\subset B_\mathrm{a}$. Then $\gamma$ is hyperbolic with attracting and repelling fixed points respectively contained in 
$B_\mathrm{a}$ and $B_\mathrm{r}$. 
\end{lem}

\begin{proof}  Since the complement of a ball is a ball, the existence of an attracting (resp. repelling) type 1
fixed point in $B_\mathrm{a}$ (resp. $B_\mathrm{r}$) follows from \cite[Thm. 10.69]{baker rumely}. The result follows. 
\end{proof}

\begin{lem}\label{lem:two points}
For any pair of distinct points $z_1, z_2 \in \PP^1(k)$, there exists a constant $C = C(z_1, z_2)>0$ 
such that 
$$
\left| \max \{ \sigma(\gamma, z_1), \sigma(\gamma, z_2)\} - \log \| \gamma \| \right|
\le C$$
for all  $\gamma\in \SL(2,k)$.
\end{lem}
\begin{proof}
Since $ \sigma(\gamma, z) \le \log \| \gamma\|$, we only need to prove the lower bound 
$\max \{ \sigma(\gamma, z_1), \sigma(\gamma, z_2)\} \ge \log \| \gamma \| -C$.
Pick $g\in \SLk$ sending  $z_1$ to $0$ and $z_2$ to $\infty$. We have
 \begin{align*}
 \max \{ \sigma(\gamma, z_1), \sigma(\gamma, z_2)\} &=
 \max \left\{ \sigma(\gamma g^{-1} , 0) + \sigma (g,z_1), \sigma(\gamma g^{-1}, \infty ) + \sigma (g,z_2)\right\}\\
 &\ge \log \| \gamma g^{-1}\| -\log \| g \| \ge \log \| \gamma\| - 2\log \| g \|  ~.
 \end{align*}
 This concludes the proof. 
\end{proof}

\subsection{Elementary and non-elementary  subgroups} \label{subs:elementary}

A subgroup $\Gamma\leq \SLk$ (resp. $\Gamma \leq \PGLk$) is said {\em reducible}
if its action on $\PP^1(k)$ fixes a point, 
  and {\em irreducible} otherwise. 
  It is {\em strongly irreducible} if it does not admit a finite orbit in $\PP^1(k)$.
  
We say that $\Gamma$ has {\em good reduction}  if it takes values in $\SL(2,k^\circ)$, or equivalently,   fixes the Gau\ss~point. It has {\em potential good reduction} if 
there exists a finite field extension $K/k$ such that $\Gamma$ is conjugate in $\SL(2,K)$ to a subgroup of $\SL(2,K^\circ)$.
Finally,  $\Gamma$ is    {\em proximal } if it  contains at least one hyperbolic element. 

\begin{prop}\label{prop:fixed point}
A finitely generated subgroup $\Gamma$ of $\PGLk$ is either proximal or has potential good reduction. 
If moreover $k$ is discretely valued,  then $\Gamma$ is conjugate to a subgroup of $\SL(2,K^\circ)$ in some quadratic extension   $K/k$.
\end{prop}
Observe that the groups of translations is not proximal but has not potential good reduction when the norm on $k$ is non-trivial, so that the assumption that $\Gamma$ is finitely generated is necessary in the previous statement. 

\smallskip

This proposition is essentially a formulation in our language of the well-known fact that a group acting on a tree with only elliptic elements has a global fixed point. We sketch the  proof for convenience. 

The key is the following lemma (see \cite[Lemma 10.4]{kapovich} or \cite[Lemme 40]{otal}).
\begin{lem}\label{lem:semi group}
Any finitely generated semi-group of $\SL(2,k)$ which does not contain any  hyperbolic element
fixes a type 2 point lying in $\H_k$.
 \end{lem}
 
 \begin{proof}
 Let $S$ be a finitely generated semi-group which does not
 contain any  hyperbolic element.   We shall prove by induction on the number of generators  the existence of a type 2 point in $\H_k$ fixed by $S$. 
 
 When $S$ is generated by a single  element thisis a direct consequence of  Proposition~\ref{prop:claselt}. 
If $S$ is generated  by two elements $g$ and $h$, pick $x,y\in \H_k$ two type 2 points fixed by $g$ and $h$ respectively. 
Let $x'$ be the unique point satisfying $[x,g(y)] \cap [x,y] = [x,x']$.  By Proposition \ref{prop:projection} this is a type 2 point. Similarly define $y'$ to be the unique type 2 point satisfying $[y,h(x))] \cap [y,x] = [y,y']$. If the segment $[x',y']$ is degenerate, the segment $[y',h(y')]$ is a fundamental domain for the action of $gh$ which is therefore hyperbolic. 
Otherwise $gh$ fixes pointwise $[x',y']$. This proves the result when $S$ is generated by two elements.  
 
 Now suppose $S$ is generated by $g_1, \ldots , g_l$ with $l\ge 3$, and that the result is known for semi-groups generated by $l-1$ 
 elements. For $i =1,2,3$, let $S_i$ be generated by $\set{  g_1, \ldots , g_l}\setminus \set{g_i}$. By the induction hypothesis 
 $S_i$ admits a type 2 fixed point $x_i \in \H_k$. Then the projection of $x_3$ on $[x_1, x_2]$ is a type 2 point in $\H_k$ 
  fixed by $S$ and we are done. 
  \end{proof}

\begin{proof}[Proof of Proposition \ref{prop:fixed point}] 
It  follows from Lemma \ref{lem:semi group} that if $\Gamma$ is not proximal then it 
 fixes a type 2 point $x_\star$ in $\H_k$. Using the notation of \S\ref{subs:hyperbolic arbitrary} 
 this means that $\sigma_{\bar{k}^a/k}(x_\star)$ lies in the $\PGL(2,\bar{k}^a)$-orbit of the Gau\ss{} point. 
 Since algebraic points over $k$ are dense in $\bar{k}^a$, the ball corresponding to $\sigma_{\bar{k}^a/k}(x_\star)$
  contains a point of $k^a$ 
 so we get  that $\sigma_{\bar{k}^a/k}(x_\star)$   lies in fact in the $\PGL(2,k^a)$-orbit of the Gau\ss{} point. 
 In other words, we can find  a finite field extension $K/k$
 and conjugate $\Gamma$ by a matrix in $\PGL(2,K)$ so that it fixes the Gau\ss{} point. 
 
Assume now that $k$ is discretely valued so that $\H_k$ is a simplicial tree. 
The point $x_\star$ is either in the $\PGL(2,k)$-orbit of the Gau\ss{} point or it belongs to a unique segment $[x_0,x_1]$
of $\H_k$ whose extremities lie in the $\PGL(2,k)$-orbit of the Gau\ss{} point. Any element fixing $x_\star$ either fixes pointwise $[x_0,x_1]$ or acts upon it 
as an involution  switching the two extremities. 
It follows that the middle point of $[x_0,x_1]$ is fixed by $\Gamma$. We conclude using Proposition~\ref{prop:quadratic}.
  \end{proof}

 There is a simple classification of subgroups that are not strongly irreducible, analogous to the Archimedean 
 case. 
 
\begin{prop} \label{prop:strongly irreducible}
Let $\Gamma\leq \PGLk$ be a finitely generated 
subgroup that is   not strongly irreducible. Then one of the following situations occurs:
\begin{enumerate}
\item  $\Gamma$ has potential good reduction;
\item $\Gamma$ is conjugate to a subgroup of the affine group $\{z\mapsto a z + b, \, a \in k^\times, \, b\in k\}$; 
\item $\Gamma$ is conjugate to a subgroup of $\{z\mapsto \lambda z^{\pm 1},\, \lambda \in k\}$.
\end{enumerate}
\end{prop}

\begin{proof}
By assumption there exists a finite $\Gamma$-orbit $x_1, \ldots, x_n$ on $\pp^1(k)$. 
If $n=1$ then $\Gamma$ is conjugate to a subgroup of the affine group. 
If $n=2$, we may assume that $x_1 =0$ and $x_2 =\infty$ and it follows 
 that any element $\Gamma$ is conjugate  to $\lambda /z$ or $\lambda z$ for some $\lambda\in k^*$. 
 
Assume now that $\Gamma$ leaves invariant  a set of $n\ge3$ distinct points $E= \{ x_1, \ldots , x_n \}$ in $\pp^1(k)$. 
The first observation is 
 that for every $\gamma\in \Gamma$, some iterate $\gamma^m$ fixes $E$ pointwise, therefore $\gamma^m = \id$. 
All elements of $\Gamma$ are thus elliptic and the previous proposition shows that $\Gamma$ has potential good reduction. 
\end{proof}

\begin{rem} 
If  $\mathrm{char}(k) = 0$, then by the Selberg lemma (see e.g. \cite{alperin})    the existence of a finite orbit 
of cardinality $n\geq 3$ implies that  $\Gamma$ is finite.
\end{rem}

\begin{prop}\label{prop:non elementary}
Let $\Gamma\leq \PGLk$ be a finitely generated subgroup. If $\Gamma$ is 
 proximal and strongly irreducible then it contains two hyperbolic elements with disjoint sets of fixed points. 
\end{prop}

\begin{proof} 
Since $\Gamma$ is proximal, it contains  a hyperbolic element $g$. Denote by $x_{\att/\rep}$ its fixed points. 
We claim that there exists an element $h \in \Gamma$ such that $h(\{ x_\att, x_\rep\})  \cap\{ x_\att, x_\rep\}= \emptyset$. 
Indeed since $\Gamma$ is strongly irreducible,  $\{ x_\att, x_\rep\}$ is not a $\Gamma$-orbit so that there exists $h\in \Gamma$ satisfying $h(x_\att) \notin \{ x_\att, x_\rep\}$. 
There are 3 possibilities:
\begin{itemize}
\item $h(x_\rep)\notin \{ x_\att, x_\rep\}$; 
\item $h(x_\rep) = x_\att$: then either $hgh$ or $h^2$ sends  $\{ x_\att, x_\rep\}$ to a disjoint pair;
\item $h(x_\rep) = x_\rep$: then there exists $j\in \Gamma$ such that $j(x_\rep)\notin  \{ x_\att, x_\rep\}$ and 
$hg^nj\inv$ is convenient for large $n$ (use Proposition \ref{prop:B att rep}).
\end{itemize}	
In any case there exists $k$ such that $k(\{ x_\att, x_\rep\})\cap \{ x_\att, x_\rep\}=\emptyset$, thus  
  $k^{- 1} \gamma k$ is a hyperbolic element whose fixed points are disjoint from $\{ x_\att, x_\rep\}$.
\end{proof}

Proposition \ref{prop:non elementary} motivates the following definition.

\begin{defn}
A finitely generated subgroup
$\Gamma$ of $\PGL(2,k)$  is     {\em non-elementary} if it is proximal and strongly irreducible. 

 A finitely generated subgroup $\Gamma$ of $\SL(2,k)$       if its image in $\PGL(2,k)$ is non-elementary.
  \end{defn}

Propositions \ref{prop:fixed point} and \ref{prop:strongly irreducible} imply the following characterization of non-elementary subgroups. The details are left to the reader. 

\begin{prop}
Let $\Gamma\leq \PGLk$ be a finitely generated subgroup. The following assertions are equivalent:
\begin{enumerate}
\item $\Gamma$ is non-elementary;
\item $\Gamma$ does not admit a finite orbit on $\PP^{1, \mathrm{an}}_k$;
\item for every $z\in \PP^{1, \mathrm{an}}_k$, $\# \Gamma \cdot z\geq 3$. 
\end{enumerate}
\end{prop}

Let us note for further reference the following variation on Proposition \ref{prop:non elementary}.

\begin{lem}\label{lem:proximal semi group}
Let $\Gamma$ be a non-elementary finitely generated subgroup of $\PGLk$. Then for every   set  
$S$ of generators of $\Gamma$,  
the semi-group generated by $S$ contains two  hyperbolic elements with distinct attracting fixed points.  
 \end{lem}
 
 \begin{proof}
 First note that since $\Gamma$ is finitely generated, there is a finite subset $S'\subset S$ such that $\langle S'\rangle$ contains 
 a finite set of generators of $\Gamma$, hence  $\langle S'\rangle = \Gamma$, so, replacing $S$ by $S'$ we may assume that 
 $S$ is finite. Denote by $G_0$ the semi group generated by $S$.
Since $\langle S\rangle = \Gamma$    the elements of $S$ do not admit a common fixed point, hence 
 by Lemma \ref{lem:semi group} there exists a hyperbolic element $g\in G_0$. Thus, letting $\rho = \norm{g}>1$ we infer that
 for $n\geq 0$, 
 $g^n$ maps $B(x_\rep(g), \rho^n)^c$ into $B(x_\att(g), \rho^n)$
 (all the balls here are in $\PP^{1, an}_k$). Since $S$ has no fixed point there exists $h\in S$ such that 
 $h(x_\att(g))\neq x_\att(g)$. Then for every $n\geq 1$, $hg^n$ belongs to $G_0$ and 
 maps $B(x_\rep(g), \rho^n)^c$ into $B(h(x_\att(g)), C\rho^n)$ for some $C = C(h)$. 
If $h(x_\att(g))\neq x_\rep(g)$, from Lemma \ref{lem:distinct balls} we infer  that  for large $n$
 $hg^n$ is hyperbolic, and its   attracting fixed point is  close to $h(x_\att(g))$, hence distinct from $x_\att(g)$. 
 
 If $h(x_\att(g))= x_\rep(g)$, then $h\inv(x_\rep(g)) =x_\att(g)$, and we consider   $hg^n h$ instead of $hg^n$. Indeed   
  for some $C$ we have that 
 $$ B(x_\att(g), C\inv\rho^n)^c\overset{h}\longrightarrow B(x_\rep(g),  \rho^n)^c\overset{g^n}\longrightarrow B(x_\att(g),  \rho^n)
 \overset{h}\longrightarrow  B(x_\rep(g), C \rho^n),$$
  so again we see that $hg^nh$ is hyperbolic for large $n$, and its attracting fixed point is  distinct from $x_\att(g)$. 
  \end{proof}

\subsection{The limit set}

\begin{thm}\label{thm:struct limit}
Let $\Gamma$ be a finitely generated and  non-elementary subgroup of $\PGLk$. Then  the following sets coincide: 
\begin{itemize}
\item
the closure in $\PP^{1,\an}_k$ of the set of fixed points of all hyperbolic elements of $\Gamma$;
\item
the smallest non-empty $\Gamma$-invariant closed subset of $\PP^{1,\an}_k$;
\item
for any given $x\in \PP^{1,\an}_k$, the set of points $y$ such that there exists a sequence $(g_n)\in \Gamma^\nn$ such that 
$\norm{g_n}\to\infty$ and $g_n \cdot x \to y$.
\end{itemize}
This set is compact, metrizable, and included in $\conv(\PP^1(k)) =\H_k \cup \PP^1(k)$. It is by definition 
 the limit set $\Lm(\Gamma)$ of $\Gamma$. 
\end{thm}
 
\begin{proof}
Denote by $\Lambda_0$ the set of fixed points of all hyperbolic elements, and by $\Lambda$
the smallest closed $\Gamma$-invariant 
subset of $\PP^{1,\an}_k$. If $x$ is fixed by some hyperbolic element $g$, then for every 
$h\in \Gamma$, $h(x)$ is fixed by $hgh\inv$. We infer that $\overline{\Lambda}_0$
 is a closed $\Gamma$-invariant set thus $\Lambda\subset \overline{\Lambda}_0$.

Conversely, pick any   $x\in \Lambda$. By     Proposition \ref{prop:non elementary}
 there exists a 
hyperbolic element $g\in \Gamma$ whose fixed point set $\{ x_\att, x_\rep\}$ is disjoint from $x$. 
Since $g^{\pm n}(x) \to x_{\att/\rep} $ it follows that $\{x_\att, x_\rep\} \subset \Lambda$ so   $\Lambda$  admits 
at least three points.  Therefore, for an  arbitrary hyperbolic element $g'\in\Gamma$ with fixed point set
 $\{ x'_\att, x'_\rep\}$, 
there exists  $y\in \Lambda\setminus \{ x'_\att, x'_\rep\}$. Then $(g')^{\pm n}(y) \to x'_{\att/\rep}$ as $n\to\infty$, 
from which we infer that  $\Lambda_0 \subset \Lambda$. We conclude that 
$\Lambda = \overline{\Lambda}_0  $.

\medskip

Fix now any point $x\in \PP^{1,\an}_k$ and denote by
 $\Lambda_1$ the set of all $y$ for which there exists a 
sequence $(g_n)$ with $\norm{g_n}\to\infty$  
and $g_n\cdot x \to y$.  Observe that $\Lambda_1$ is $\Gamma$-invariant and non-empty since $\Gamma$ contains a hyperbolic element. 
We claim that it is also closed.  Indeed by a theorem of Poineau~\cite[Th\'eor\`eme 5.3]{poineau}, for any $y'$ in the closure of $\Lambda_1$ there exists a sequence $y_n \in \Lambda_1$
such that $y_n \to y'$. For each $n$, pick a sequence with $\norm{g_{m,n}}
\ge m+n$ such that $g_{m,n}\cdot x \to y_n$.
The set $\{ g_{m,n}\cdot x\}$ contains $\{y_n\}$ in its closure hence $y'$ too. Again by Poineau's theorem, there exists a subsequence
$g_{m_j,n_j}\cdot x \to y$. This shows that  $\Lambda_1\supset\Lambda$.

\smallskip

Now suppose $\norm{g_n}\to\infty$ and $g_n\cdot x \to y\in\Lambda_1$. We want to show that $y$ belongs to $\Lambda$.  If $x$ belongs to $\Lambda$ then the closure of $\Gamma\cdot x$ is contained in $\Lambda$
and the result follows. 
So suppose   that  $x$ does not belong to $\Lambda$.

Recall from Proposition \ref{prop:B att rep} that we can associate to every $g\in \PGLk$ two closed balls $B_{\rm att}(g)$ and 
$B_{\rm rep}(g)$   in $\PP^{1,\an}_k$ such that 
 $g (\berk\setminus B_{\rm rep}(g)) \subset B_{\rm att}(g)$. We claim that for 
 large enough $n$,   $B_{\rm att}(g_n)$   
 intersects $\Lambda$. Indeed pick any 2 distinct points in $\Lambda$. Then since $\norm{g_n}\to \infty$, for large $n$ 
   one of these points does not belong to
     $B_{\rm rep}(g_n)$, hence its image under $g_n$  belongs to  
 $B_{\rm att}(g_n)$,  and also to $\Lambda$ by invariance, so we get that 
 $B_{\rm att}(g_n)\cap \Lambda\neq \emptyset$. 
 Similarly,  $B_{\rm rep}(g_n)\cap \Lambda\neq \emptyset$.
 
  In particular we see that for large $n$, $x\notin B_{\rm rep}(g_n)$. Indeed otherwise 
 since the diameter of $B_{\rm rep}(g_n)$ tends to zero we would infer that $x$ belongs to $\Lambda$, which is not the case.
  Thus we conclude that   $g_n\cdot x\in B_{\rm att}(g_n)$ for large $n$, 
  so every neighborhood of $y$ intersects $\Lambda$ and it follows that $\Lambda_1\subset\Lambda$.

\medskip

The limit set is a closed subset of $\PP^{1,\an}_k$ which is compact, hence it is also compact. 
Since $\Gamma$ is 
countable hence $\Lambda_0$ is countable too.  
It follows that $\Lambda$ is included in the closure of the convex hull of a countable set. 
Such a set is always metrizable  (see e.g. the proof of~\cite[Lemma 7.15]{valtree}, or~\cite[Lemma 5.7]{baker rumely}).
Finally $\Lambda_0$ is a subset of $\PP^1(k)$, hence $\Lm(\Gamma)$ is included in its closure which is contained in 
$\H_k \cup \PP^1(k)$. 
\end{proof}

\section{Random products of matrices in $\SLk$}\label{sec:furstenberg}

In this section we  work on an arbitrary  complete non-trivially
 valued field $(k, \abs{\cdot})$ --shortly to be assumed non-Archimedean. We keep the  notation of   
the previous section. 
We consider   a measure $\mu$ with at most countable support
in  $\SLk$, 
and make the following  assumptions:
 \begin{itemize}
 \item[(B1)] $\Gamma = \langle \supp(\mu) \rangle$ is non-elementary. 
 \item[(B2)] $\mu$ has finite first moment $\displaystyle \int \log \norm{\gamma} d\mu(\gamma)<\infty$. 
 \end{itemize}

The measure $\mu$ acts by convolution on the set of probability measures on  $\berk$ by  $\nu
\mapsto \mu \ast \nu$. The  measures invariant under this action are called {\em stationary}. 
We use the  
probabilistic notation     
 $(\om, \prob)  = (\SLk^{\nn^*}, \mu^{\nn^*})$, and for  $\omega = (\gamma_n)_{n\geq 1}\in \om$ we respectively let  
 $r_n(\omega)  = \gamma_1\cdots \gamma_n$ and  $\ell_n(\omega) = 
\gamma_n\cdots \gamma_1$   the  $n^{\rm th}$ step of the right and left random walk on $\Gamma$ 
 with transition probabilities given by $\mu$.  
 We denote by  $\mu^n$ the $n^{\rm th}$ convolution power of $\mu$, that is the image of  $\mu^
{\otimes n}$ under the map $(g_1, \ldots ,g_n)\mapsto g_1\cdots g_n$. It is also the distribution of $r_n(\omega)$ and  
$\ell_n(\omega)$.
 
 The {\em Lyapunov exponent} of  $\mu$ is defined to be the following non-negative real number:
\begin{equation}\label{def:lyapunov exponent} \chi (\mu ) : = \lim _{n\rightarrow \infty} \frac{1}{n}\int 
\log \norm{\gamma} d\mu^{n}(\gamma) = \lim _{n\rightarrow \infty} \frac{1}{n}\int \log \norm{ \gamma_1\cdots 
\gamma_n} d\mu(\gamma_1)\cdots d\mu(\gamma_n)~. \end{equation}
It follows from Kingman's sub-multiplicative ergodic 
theorem that $\chi(\mu) = \lim_{n\to\infty} \frac1n \log \norm {\ell_n(\omega)}$ for $\sP$-a.e. $\omega$. 

\medskip

The main result of  this section is the following theorem. Recall the notation $\sigma(\gamma, v) = \log\norm{\gamma V}$, where 
$V = (V_1, V_2)$ is a lift of norm 1 of $v$ and $\norm{V} = \max(\abs{V_1}, \abs{V_2})$. 

\begin{thm}\label{thm:positive} 
Let $\mu$ be a  probability measure with countable support in $\SLk$, satisfying  
 (B1). Then there  a unique stationary probability measure $\nu$ on $\berk$, that is stationary under the action of $\mu$. 
  This measure has no atoms, it is supported on the limit set of $\Gamma$,
  and when $k$ is non-Archimedean it gives full mass to $\PP^1(k)$ so that $\nu (\hh_k)=0$. 
 
   If furthermore (B2) holds, then the Lyapunov exponent   $\chi(\mu)$ 
of the associated random product of matrices  is positive and satisfies the 
   following formula 
  \begin{align}
 \chi(\mu) & = \int \sigma(\gamma, v) \, d\mu(\gamma)\, d\nu(v) \label{eq:furstenberg}.
\end{align} 

 Finally for $\prob$-a.e. $\omega$ and for $\nu$-a.e. $v$,  we have: 
 \begin{align}
 \chi(\mu)  =  \lim_{n\to\infty} \frac1n \log \sigma(\ell_n(\omega), v) \label{eq301}.
\end{align} 
  \end{thm}

  \begin{rem}
  If $\Gamma$ is generated by $\supp(\mu)$ as a semi-group then $\supp(\nu) = \mathrm{Lim}(\Gamma)$. Indeed 
  $\supp(\nu)$ is contained in $\mathrm{Lim}(\Gamma)$, closed, and $\supp(\mu)$-invariant, hence $\Gamma$-invariant. In the general case, 
  however, the inclusion $\supp(\nu)\subset \mathrm{Lim}(\Gamma)$ can be strict. 
  \end{rem}

There are many statements  of this kind in the literature, and when $k$ is archimedean 
this statement is precisely  Furstenberg's theorem on random products of
matrices \cite{furstenberg} (see \cite[Chap. II]{bougerol lacroix} for a simple exposition). 
In the non-Archimedean setting, it was observed by several authors 
 (see in particular  \cite{guivarch})  that   Furstenberg's theory 
 can be adapted  without much harm
  to local fields. The novelty here 
  is that $k$ is arbitrary, in particular may not be locally compact. This leads us to resort to Berkovich theory, and also prevents us from 
  taking cluster limits of sequences of elements in $\SLk$, which is commonplace in the classical presentation of the topic. 
  
  On the other hand it was recently proved  by Maher and Tiozzo \cite{maher tiozzo} that non-elementary
 random walks on non-necessarily proper Gromov
 hyperbolic spaces have positive drift, from which the first conclusion of the theorem follows. 

Nevertheless we  provide a complete proof of the theorem for at least two reasons: first
 our algebraic setting allows us to provide a relatively short  and self-contained
  proof, and also the 
 conclusion on the stationary measure does not straightforwardly follow from \cite{maher tiozzo} since Maher 
 and Tiozzo work in a horofunction compactification that is not directly related to the Berkovich projective line.  

\medskip

As in the classical case, the uniqueness of the stationary measure and the positivity of the Lyapunov exponent 
follow from a contraction statement, which asserts that 
  if $\nu$ is any stationary measure, then for $\prob$-a.e.  $\omega$, 
 $\norm{r_n(\omega)}\to \infty$ and 
 $(r_n(\omega))_*\nu$ converges   to a Dirac mass at a point 
 ${e(\omega)}$  which does not depend on $\nu$ (see below Lemma \ref{lem:contraction}). 
To prove this result, we adapt the   arguments of Guivarc'h and Raugi \cite{guivarch raugi}.

In the remaining of this section we assume that the norm on $k$ is \emph{non-Archimedean}.

\subsection{Uniqueness of the stationary measure} \label{subs:abstract2}
 
As a first step towards Theorem~\ref{thm:positive} in this section we prove the following result. 
\begin{thm}\label{thm:partial unique}
Let $\mu$ be a  probability measure with countable support in $\SLk$, satisfying   
 (B1), and suppose $\nu$ is a $\mu$-stationary measure having no atom and such that $\nu(\PP^1(k)) = 1$.
 
\begin{enumerate}
\item
There exists a measurable map $e: \Omega \to \PP^1(k)$ such for a.e. $\omega$, 
$r_n(\omega)_* \nu \to \delta_{e(\omega)}$. Moreover  for a.e. $\omega$ we have $\|r_n(\omega)\| \to \infty$,  and 
$d_{sph}( e(\omega), B_\att(r_n(\omega))) \le 2 \|r_n(\omega)\|^{-1}$. 
\item
The identity $\nu = \int \delta_{e(\omega)} \, d\prob(\omega)$ holds,  and $\nu$  is the unique $\mu$-stationary measure which has no atom and gives full mass to $\PP^1(k)$.
\end{enumerate}
\end{thm}

We start with the following classical lemma. 
\begin{lem}\label{lem:guivarch raugi}
Under the assumptions of Theorem \ref{thm:partial unique}, for  $\prob$-a.e.  
$\omega$ there exists a probability measure  $\nu_\omega$ 
such that for  every $\gamma$ belonging to the semi-group generated by 
$\supp(\mu)$,  the sequence $r_n(\omega)_*g_*\nu$ converges weakly to $\nu_\omega$. 
This measure a.s. puts full mass on $\PP^1(k)$ and  $\nu = \int \nu_\omega \; d\prob(\omega)$.
\end{lem}

\begin{proof}
The support of $\nu$ is a compact metrizable space by Lemma~\ref{lem:support}, so that we may apply \cite[Lemma 2.1 p.19]{bougerol lacroix}. 
We obtain  for  $\prob$-a.e.  $\omega$ the existence of a probability measure  $\nu_\omega$ 
such that for  $\mu^\infty$-a.e. $\gamma$,   $r_n(\omega)_*\gamma_*\nu$ converges weakly to $\nu_\omega$, where  
$\mu^\infty = \sum_{n = 0}^\infty 2^{-n-1}\mu^n$.  
Since $\Gamma$ is countable the measure $\mu$ is purely atomic and its support is precisely the semi-group generated by 
$\supp(\mu)$. Thus for every $\gamma$ in this semi-group, we obtain
$r_n(\omega)_*\gamma_*\nu \cvf \nu_\omega$. 

The stationarity property implies that for every $n$, $\nu = \int r_n(\omega)_* \nu\, d\prob(\omega)$, and the dominated convergence theorem implies 
   $\nu = \int \nu_\omega \; d\prob(\omega)$. In particular  $\nu_\omega(\PP^1(k)) =1$ almost surely.  
  \end{proof}
  
Next we prove the divergence of the norms of  generic random products. 

\begin{lem} \label{lem:contraction}
For $\prob$-a.e. $\omega$ we have  that   $\| r_n(\omega)\| \cv\infty$. 
 \end{lem}

\begin{proof}
Let $\omega \in G^\nn$ be a sequence satisfying the conclusion of Lemma \ref{lem:guivarch raugi}, and 
let us show that   $\| r_n(\omega)\| \cv\infty$. We proceed by contradiction, so assume there exists  
a subsequence  $(n_j)$ such that  $\|r_{n_j}(\omega))\|\leq C$ for some $C\geq 1$. 
Since $\Gamma$ is non elementary, by Lemma \ref{lem:proximal semi group}  
the semi-group generated by $\supp(\mu)$   contains two hyperbolic   elements $\gamma_1$, $\gamma_2$ 
with distinct attracting fixed points,  that we fix from now on. 
Denoting by $x_{\att}(\gamma_i)$   the respective    attractive fixed points,  we   
fix  $r$ small enough so that $d(x_\att(\gamma_1),x_\att(\gamma_2)) > 2r (C^2+1)$.

The measure $\nu_\omega$ charges $\PP^1(k)$, so that one can find a point $z\in \PP^1(k)$ such that 
$m := \nu_\omega(B(z,r))>0$. Since the 
measure $\nu$ has no atom,  for $i=1, 2$
 $ \nu(B(x_\rep(\gamma_i),\rho))$   tends to  $0$ when $\rho\to0$, 
so  we can fix  $N$ (large) such that 
$$((\gamma_1^N)_*\nu) (B(x_\att(\gamma_1),r)) \geq 1- \frac{m}{4}  \text{ and } ((\gamma_2^{N})_*\nu)( B(x_\att(\gamma_2),r)) \geq 1- \frac{m}{4}.$$

Since on the other hand $B(z,r)$ is open, from the choice of  $\omega$, there exists  $j$ such that   
$$\lrpar{ r_{n_j}(\omega)_* (\gamma_1^N)_*\nu} (B(z,r)) \geq \frac{m}{2} 
\text{ and } \lrpar{ r_{n_j}(\omega)_*(\gamma_2^{N})_*} (B(z,r)) \geq \frac{m}{2} .$$ 

Applying  Lemma \ref{lem:lipschitz} it follows that for $i=1,2$,
$$
 (\gamma_i^N)_*\nu(B(r_{n_j}(\omega)^{-1}z, rC^2))\geq m/2,$$ hence
 $$B(r_{n_j}(\omega)^{-1}z, rC^2)\cap B(x_\att(\gamma_i),r) \neq \emptyset.$$
which implies that  $d(x_\att(\gamma_1), x_\att(\gamma_2)) \le 2rC^2 + 2r$, a contradiction. 
\end{proof}

The proof of Theorem~\ref{thm:partial unique} will be complete if we prove that

\begin{lem} \label{lem:Dirac}
For $\prob$-a.e. $\omega$, the measure $\nu_\omega$ is a Dirac mass at a point $e(\omega)\in \PP^1(k)$ 
which does not depend on  $\nu$, and  satisfies $d_{\sph}(e(\omega), B_\att(r_n(\omega))) \le 2\|r_n(\omega)\|^{-1}$. 
\end{lem}
 
The proof of this lemma relies on the following elementary fact which asserts that the measure of small balls is uniformly small. 

\begin{lem}\label{lem:uniform}
Let $\nu$  be an atomless Borel probability measure on a complete metric space $(X,d)$. There exists a function 
$\eta : \R_+ \to \R_+$ such that $\eta(r)\to 0$ as   $r\to 0$ such that  for every  $x\in X$ we have that  $\nu(B(x,r))\leq \eta(r)$.
\end{lem}

\begin{proof}
Assume by way of contradiction that there exists  $\eta>0$, a sequence $(x_n)\in X^\nn$ and a sequence of 
radii  $r_n\cv 0$ such that  $\nu(B(x_n, r_n))\geq \eta$. Extracting a subsequence we may assume that  $\sum r_n$ converges. 
We will show that  $(x_n)$ 
admits a Cauchy subsequence $(x_{n_j})$, thus converging to some  $x$. Since for any $r>0$, we have that  
   $B(x_{n_j}, r_{n_j}) \subset B(x, r)$ for large $j$, it follows that 
   $\nu (B(x,r))\geq \eta$ for every  $r$,   contradicting the assumption on  $\nu$. 

To show that  $(x_n)$ admits a  Cauchy subsequence, we define the set
$$A = \set{n\in \nn,\ \forall p>n, \  B(x_p, r_p) \cap B(x_n, r_n)  = \emptyset}.$$ 
Write $A$ as an increasing sequence $A = \set{a_1, a_2, \ldots}$. By construction, for every $k>l$, $B(x_{a_k}, r_{a_k}) \cap B(x_{a_l}, r_{a_l}) = \emptyset$. Since each of these ball has mass at least $ \eta$, there are at most $1/\eta$ of them, hence $A$ is finite.

Now if $n_1> \max A$, by assumption there exists  $n_2>n_1$ such that $B(x_{n_2}, r_{n_2})\cap B(x_{n_1}, r_{n_1})\neq \emptyset$. Repeating this process we construct a subsequence  $(n_j)_{j\geq 1}$.  Since the series  $\sum r_n$ converges, the sequence  $(x_{n_j})$ is Cauchy, and the lemma follows. 
\end{proof}

\begin{proof}[Proof of Lemma \ref{lem:Dirac}]
Recall the definition of $B_\att(\gamma)$ and $B_\rep(\gamma)$ from Proposition~\ref{prop:B att rep}: these are two 
balls in $\PP^{1,\an}_k$ of spherical diameter $\|\gamma\|^{-1}$ such that $\gamma (\PP^{1,\an}_k \setminus B_\rep(\gamma)) \subset B_\att(\gamma)$. 
Observe that they necessarily intersect $\PP^1(k)$.

From Lemmas \ref{lem:guivarch raugi} and \ref{lem:contraction},  for almost every 
$\omega$ we have that  $(r_n(\omega))_*\nu\cvf\nu_\omega$ and $\| r_n(\omega)\| \to \infty$.  
With $\eta$ as in  Lemma \ref{lem:uniform}, we have that
$\nu (B_{\rm rep}(r_n(\omega)))\leq \eta ( \| r_n(\omega)\|^{-1})$, hence
\begin{equation}\label{eq:lw meas}
\nu (B_{\rm att}(r_n(\omega)))\geq  1- \eta ( \| r_n(\omega)\|^{-1}) \ge \frac34
\end{equation}
for $n$ large enough. For each $n$ choose any $x_n \in B_{\rm att}(r_n(\omega)) \cap \PP^1(k)$.
Then~\eqref{eq:lw meas} implies $d_{\sph}(x_n,x_m) \le 2\|r_n(\omega)\|^{-1}$ for all $m\ge n$
hence $x_n$ forms a Cauchy sequence. 
Observe that the limit of this sequence belongs to $\PP^1(k)$, is at distance at most $2\|r_n(\omega)\|^{-1}$ from $B_\att(r_n(\omega))$, and does   depends neither on the choice of the sequence $(x_n)$, nor on the 
stationary measure. We may thus denote it by $e(\omega)$. 

Finally, for each $r>0$, we have that $B_{\rm att}(r_n(\omega)) \subset B(x_\infty,r)$ for large $n$ since $\|r_n(\omega)\|^{-1} \to0$. 
We conclude that $\nu(B(e(\omega),r)) \ge \liminf_n B_{\rm att}(r_n(\omega)) =1$, so that $\nu_\omega = \delta_{e(\omega)}$.
\end{proof}

\subsection{Proof of Theorem \ref{thm:positive}}
Recall that $\mu$ is a measure with countable support on $\SL(2,k)$   satisfying   condition  (B1). 
We  first show the existence and uniqueness of a $\mu$-stationary measure, and then prove that this measure 
has no atom and puts full mass on $\PP^1(k)$. 
Then assuming (B2) we establish the analog of Furstenberg's formula that expresses the Lyapunov exponent $\chi(\mu)$ in terms of the stationary measure. The positivity of $\chi(\mu)$ and the identities~\eqref{eq:furstenberg} and \eqref{eq301}
follow from this formula.

\medskip

Recall first from Theorem~\ref{thm:struct limit}   that the limit set of $\Gamma$ is a compact metrizable space. 
For any $x\in \Lm(\Gamma)$, 
one can thus extract a converging subsequence from $\frac1{n} \sum_{i=0}^{n-1}\mu^ i \ast \delta_{x}$, and the limit measure $\nu$
is $\mu$-stationary. The uniqueness of the stationary measure then follows from Theorem~\ref{thm:partial unique} 
together with  the next lemma whose proof will be  given at the end of this section.

\begin{lem}\label{lem:Hk}
Let $\mu$ be a   probability measure with countable support in $\SLk$, satisfying   
 (B1).  If $\nu$ is any $\mu$-stationary probability measure on $\berk$, 
then $\nu$ has no atom and gives full mass to  $ \PP^1(k)$. 
\end{lem}

By Lemma \ref{lem:guivarch raugi} the support of $\nu$ is contained in the limit set. 
It remains to prove~\eqref{eq:furstenberg},~\eqref{eq301} and  the positivity of the Lyapunov exponent. 
From now on we assume that the moment condition (B2) holds, and   proceed in several steps.

\begin{itemize}
\item [\emph{Step 1.}] We first claim that for $(P\times \nu)$ a.e.  
$(\omega,v)$,  one has $\sigma(\ell_n(\omega), v)\cv\infty$.
\end{itemize}

To see this we introduce the reversed random walk on $\Gamma$, that is the random walk  associated to 
   $\check \mu$, the image of  $\mu$ under the involution  $\gamma\mapsto \gamma^{-1}$. It satisfies the assumptions (B1-2)
   so from what precedes we know that it admits a unique stationary measure  $\check\nu$ on $\berk$ having no atom and putting full mass on $\PP^1(k)$.
 Define an involution   $G^\nn\to G^\nn$  by $\omega = (\gamma_n)_{n\geq 1 }\mapsto \check \omega = (\gamma_n^{-1})_{n\geq 1}$, 
 so that
$\ell_n(\omega)   = r_n(\check \omega)^{-1}$, and
$B_{\rm rep}(\ell_n(\omega))  = B_{\rm att}(r_n(\check \omega))$. 

By Theorem~\ref{thm:partial unique}, we have $\|l_n(\omega)\| \to \infty$ for a.e. $\omega$, and there exists a measurable map $\check{e}: \Omega \to \PP^1(k)$ such that 
 $\check \nu = \int \delta_{\check e(\omega)} d\check\mu(\omega)$, and $d_{\sph}(\check e(\omega), B_\att(l_n(\omega))) \le 2\|l_n(\omega)\|^{-1}$. 
 
If $v\in \pp^1(k)$ is different from $\check e(\omega)$, then for  $n$ large enough,  $v\notin B_{\rm rep}(\ell_n(\omega))$ hence 
$\sigma(\ell_n(\omega), v )= \norm {\ell_n(\omega)}$ by Lemma~\ref{lem:expand vector 1}. Since $\nu$ gives no mass to points, 
for $\prob$-a.e. $\omega$ and for $\nu$-a.e. $z$, we have  $\sigma(\ell_n(\omega), v )\to \infty$, thus
 by   Fubini's  theorem,  $\sigma(\ell_n(\omega), v)\cv\infty$ holds  $(P\times \nu)$ a.s., as claimed.

Denote by $\theta : \Omega \to \Omega$ the shift map $\omega = (\gamma_n)_{n\ge 1} \mapsto \theta(\omega) = (\gamma_{n+1})_{n\ge 1}$, and observe that
the skew product map $\Theta(\omega,v) = (\theta(\omega), l_1(\omega) \cdot v)$ preserves the measure $m = \prob\times \nu$. 

\begin{itemize}
\item [\emph{Step 2.}] 
  We next show that $\prob\times \nu$ is ergodic under  $\Theta$. 
\end{itemize}

From  (Kifer's  version of) the  Kakutani  random ergodic theorem (see \cite[Thm. 3.1]{furman}), it is sufficient to prove that 
for any Borel set $E$ such that  $\nu(E\Delta \gamma E)=0$ for $\mu$-a.e. $\gamma$, then we have $\nu(E)=0$ or $1$. 

Pick such a set $E$. First note  that by stationarity and by the countability of $\supp(\mu)$, for $\mu$-a.e. $\gamma\in \Gamma$, we have $\gamma_*\nu\ll\nu$. In particular, we get 
$\gamma_*\nu(E\Delta \gamma E)=0$  so that 
$$
\mu \ast (\nu|_E) = \int \gamma_*\nu \, \mathbf{1}_{\gamma (E)} \, d\mu(\gamma)
=  \int \gamma_*\nu \, \mathbf{1}_{E} \, d\mu(\gamma) =\nu|_E~.
$$
By the uniqueness of the $\mu$-stationary measure, we conclude that $\nu(E) =0$ or $1$, as required. 

\begin{itemize}
\item [\emph{Step 3.}]  Proof of~\eqref{eq:furstenberg} and~\eqref{eq301}.
\end{itemize}

The  Birkhoff ergodic theorem applied to $\Theta$ and the function $\sigma(l_1(\omega),v)$, together with 
 the cocycle relation~\eqref{eq:cocycle} now yield
for $(\prob,\nu)$-a.e. $(\omega,v)$
\begin{equation}\label{eq001}
\lim_{n\to\infty} \frac1n \sigma(l_n(\omega),v) = \int \sigma (l_1(\omega),v) \, d\prob(\omega)\, d\nu(v) =  \int   \sigma(\gamma, v) d\mu( \gamma)d\nu(v).
\end{equation}
By Fubini's theorem, for $\prob$-a.e. $\omega$ we have that for $\nu$-a.e. $v$,   the limit in \eqref{eq001} exists. Since $\nu$ 
gives no mass to points, this holds for at least two distinct points so by 
 Lemma~\ref{lem:two points} we get that for such $\omega$
\begin{equation}\label{eq002}
\lim_{n\to\infty} \frac1n \log \| l_n(\omega)\| = \lim_{n\to\infty} \frac1n \sigma(l_n(\omega),v)~.
\end{equation}
Thus, combining~\eqref{eq001} and~\eqref{eq002} we get~\eqref{eq:furstenberg} and~\eqref{eq301}.

\begin{itemize}
\item  [\emph{Step 4.}]  Positivity of the Lyapunov exponent. 
\end{itemize}

To that end we rely on the following general lemma, a proof of which can be found in \cite[Lemma 2.3, p. 22]{bougerol lacroix}.  For simplicity we write $F^+ = \max \{ 0, F\}$.
 
\begin{lem}\label{lem:birkhoff}
Let  $\Theta$ be a measurable map on a probability space  $(X, m)$. If $F:X\to \re$ is a measurable function such that
$\int F^+ dm <\infty$ and  $\lim_{n\cv\infty} \sum_{i=0}^{n-1} F\circ \Theta^i = +\infty $ a.s., then  
$F\in L^1(X,m)$ and $\int F \,dm >0$. 
 \end{lem}
We apply the previous lemma to  the function
$F: (\omega, v) \mapsto \sigma(\ell_1(\omega), v)$  
on $X=\om\times \berk$, and  to the skew product map $\Theta(\omega,v) = (\theta(\omega), l_1(\omega) \cdot v)$. 
By Step 1 for $\prob\times \nu$-a.e $(\omega,v)$ we have that 
$$
\lim_{n\cv\infty} \sum_{i=0}^{n-1} F\circ \Theta^i (\omega,v) = \lim_{n\cv\infty}  \sigma(l_n(\omega),v) = +\infty.
$$
Therefore Lemma \ref{lem:birkhoff} yields
\begin{equation}\label{eqL}
\int \sigma(\gamma, v)\, d\mu(\gamma)d\nu(v)
>0
~,
\end{equation}
 and the proof of Theorem~\ref{thm:positive} is complete.

\begin{proof}[Proof of Lemma~\ref{lem:Hk}]
Assume by contradiction that $\nu$ charges a point in $\PP^{1,\an}_k$. Then for every $\alpha>0$, the set 
$\set{x, \nu(\set{x})>\alpha}$ is finite, thus there exists an atom of maximal mass $\alpha_0$. It follows that the set $\set {x, \nu(\set{x})=\alpha_0}$ is finite and invariant under every element of $\supp(\mu)$,
hence  $\Gamma$-invariant, 
which is impossible since $\Gamma$ is non-elementary.

For second assertion, we start by proving that 
 that  $\nu(\hh_k) =0$. Observe first that any closed ball  for the hyperbolic metric  is also closed in $\PP^{1,\an}_k$ hence is a Borel set. 
Denote by  $\B(x,R) = \{ d_{\H}(\cdot, x) < R\}$ (resp. $\overline{\B}(x,R) = \{ d_{\H}(\cdot, x) \le R\}$) 
   the open (resp. closed) ball of radius $R$ relative to 
 the hyperbolic metric.
It is thus enough to prove that  for every $x\in \hh_k$ and   every 
$R>0$, $\nu(\overline{\B}(x,R)) = 0$.

By Lemma \ref{lem:proximal semi group}, the semi-group generated by $\supp(\mu)$
 contains a hyperbolic element  $h$. 
Replacing  $\mu$ by some $\mu^k$ if necessary (this does not affect the stationarity of $\nu$) 
we may assume
$h\in \supp(\mu)$. Put $s= \sup_{x\in \hh_k} \nu(\B(x,R))$ and suppose by way contradiction that 
$s>0$.  

Since the series  $\sum_{n\geq 0} \nu \lrpar{\overline{\B}(x_{\mathrm{g}}, (n+1)R)\setminus \B(x_{\mathrm{g}}, R)}$ converges,  there exists $B$ such that 
for $d_\H(x,x_{\mathrm{g}})>  B$,   
$\nu(\B(x,R))\leq s/2$. Since $h$ is hyperbolic and is an isometry for $d_\H$, we have $d_\H(h^m(x_{\mathrm{g}}) , x_{\mathrm{g}}) \to \infty$, hence 
there exists an integer $m$ such that $d_\H(x,x_{\mathrm{g}}) \le B$ implies   $d_\H(h^m(x), x_{\mathrm{g}})\ge 2B$ hence $\nu (\B(h^m(x),R)) \le s/2$. 
   
   Put  $\e = \mu^m(\set{h^m})$ and pick
$x\in \hh_k$ such that $\nu(\B(x,R))\geq s(1-\e/3)$. In particular $d_{\hh_k}(x,x_{\mathrm{g}})\le B$. The invariance relation
 $\mu^m\ast \nu = \nu$ yields
\begin{align*}
s\lrpar{1-\frac{\e}{3}}  \leq \nu(\B(x,R)) &= \sum_{\gamma \in \Gamma} \mu^m(\set{\gamma}) \nu(\gamma (\B(x,R)))  \\
&\leq \sum_{\gamma \in \Gamma\setminus{h^m} } \mu^m(\set{\gamma}) \nu(  \B(\gamma x,R)) + 
\mu^m(\{h^m\}) \nu(\B(h^m(x), R))\\
&\leq (1- \mu^m(\{h^m\})) s + \mu^m(\{h^m\}) \frac{s}{2}  = s \lrpar{1-\frac{\e}{2}}.
\end{align*}
From this  contradiction we conclude that $\nu(\hh_k)=0$. 

To prove that $\nu$ gives full mass to $\PP^1(k)$,  consider
 the projection  $\pi$   on  the closed subtree 
$\overline{\mathrm{Conv}(\PP^1(k))} = \hh_k \cup \PP^1(k)$. Since $\hh_k \cup \PP^1(k)$ is $\Gamma$-invariant, for any 
$x\in \PP^{1, \an}_k$ and any $\gamma\in \Gamma$ we have that $\pi(\gamma(x))   = \gamma(\pi(x))$. It follows that 
$\pi_*\nu$ is stationary. By the first part of the proof, $\pi_*\nu$ gives full mass to $\PP^1(k)$ (which is a Borel set). 
Hence $\nu$ gives full mass to 
$\pi\inv(\PP^1(k)) = \PP^1(k)$. This completes the proof. 
\end{proof}

\subsection{Distribution of attracting and repelling fixed points}
In the course of the proof of our main result we shall need the following interpretation of the 
stationary measures as the distribution of fixed points of hyperbolic elements.

Recall that the dual measure $\check{\mu}$ is defined as the image of $\mu$ under the involution $\gamma
\mapsto \gamma^{-1}$, and
that the measure $\mu$ satisfies (B1) (resp. (B2)) iff $\check{\mu}$ does. 
We also set $\check{\omega} = (\gamma_n^{-1})$ when $\omega  = (\gamma_n)$ so that $l_n(\omega) = r_n(\check{\omega})^{-1}$.
Observe that $\|l_n(\omega)\| = \|r_n(\check{\omega})^{-1}\| =  \|r_n(\check{\omega})\|$ hence $\chi(\check{\mu}) = \chi(\mu)$. 

\begin{thm}\label{thm:loxodromic}
Let $\mu$ be a   probability measure with countable support in $\SLk$, satisfying    
 (B1).  
  Then 
 \begin{equation}\label{eq:proba hyperbolic}
 \prob\lrpar{\left\{ \omega, \, r_n(\omega)
 \text{ is hyperbolic}\, \right\}}
  \underset
  {n\cv\infty}\longrightarrow 1.
  \end{equation}  
  In addition, the asymptotic distribution for the weak-$\ast$ topology on $\berk$
   of the attracting (resp. repelling) fixed point of $r_n(\omega)$   is given by the unique $\mu$-stationary (resp. $\check{\mu}$-stationary) probability measure $\nu$ (resp. $\check{\nu}$). 
   
   If furthermore $\mu$ satisfies (B2), then for every $\e>0$
 \begin{equation}\label{eq:trace}
 \prob\lrpar{\set{
    \abs{ \unsur{n}\log\abs{\tr (r_n(\omega))} - \chi(\mu)} <\e }} \underset
  {n\cv\infty}\longrightarrow 1.
  \end{equation}
\end{thm}

\begin{rem}\label{rem:301}
The meaning of the   statement on the distribution of periodic points is the following. 
For each $n$ let $\Omega_n$ be the set of $\omega\in \Omega$ such that $r_n(\omega)$ is hyperbolic. One can then
define the measurable map $\Att^n, \Rep^n: \Omega_n \to \PP^1(k)$ by sending $\omega$ to the attracting and repelling fixed points of $r_n(\omega)$. 
The theorem asserts that $\prob(\Omega_n) \to1$ when $n\to\infty$, and 
 $\Att^n_* \prob \to \nu$, and  $\Rep^n_* \prob \to \check{\nu}$ as $n\to\infty$.
\end{rem}

\begin{rem}\label{rem:302} Under stronger moment assumptions on $\mu$, \eqref{eq:proba hyperbolic} can be turned
  into an almost sure limit. For instance it is shown in 
 \cite[Thm. 1.4]{maher tiozzo}  that if
 $\mu$ has bounded support the probability in \eqref{eq:proba hyperbolic} is exponentially close to $1$, thus by 
  Borel-Cantelli, $\unsur{n}\log\abs{\tr (r_n(\omega))}$ converges a.s. to $\chi(\mu)$.    
\end{rem}

\begin{proof} 
Consider the probability measure $\nu\times\check{\nu}$ on the product space $\PP^1(k) \times \PP^1(k)$.
This measure is the image of $\prob$ under the measurable map $\omega \mapsto (e(\omega), e(\check{\omega}))$. 

By Fubini, and since the measures $\nu$ and $\check{\nu}$
do not have any atom, we have $\nu\times\check{\nu} (\Delta) = 0$ where $\Delta$ denotes the diagonal in $\PP^1(k) \times \PP^1(k)$.
It follows that 
\begin{equation}\label{eq333}
(\nu\times\check\nu)\set{(x,y)\in \pu(k)^2, \ d_{\rm sph}(x,y)\leq \delta} \underset
  {\delta\cv0}\longrightarrow 0~.
  \end{equation}
By Theorem~\ref{thm:partial unique} $\| r_n(\omega)\|\to\infty$ a.s. and  the asymptotic distribution of the $B_\att(r_n(\omega))$ is given by $\nu$. 
Similarly the asymptotic distribution of the $B_\rep(r_n(\omega))$ is given by $\check{\nu}$, so that 
for each $\delta>0$, we infer that 
$$\limsup_{n\cv\infty} \prob\left\{d_{\rm sph}(B_\att(r_n(\omega)), B_\rep(r_n(\omega))) \leq \delta\right\} \leq 
(\nu\times\check\nu)\set{(x,y)\in \pu(k)^2, \ d_{\rm sph}(x,y)\leq \delta}~,$$
which can be made as small as we wish.

Fix any real number $\varepsilon>0$, and choose $\delta>0$ such that the left hand side in~\eqref{eq333} is at most $\e/2$. 
Then there exists $N = N(\e)$ and  a set $\Omega_\varepsilon$ with $\prob(\Omega_\e)\geq 1-\e/2$,  
 such that if $\om\in \Omega_\e$ and for $n\ge N(\omega)$, we have
$$
d_{\rm sph}(B_\att(r_n(\omega)), B_\rep(r_n(\omega))) > \delta~ .
$$
In addition   $\| r_n(\omega)\|\to\infty$ a.s,  so increasing $N$ and discarding a set of probability $\e/2$  
we may further  assume that 
$$
\| r_n(\omega) \|\ge 2\delta\inv \text{ on } \Omega_\varepsilon~.$$ 
Now if we pick $\omega \in \Omega_\varepsilon$ and   $n\ge N$. The two closed balls $B_\att(r_n(\omega)), B_\rep(r_n(\omega))$
have diameter at most $ \delta/2$ hence are disjoint. Since $r_n(\omega)$ maps $\PP^{1,\an}_k\setminus  B_\rep(r_n(\omega))$ into 
$B_\att(r_n(\omega))$,   we conclude by Lemma \ref{lem:distinct balls}. 

If furthermore (B2) holds then from Lemma~\ref{lem:expand vector 2}, we conclude that 
$$
\left| \log |\tr(r_n(\omega)| - \log \| r_n(\omega)\| \right| 
\le 
- \log \min \{ 1, \delta\}~,
$$
and \eqref{eq:trace} follows since $\lim_{n\to\infty} \frac1n \log\|r_n(\omega)\| = \chi(\check{\mu})= \chi(\mu)$.
\end{proof}

\section{Degenerations: non elementary representations}\label{sec:non elem}
 In this section we fix a  finitely generated group  $G$, endowed with some  probability measure $\sm$. 
 Recall the notation  $\Laurent = \cdt$ and that $\mm$ is the ring of holomorphic functions in 
 $\dd$ with meromorphic extension at the origin. We
  fix a representation  $\rho: G\to \SL(2, \mm)$, that is a family of representations
 $\rho_t : G \to \SLC$ for $t\in \D^*$ such that for any $g\in G$
  $ t\mapsto \rho_t(g)$ is holomorphic on $\D^*$ and extends meromorphically through the origin. We suppose that 
 \begin{itemize}
 \item[(A1)]  $\supp(\sm)$ generates 
 $G$;
 \item[(A2)] $\sm$ has finite first moment $\displaystyle \int \length(g) d\sm(g)<\infty$. 
 \end{itemize}

We also assume that if $\rho_\na$ denotes the induced representation $\rho_\na: G \to \SL(2,\Laurent)$, then 
$\rho_\na(G)$ is non-elementary. 
Our aim is to prove Theorem~\ref{Thm:atomic} and to infer the non-elementary case of 
Theorem~\ref{Thm:lyapunov na}.

 \subsection{Basic remarks on meromorphic families of representations}
     
 Since we have to deal with both Archimedean and non-Archimedean objects, we slightly change notation: $\norm{\cdot}$ 
 denotes the operator norm in $\SLC$ associated to 
any norm on $\cd$, say  $\norm{(x,y)} =\max({\abs{x}, \abs{y}})$,  and  
 $\norm{\cdot}_{\na}$  denotes the non-Archimedean norm on $\SL(2, \Laurent)$ associated to the $t$-adic norm in $\mm$ given by
$\abs{f}_\na = \exp(-\ord_{t=0}(f))$.  
  More generally we use the subscript $\na$ to label non-Archimedean objects.

For any $t\in \D^*$, we set $ \Gamma_t =\rho_t(G)$ and let $\mu_t$ be the push-forward of $\mu$ under $\rho_t$. 
Analogously, denoting by $\rho_\na:G\to \SL(2, \Laurent)$ the non-Archimedean representation naturally associated to 
$\rho$,  we let  $\Gamma_\na = \rho_\na(G)$  and $\mu_\na = (\rho_\na)_*\mu$.
 
Observe that for every $g\in G$, $\log \norm{\rho(g)}_\na\leq C \length (g)$   for some uniform constant $C>0$, and likewise for 
$\log\norm{\rho_t(g)}$. 
In particular we have
\begin{lem}\label{lem:AtoB}
The condition (A2) implies the moment condition (B2) for the  measures $\mu_t$ and $\mu_\na$. 
\end{lem}
   
  We will need some uniformity on the control of $\norm{\rho_t(g)}$.

\begin{lem}\label{lem:uniform disk}
For every homomorphism   $\rho: G\to \SL(2, \mm)$  there exists $C = C(\rho)>0$ such that 
 for every $g\in G$ one can write $\rho_t(g) = t^{-\log\|\rho(g) \|_\na}\cdot \widetilde{\gamma}(t)$, with $\widetilde{\gamma}\in \SL(2,\cO(\D))$, and
\begin{equation}\label{eq:uniform disk}
\abs{\log\norm{\widetilde \gamma}_{L^{\infty}(\overline{\D}(0,1/2))} }\leq C\, \mathrm{length}(g).
\end{equation}
\end{lem}
\begin{proof}
let $(s_i)$ be a finite symmetric set of generators of $G$  and 
 write $g$ as a reduced word in $G$, $g = s_{i_1}\cdots s_{i_n}$, $n = \mathrm{length(g)}$. Let $\sigma_i = \rho(s_i)$ and 
 write $\sigma_i   = t^{-\alpha_i} \widetilde \sigma_i$ with $\alpha_i = \log\norm{\sigma_i}_\na$ and $\widetilde \sigma_i$ holomorphic and non vanishing at 0, 
so that  $$\rho_t(g)= t^{-\sum_{j=1}^n\alpha_{i_j}} \; \widetilde \sigma_{i_1}\cdots \widetilde \sigma_{i_n}~.$$
Set $A:= \max \alpha_i$, and write
$\rho_t(g) = t^{-\log\|\rho(g) \|_\na}\cdot \widetilde{\gamma}(t)$ with $\widetilde{\gamma} \in \SL(2, \cO(\D))$. Then we get that 
$\widetilde{\gamma}(t) = t^{-\alpha} \cdot \widetilde \sigma_{i_1}\cdots \widetilde \sigma_{i_n}(t)$  with $0\le \alpha  = \sum_{j=1}^n\alpha_{i_j} -  \log\|\rho(g) \|_\na \le An$. 
By using the maximum principle we can estimate
\begin{multline*}
\sup_{|t|\le1/2}|\widetilde{\gamma}(t)| \le  \sup_{|t|=1/2} |\widetilde{\gamma}(t)| \le 
 (1/2)^{-\alpha}  \prod_{j=1}^n \lrpar {\sup_{|t|=1/2} \abs{\widetilde \sigma_{i_j}(t) }}\\\le 
 2^{An} \, \left(\max_{i}\norm{ \widetilde \sigma_i}_{L^\infty(D(0, 1/2))} \right)^n\le D^n 
 \end{multline*}
for some $D\geq1$. Likewise we have 
that 
$$\sup_{|t|\le1/2}|\widetilde{\gamma}(t)| \ge  \sup_{|t|=1/2} |\widetilde{\gamma}(t)| 
\ge  (1/2)^{-\alpha}  \sup_{|t|=1/2} \abs{\widetilde \sigma_{i_1}\cdots \widetilde \sigma_{i_n}(t)}
=  (1/2)^{-\alpha}   \abs{\widetilde \sigma_{i_1}\cdots \widetilde \sigma_{i_n}(0)} \ge E^n$$ for some $E>0$
 and we are done.
\end{proof}

Let us  also note the following basic but crucial observation. 
 
\begin{lem}\label{lem:specialization NE}
If $\Gamma_\na$ is non-elementary, then  for  small enough $t\in \dd^*$, 
$\Gamma_t\leq \SLC$ is non-elementary. 
\end{lem}

\begin{proof} 
By definition $\Gamma$ contains two hyperbolic elements $\gamma_1$ and  $\gamma_2$ with disjoint fixed points on $\PP^{1,\an}_{\Laurent}$. Since for $i=1,2$
$\abs{\tr (\gamma_i)}_\na>1$, it follows that $\abs{\tr(\gamma_{i,t})} \to \infty$ as $t\cv 0$. Thus $\gamma_{1,t}$ and 
$\gamma_{2,t}$ are loxodromic for 
small $t$, and  they have well-defined attracting and repelling fixed points $\att(\gamma_{i,t})$, $\rep(\gamma_{i,t})$.
Saying that  $\gamma_1$ and  $\gamma_2$ have   disjoint fixed points sets 
on $\berkl$ implies that the formal expansions of 
the curves $t\mapsto \att(\gamma_{i,t})$ and  $t\mapsto \rep(\gamma_{i,t})$ are all distinct, 
thus the corresponding points in $\pu$ 
must be disjoint in some punctured neighborhood of the origin. 
\end{proof}

 \subsection{Models and $\PP^{1,\an}_{\Laurent}$}  \label{subs:models}
In this paragraph we review  the notion of model. 
Set  $X = \dd\times \PP^1_\C$. A (bimeromorphic) model of $X$ is a surface $Y$ together with a bimeromorphic holomorphic map
$\pi_Y:Y\to X$ that is biholomorphic above $X\setminus (\{0\}\times \pu)$. The fiber $\pi_Y\inv(\set{t}\times \pu)$ will be denoted by $Y_t$. 
We will only consider the case where $Y$ is smooth, in which case $\pi_Y$ is simply a composition of
point blow-ups above the central fiber and $Y_0$ is a divisor with simple normal crossings.

We say that a
 model $Y'$ dominates a model $Y$ if the birational map $\pi_{Y'}:Y'\cv X$ factors through $Y$. The set of 
 models is then directed   in the sense that given two models $Y$ and
 $Y'$, there exists a third one $Y''$ dominating both.
  
  \medskip
 
We now explain the basic  correspondence between models and finite subsets of $\berkl$.  
Recall as a set, $\PP^{1,\an}_{\Laurent}$
 is the one point compactification of the space of multiplicative semi-norms on $\Laurent[z]$ whose restriction to $\Laurent$
is the $t$-adic norm. Let $Y$ be any model, and pick any irreducible component $E$ of $Y_0$: we will
 define a type 2 point $\zeta_E \in \PP^{1,\an}_{\Laurent}$.

Observe that any element  $f\in \C(t)[z]$ defines a rational function on $\D \times \PP^1_\C$ so that we can 
define 
$$
\abs{f}_{\zeta_E} = |f(\zeta_E)| = \exp \left( - \frac1{b_E}\, \ord_E( f \circ \pi_Y)\right)
$$ 
where $b_E = \ord_E ( t \circ \pi_Y)$. Dividing by $b_E$ guarantees that $\zeta_E\rest{\Laurent} = \abs{\cdot}$ and 
also that the definition is model-independent in the sense that 
if  $Y'$ dominates $Y$ and $E'$ is the strict transform of $E$ then $\zeta_E=\zeta_{E'}$. 
Since the field $\C(t)$ is dense in $\Laurent$,  it follows that $\zeta_E$ 
extends uniquely to a semi-norm on $\Laurent[z]$
hence defines a point in $\PP^{1,\an}_{\Laurent}$. This point is of type 2 and is defined over the field extension $\C(\!(t^{1/{b_E}})\!)$. 
Conversely, any type 2 point of $\PP^{1,\an}_{\Laurent}$ is equal to $\zeta_E$ for some irreducible component $E$ of the central fiber of some model $Y$ over $X$, see~\cite[Lemma~7.16]{NL}.

We denote by $ S(Y)$ the set of all type 2 points $\zeta_E \in \PP^{1,\an}_{\Laurent}$ where $E$ ranges over the set of irreducible components of $Y_0$.
It is an elementary fact  that $S(Y') \supset S(Y)$ if and only if $Y'$ dominates $Y$. 

\begin{prop}(see \cite[\S 4.1]{demarco faber 2})
Let $S$ be any finite set of type 2 points in $\PP^{1,\an}_{\Laurent}$. Then there exists
 a (smooth) model $Y$ such that $S \subset S(Y)$. 
\end{prop}

\begin{rem} 
In fact there is  an isomorphism of partially ordered sets between finite sets of type 2 points endowed with the inclusion and proper bimeromorphic maps $\pi: Y \to \PP^1_\C \times \D$
with $Y$ a normal complex analytic variety, see~\cite[Theorem~7.18]{NL}. 
\end{rem}

The previous proposition together with  Proposition~\ref{prop:subtree Hk} yield the following corollary. 
 \begin{cor}\label{cor:cut to small}
 Let  $\nu$ be any probability measure on  $\berkl$ which is non atomic and gives zero mass to $\hh_\Laurent$. 
 Then for every $\e>0$ there exists a model $Y$ such that
 every connected component $U$ of $\berkl\setminus S(Y)$ satisfies $\nu(U)<\e$. 
\end{cor}

\begin{rem}\label{rem:useless}
Suppose that $\nu$ puts full mass on the set $\PP^{1,\an}(\Laurent)$ of points of type $1$ defined over $\Laurent$. 
Then one can actually choose $\pi: Y \to \PP^{1}_\C \times \D$ to be a composition of blow-ups
at \emph{free} points, i.e. lying in the regular locus of the central fiber. In other words, one can choose
the central divisor $\pi^* (\PP^{1}_\C \times \{ 0 \})$ to be \emph{reduced}. 
\end{rem}

\subsection{Reduction map and action of $ \SL(2,\mm)$}\label{subs:reduction}

Fix any model $Y$, and let $\zeta$ be any type 2 point in $\PP^{1,\an}_{\Laurent}$. Then we
 can find a model $Y'$ dominating $Y$ and an irreducible component $E$ of $Y'_0$ such that
$\zeta = \zeta_E$. If the natural map $\pi_{Y',Y}: Y' \to Y$ contracts $E$ to a point,  
 we set $\red_Y(\zeta_E) = \pi_{Y',Y}(E)$. Otherwise we let $\red_Y(\zeta_E)$ be the 
generic point of the curve $\pi_{Y',Y}(E)$ which is a (non-closed) point in the $\C$-scheme $Y_0$. 
The mapping  $\red_Y$ is called the \emph{reduction map}. 
It extends canonically to an anti-continuous map $\red_Y: \PP^{1,\an}_{\Laurent}\to Y_0$
(i.e. the preimage of a closed set in the Zariski topology is open for the 
Berkovich topology), see e.g.~\cite[\S 5.2.4]{temkin}, or~\cite[\S 4.2]{demarco faber 2}.

It can be shown that it $\eta_E$ is
 the generic point   of an irreducible component $E$ of $Y_0$, then
  $\red_Y^{-1}(\eta_E) = \{ \zeta_E\}$. If $p\in Y_0$ is a closed point then $\red_Y^{-1}(p)$ 
  is a connected component of  $\berkl\setminus {S}(Y)$   whose boundary
consists of the points $\zeta_E$ where $E$ ranges over all  irreducible components of $Y_0$ containing $p$. In particular
the boundary of $\red_Y^{-1}(p)$ consists of one or two points. 
 
The reduction map behaves well under proper modifications.
 
 \begin{lem} \label{lem:reduction modification}
 If   $\pi: Y'\to Y$ is a birational morphism, then $\red_Y = \pi\circ \red_{Y'}$.
 \end{lem}
 
 \begin{proof}
 If $\zeta$ is any type 2 point, there exists a model $Y''$ dominating $Y'$ such that $\zeta = \zeta_E$ for some component $E$ of $Y''_0$. Then it
 follows immediately from the definitions that $\red_Y (\zeta)= \pi\circ \red_{Y'}(\zeta)$. This identity then extends to $\berkl$ because  
  $\red_Y$ and $\red_{Y'}$ admit   unique anti-continuous extensions to $\berkl$ and $\pi$ is continuous for the Zariski topology. 
 \end{proof}

Let us now pick  $\gamma \in \SL(2, \mm)$, and denote by   $\gamma_\na$ its natural image in  $\SL(2, \Laurent)$. 
Observe that  $\gamma$ induces a biholomorphism from 
$\dd^*\times \PP^1_\C$ to itself  commuting with the first projection, and extending meromorphically 
to $\dd \times \PP^1_\C$. More generally   
given any two models $Y,Y'$ this biholomorphism extends to a  bimeromorphic map 
$\gamma_{Y,Y'}: Y \dashrightarrow Y'$. Its properties can be described from $\gamma_\na$ and the reduction map as follows. 

 \begin{prop}\label{prop:val-geom}
 Let $Y,Y'$ be two models over $X$, and pick $\gamma\in \SL(2,\mm)$. 	
 \begin{enumerate}
 \item
The induced bimeromorphic map $\gamma_{Y,Y'}: Y \dashrightarrow Y'$ has an indeterminacy point at $p\in Y_0$ iff
there exists a type 2 point $\zeta \in \red^{-1}_{Y}(p)$ such that $\gamma_\na(\zeta) \in S(Y')$.
 \item
 Suppose $\gamma_{Y,Y'}$ is holomorphic at $p\in Y_0$. Then for any $\zeta \in \red_{Y}^{-1}(p)$, we have 
 $\gamma_{Y,Y'}(p) = \red_Y' (\gamma_\na(\zeta))$.
 \end{enumerate}
 \end{prop}

\begin{proof}
The proposition follows easily from the basic properties of the reduction map together with the following lemma. 
\end{proof}

\begin{lem}  
Let $\zeta\in\berkl$ be a type 2 point. Fix a model  $Y_1$ dominating  $Y$ and a component $E$ of $(Y_1)_0$ such that 
$\zeta = \zeta_E$. Let $\zeta' = \gamma_\na(\zeta)$, and fix  a model  $Y'_1$ dominating  $Y'$ and a 
component $E'$ of $(Y'_1)_0$ such that 
$\zeta' = \zeta_{E'}$.
Then $\gamma_{Y_1, Y'_1}(\eta_E) = \eta_{E'}$, where $\eta_E$ (resp. $\eta_{E'}$)
 is the generic point of $E$ (resp. $E'$). 

Conversely, given any two models $Y_1$, $Y_1'$ and respective irreducible components $E \subset (Y_1)_0$
and $E' \subset (Y'_1)_0$, if $\gamma_{Y_1, Y'_1}(\eta_E) = \eta_{E'}$ then $\zeta_{E'} = \gamma_\na(\zeta_{E})$. 
\end{lem}
  
\begin{proof} 
The pull-back by $\gamma$ of  any rational function on $Y'_1$ vanishing at $\eta_{E'}$ 
necessarily vanishes at $\eta_E$ since $\zeta_{E'} = \gamma_\na(\zeta_E)$. Since
for any point $p$ not lying on $E'$ there exists a rational function on $Y'_1$ which 
is non-zero at $p$, and zero on $E'$, the first claim follows. 

For the second claim, pick $f$ a rational function on $Y'_1$, and choose points $p\in E$, $p'\in E'$ such that 
$\gamma_{Y_1, Y'_1}$ is regular at $p$, $p' = \gamma_{Y_1, Y'_1}(p)$, and $E$ (resp. $E'$) is regular at $p$ (resp. at $p'$). 
Then we may choose coordinate $(x,y)$ at $p$ and $(x',y')$ at $p'$ such that 
$E = \{ x=0\}$, $E' = \{ x'=0\}$ and $\gamma_{Y_1, Y'_1} (x,y) = (x^k, \star)$. It follows 
that 
\[-\log|f(\gamma_\na(\zeta_{E}))| = \frac1{b_E}\ord_{E}(f \circ \gamma_{Y_1, Y'_1}) = \frac{k}{b_E}\ord_{E'}(f)
\]  which implies $ b_{E'}  = b_E/k$ and $\zeta_{E'} = \gamma_\na(\zeta_{E})$.
\end{proof}

  Using the reduction map, one can  see  that  for large 
  $\norm{\gamma}_\na$,  the meromorphic map $\gamma_Y$ acts on the central fiber $Y_0$ like  a (brutal) North-South transformation.

\begin{prop}\label{prop:nordsud}
Let $Y$ be any model. 

There exists a constant $C=C(Y)$ depending only on $Y$ such that 
 for  any $\gamma \in \SL(2,\mm)$ such that $\norm{\gamma}_\na\geq C$, there exist two points 
 $\mathrm{att}(\gamma_Y)$ and $\mathrm{rep}(\gamma_Y)$ in $Y_0$ 
 (not necessarily distinct) such that the induced bimeromorphism  $\gamma_Y: Y \dashrightarrow Y$  is holomorphic on $Y_0\setminus \mathrm{rep}(\gamma_Y)$ and 
 $\gamma_Y(Y_0\setminus \mathrm{rep}(\gamma_Y) )  = \mathrm{att}(\gamma_Y)$.
 \end{prop}
\begin{proof} 
Choose $C > \max_{S(Y)} \exp(d_\H (\zeta, x_{\mathrm{g}}))$, and pick $\gamma$ of norm $\ge C$.

By Proposition \ref{prop:B att rep} there exist two disjoint closed  Berkovich disks $B_{\rm att}(\gamma)$ and 
$B_{\rm rep}(\gamma)$   of spherical diameter $\norm{\gamma}^{-1}_\na$ such that 
 $\gamma (\PP^{1,\an}_{\Laurent}\setminus B_{\rm rep}(\gamma)) \subset B_{\rm att}(\gamma)$.
 
  Observe that $ B_{\rm rep}(\gamma)$ cannot contain any point $\zeta \in S(Y)$ since otherwise
we would get that  $d_\H (\zeta, x_{\mathrm{g}}) \ge \log \norm{\gamma}_\na^{-1}$, contradicting the choice of $C$. 
Thus, $ B_{\rm rep}(\gamma) \cap S(Y) =  \emptyset$, and
since $B_{\rm rep}(\gamma)$ is connected, it is contained in a connected component of $\PP^{1,\an}_{\Laurent}\setminus S(Y)$.
It follows that $B_{\rm rep}(\gamma) \subset  \red_Y^{-1}(\mathrm{rep}(\gamma_Y))$. 
Similarly we have $B_{\rm att}(\gamma) \cap S(Y) =  \emptyset$, and $B_{\rm att}(\gamma)\subset \red_Y^{-1}(\mathrm{att}(\gamma_Y))$. 

Now pick any point $p \in Y_0$ different from $\mathrm{rep}(\gamma_Y)$. Then $\red_Y^{-1}(p)$ is disjoint from $B_{\rm rep}(\gamma)$ so it is mapped into
$B_{\rm att}(\gamma)$ by $\gamma_\na$. The first item of Proposition~\ref{prop:val-geom} then asserts that 
 $\gamma_{Y}$ is holomorphic at $p$ and the second one 
   that  $\gamma_{Y}(p) =\red_Y( B_{\rm att}(\gamma)) = \mathrm{att}(\gamma_Y)$.
    \end{proof}
    
\begin{rem}\label{rem:more precise}
The proof shows that  $\mathrm{rep}(\gamma_Y) = \red_Y (\zeta)$ for any $\zeta \in B_{\rm rep}(\gamma)$. 
In particular if $\gamma$ is hyperbolic and $\zeta_{\mathrm{rep}}$ is its repelling fixed point, then
$\mathrm{rep}(\gamma_Y) = \red_Y (\zeta_{\mathrm{rep}})$.
 Also
if $S(Y)$ contains the Gau{\ss} point  then taking $C > \exp (\mathrm{diam}_\mathbb{H} (S(Y)))$ is enough. 
\end{rem}

Let $\nu$ be any Radon measure on the Berkovich projective line $\PP^{1,\an}_{\Laurent}$. 
Then  by definition  the \emph{residual measure} $(\red_Y)_* \nu$ is
 the atomic measure on $Y_0$ satisfying 
$$
 {(\red_Y)_* \nu} (\{ p \} ) = \nu (\red_Y^{-1}(p))~, 
$$
for any closed point $p\in Y_0$. Note that if $\nu$ gives no mass to $S(Y)$  $(\red_Y)_* \nu$  is the
 push forward of $\nu$ in the usual sense. 
Since the union of the open sets $\red_Y^{-1}(p)$ as $p$ ranges through closed points of $Y_0$ 
is equal to the complement of $S(Y)$ in $\PP^{1,\an}_{\Laurent}$, it follows that   
the  total mass of $(\red_Y)_* \nu$ is equal to 
the   mass of the restriction of $\nu$ to  $\PP^{1,\an}_{\Laurent}\setminus S(Y)$.

\begin{lem}\label{lem:cvg mass}
If  $\nu_n$ converges weakly to $\nu$ and $\nu(S(Y)) =0$, 
then $(\red_Y)_* \nu_n$ converges in mass to $(\red_Y)_* \nu$. 
\end{lem}

\begin{proof}
Pick any closed point $p\in Y_0$. Recall that Radon measures are regular. 
Since $\nu_n \to \nu$  and $\red_Y^{-1}(p)$ is open we deduce that 
$$
\liminf_{n\to\infty}  (\red_Y)_* \nu_n \{ p \} \ge (\red_Y)_* \nu \{ p \}~. 
$$ 
On the other hand, since the boundary of $\red_Y^{-1} \{ p \}$ is included in $S(Y)$,    the assumption  $\nu(S(Y)) =0$
implies that 
$$
\limsup_{n\to\infty}    (\red_Y)_* \nu_n (\{ p \}) \le  \nu \left( \overline{ \red_Y^{-1} \{ p \}} \right) =  (\red_Y)_* \nu (\{ p \}) ~. 
$$ 
Therefore we conclude that for every $p\in Y_0$, 
$(\red_Y)_* \nu_n(\{p \}) \to (\red_Y)_* \nu (\{p\})$. Since all these measures 
are atomic   and of uniformly bounded mass, the result follows. 
\end{proof}

 \subsection{Proof of Theorem~\ref{Thm:atomic}}\label{subs:atomic}
Let $(G, \sm)$ satisfying (A1),  and a representation $\rho : G \to \SL(2, \mm)$,
and suppose that the induced representation $\rho_\na: G \to \SL(2, \Laurent)$ is non-elementary. 
Lemma~\ref{lem:specialization NE} implies that $\rho_t: G \to \SLC$ is also non-elementary for  small $t$    
so   by Theorem  \ref{thm:positive}  (in the complex case)
 it makes sense to talk about the unique probability measure  $\nu_t$ on $\PP^1(\C)$ that is 
 stationary under $\mu_t = (\rho_t)_* \sm$. 
 We also denote by $\nu_\na$ the unique   probability measure on $\PP^{1,\an}_{\Laurent}$ that is stationary under   $\mu_\na = (\rho_\na)_* \sm$.

Fix a model $Y$. Since  $\pi_Y: Y \to \D \times \PP^1_\C$ is a biholomorphism outside the 
central fiber, we can view $\nu_t$ as a probability measure on $Y_t$, which we denote by $\nu_{Y_t}$.
Our aim is to prove that $\nu_{Y_t}$ converges as $t\to 0$ to
an atomic measure on the central fiber  given by $(\mathrm{res}_Y)_* \nu_\na$.

Fix $\e>0$, and choose a model\footnote{Observe that Remark~\ref{rem:useless} applies here 
so that we can further assume the central fiber to be reduced.} $Y'$ such that  the conclusion  of Corollary \ref{cor:cut to small} 
holds 
for  the unique stationary probability measure $\check\nu_{\sf na}$ associated to    
the reversed random walk. We can  assume that 
 $Y'$ dominates $Y$, and write $\nu_{Y'_t}$ for the stationary measure on $Y'_t$. 

Consider
the Markov  operator   $P_t = \int (\gamma_t)_* d\mu_t(\gamma)$  
acting on the space of probability measures on  $Y'_t$, whose $n$-fold iterate is $P_t^n  = \int (\gamma_t)_* d\mu_t^n(\gamma)$.
Observe that for every $n\geq 1$ we have that 
   $P_t^n\nu_{Y'_t}= \nu_{Y'_t}$. To analyze this identity as $t$ tends to 0, we
  extract a   sequence $t_j\to0$ such that  $\nu_{Y_{t_j}}$ and $\nu_{Y'_{t_j}}$ converge
  to respective probability measures $\nu_{Y_0}$ on $Y_0$, and $\nu_{Y'_0}$ on $Y_0'$. Note that by construction the push-forward 
 of $\nu_{Y'_0}$ under  the canonical projection map $Y' \cv Y$ is   $\nu_{Y_0}$.
  We will show that $\nu_{Y_0} = (\mathrm{res}_Y)_* \nu_\na$.

 Let $C = C(Y')$ be the constant given by Proposition \ref{prop:nordsud}, and 
define $A = A(C)= \set{\gamma\in \SL(2, \Laurent), \ \norm{\gamma}\geq C}$. 
 We   define two measurable maps on $A_n$ with values in $Y'_0$, namely
 $\Att_{Y'}(\gamma):= \mathrm{att}(\gamma_{Y'})$ and $\Rep_{Y'}(\gamma):= \mathrm{rep}(\gamma_{Y'})$. 
  The image    $\mu^n_\na|_{A}$ under 
  $\Rep_{Y'}$ (resp.  $\Att_{Y'}(\gamma)$)
  is by definition the distribution of repelling (resp. attracting) points on $Y'$ 
  for the random walk at time $n$.
  
   \begin{lem}\label{lem:distrib rep-att}
  The sequence of atomic measures $(\Rep_{Y'})_*(\mu^n_\na|_{A}) $ (resp.  $(\Att_{Y'})_*(\mu^n_\na|_{A}) $) converges  in mass to $(\mathrm{res}_{Y'})_*\check\nu_\na$ (resp. to $(\mathrm{res}_{Y'})_*\nu_\na$).
 \end{lem}
 
 Taking this lemma for granted for the moment, let us complete the proof of the theorem. 
 Observe first  that for any $\gamma \in A$, 
 from the description of the action of $\gamma$  in Proposition \ref{prop:nordsud}, we see that
  any cluster value of $(\gamma_{t_j})_*\nu_{Y'_{t_j}}$ is of the form 
 $$\big(1- \nu_{Y'_0}\lrpar{\set{\rep(\gamma_{Y'})}}\big) \delta_{\att(\gamma_{Y'})} +  \mathrm{error}
 = 
 \delta_{\att(\gamma_{Y'})} +  \mathrm{error}
 $$
 where the    error on the right hand side is a signed measure of total mass
  $\le 2 \nu_{Y'_0}\lrpar{\set{\rep(\gamma_{Y'})}}$. Using  the identity
 $\nu_{Y'_t} = P_t^n\nu_{Y'_t}$ 
 and  letting $t=t_j\cv0$ we infer that 
\begin{equation}\label{eq:limit stationary}
\nu_{Y'_0}  = \int_{A} \delta_{\att(\gamma_{Y'})} d\mu_\na^n(\gamma) + \mathrm{error} 
\end{equation}
  where the mass of the error is 
\begin{align}
\notag \m(\mathrm{error})  & \leq 2 \int_{A} \nu_{Y'_0}\lrpar{\set{\rep(\gamma_{Y'})}}  d\mu_\na^n(\gamma) + |1 -\mu_\na^n(A)| \\
 \label{eq:error} &= 2 \int_{x\in Y'_0} \nu_{Y'_0}(\{ x \}) d \left( (\Rep_{Y'})_*(\mu_\na^n|_{A})\right)+ |1 -\mu_\na^n(A)| 
 ~. \end{align}
 By Lemma~\ref{lem:distrib rep-att}, $(\Rep_{Y'})_*(\mu_\na^n|_{A})$ converges in mass towards $(\mathrm{res}_{Y'})_*\check\nu_\na$. Since 
 every atom of the latter measure has mass $\le \varepsilon$ by construction, 
 we get that every atom of  $(\Rep_{Y'})_*(\mu^n|_{A})$ has mass $\le 2\varepsilon$
 for $n$ large enough. It follows that the integral in 
 \eqref{eq:error} is bounded by $\le 2 \varepsilon$. Since in addition $\mu^n(A) \to 1$ by Lemma~\ref{lem:contraction}, we conclude that 
 for $n$ large enough
 \begin{equation}
 \nu_{Y'_0}  - \int_{A} \delta_{\att(\gamma_{Y'})} d\mu_\na^n(\gamma) 
 \end{equation}
has total mass $\le 5 \varepsilon$. 

To conclude, we observe   that applying Lemma~\ref{lem:distrib rep-att} again, 
the sequence of measures 
$$ \int_{A} \delta_{\att(\gamma_{Y'})} d\mu_\na^n(\gamma) = \int_{x\in Y'_0} \delta_x d\left((\Att_{Y'})_*(\mu_\na^n|_{A})\right)$$
converges in mass to $(\mathrm{res}_{Y'})_*\nu$. 
We thus infer that $ \nu_{Y'_0}  - (\mathrm{res}_{Y'})_*\nu_\na$ has total mass $\le 5 \varepsilon$. Pushing down this information to $Y$
we get   the same bound for $\nu_{Y_0} - (\mathrm{res}_{Y})_*\nu_\na$ on $Y$. Since   $Y$ does not depend on $\e$ 
and   $\e$ can be made  arbitrarily small, we conclude that $\nu_{Y_0} = (\mathrm{res}_{Y})_*\nu_\na$, as required. \qed

 \begin{proof}[Proof of Lemma \ref{lem:distrib rep-att}]
 Let $(\Omega, \prob) =  (\SL(2,\Laurent)^{\N^*}, \mu_\na^{\N^*})$, and consider the set
 $$\Omega_n = \{ \omega \in \Omega, \, r_n(\omega)\in \SL(2, \Laurent) \text{ is
hyperbolic}\} ~.$$ 
If $\Rep^n_\na : \Omega_n \to \PP^{1,\an}_{\Laurent}$ denotes the measurable map sending
$\omega$ to the repelling fixed point of $r_n(\omega)$, then  by Theorem~\ref{thm:loxodromic}
$\prob(\Omega_n) \to 1$ and
 $(\Rep^n_\na)_* \prob \to \check{\nu}_\na$. 
Pushing forward this convergence   by the residue map $\red_{Y'}$ and applying 
  Lemma~\ref{lem:cvg mass}, we thus get  that 
  $(\red_{Y'})_* (\Rep^n_\na)_* \prob$ converges in mass to  $(\red_{Y'})_*\check{\nu}_\na$. 
 
Now define $\Omega'_n = \Omega_n \cap r_n^{-1}(A)$, which  also satisfies $\prob(\Omega'_n) \to 1$ by Lemma~\ref{lem:contraction}. 
It  follows from  Remark~\ref{rem:more precise} that 
 for   $\omega \in \Omega'_n$,  $\rep_{Y'}(r_n(\omega)) = \red_{Y'}(\Rep_\na^n (\omega))$, in other words, 
the repelling fixed point of $r_n(\omega)$ is mapped under $\red_{Y'}$ to  $\rep_{Y'}(r_n(\omega))$. 
We thus obtain that 
$$ (\red_{Y'})_*(\Rep^n_\na)_* \prob = (\Rep_{Y'})_*(\mu^n_\na|_{A})  + \mathrm{error} $$
where the mass of the error tends to $0$ as $n\to\infty$. This completes the proof.  
 \end{proof}

 \subsection{Proof of Theorem~\ref{Thm:lyapunov na} in the non-elementary case}  \label{subs:asymptotics}
As in the previous section we work with  a representation $\rho : G \to \SL(2, \mm)$  such that the induced representation $\rho_\na: G \to \SL(2, \Laurent)$ is non-elementary, and further assume that 
(A1) and (A2) hold. For $t\in \dd^*$ we set  
 $\mu_t = (\rho_t)_* \sm$  and we also put $\mu_\na = (\rho_\na)_* \sm$. Recall that for $t\neq 0$ 
  the Lyapunov exponent 
  $$\chi(t):= \lim_{n\cv\infty} \unsur{n} \int \log\norm{\gamma} d\mu_t^n (\gamma)$$
  is a well defined positive number, and that the non-Archimedean Lyapunov exponent
   $$\chi_\na := \lim_{n\cv\infty} \unsur{n} \int \log\norm{\gamma}_\na d\mu_\na^n(\gamma)$$
   is also well-defined and positive  by  Theorem \ref{thm:positive}. 
   
We have to show that 
 \begin{equation} \label{eq:lyapunov na2}
   \unsur{ {\log\abs{t}\inv}}\chi(t)   \longrightarrow  \chi_{\sf na} \text{ as } t\cv 0.
  \end{equation}

 If $(K, \abs{\cdot})$ is any metrized field, and $z =[z_1:z_2]\in \PP^1(K)$,  recall the notation 
   $$\sigma(\gamma, z) = \log \frac{\norm{\gamma Z}}{\norm{Z}}~,$$ where 
 $Z=(z_1, z_2)\in K^2\setminus\set{0}$ and $\norm{Z} = \max(\abs{z_1}, \abs{z_2})$. 
 To establish  \eqref{eq:lyapunov na2}, we first  need to relate the classical and 
non-Archimedean   expansion rates for a single group element. 
 We start with the following consequence of Lemma~\ref{lem:uniform disk}.

 \begin{lem}\label{lem:sup2}
There exists a constant $C>0$, such that for any model $Y$,  any $g\in G$, and 
 any point  $y_{t}\in Y_{t}$ with $|t|\le 1/2$, we have 
 \begin{equation}\label{eq:sup2}
 \abs{\frac{1}{\log|t|^{-1}} \sigma(\gamma_t, y)}
 \leq  
 \log \norm{\gamma}_\na + \frac{C \length (g)}{\log\abs{t}\inv}
~,\end{equation}
 where  $\gamma = \rho(g) \in \SL(2,\mm)$.
  \end{lem}
 
 \begin{proof}
 As before for $t\neq 0$ we can naturally identify the fibers $Y_t$ and $\set{t}\times \pu$, and we write
 $y_t = [z_{1t}:z_{2t}]$,  with $Z_t =   (z_{1t}, z_{2t})$ and $\max\{ |z_{1t}|, |z_{2t}|\} =1$. 
The upper bound $$\sigma(\gamma_t, y_{t}) = \log\norm{\gamma_t Z_t} \leq  \log\norm{\gamma_t} \le 
(\log\abs{t}\inv)  \log\norm{\gamma}_{\sf na} + C \length (g)$$ 
 follows directly from the definitions and Lemma \ref{lem:uniform disk}. To get the lower bound it is enough to write
$ \norm{\gamma_t Z_t} \geq \norm{\gamma_t\inv}\inv \norm{Z_t}$ and remind that $\norm{\gamma_t} = 
\norm{\gamma_t\inv}$.
 \end{proof}

 The main step of the proof  is the following proposition.

 \begin{prop}\label{prop:hybrid}
 For every model $Y$, there exists a constant $c(Y)>0$ satisfying the following property.
 For every 
 $\gamma \in \SL(2, \mm)$, 
 there exists a   point $\alpha(\gamma)\in Y$ such that if  $t_j\cv0$ and 
  $(y_{t_j})\in Y_{t_j}$ is any sequence not accumulating $\alpha(\gamma)$, then 
  \begin{equation}\label{eq:norm}
  \liminf_{j\cv\infty}  \unsur{\log\abs{t_j}\inv} \sigma(\gamma_{t_j} , y_{t_j})  \geq  \log\norm{\gamma}_{\sf na}  - c(Y).
  \end{equation}
 \end{prop}

This says that the (positive) upper bound   that was obtained in \eqref{eq:sup2}  is almost achieved everywhere on $Y$
when $t\to0$, except at one point,  up to an error that is uniform in $\gamma$ (compare Lemma \ref{lem:expand vector 1}).

\medskip

Let us postpone the proof of the proposition to the end of the section, and first complete the proof of~\eqref{eq:lyapunov na2}.

 Fix $\e>0$. Apply Corollary~\ref{cor:cut to small}  to get 
   a model $Y$ in which all the atoms of the residual measure $\nu_{Y_0}= (\red_Y)_* \nu_\na$ 
   are smaller than $\e$. 
 By Theorem \ref{Thm:atomic}, 
  we have that $\nu_{Y_t}\cv \nu_{Y_0}$ as $t\cv 0$ where $\nu_{Y_t}$ is the pull-back of $\nu_t$ to $Y$.    

We first work with a fixed $g\in G$, and as usual we write $\gamma = \rho(g) \in \SL(2,\mm)$.   By Lemma~\ref{lem:sup2}, we have that 
 \begin{equation}\label{eq:666}
 \limsup_{t\cv0}  \unsur{\log\abs{t }\inv} \int \sigma(\gamma_t, y) d\nu_t(y)  \leq  \log \norm{\gamma}_{\sf na}.
 \end{equation}
To obtain a lower bound, we fix a small neighborhood $U$ of $\alpha(\gamma)$ in $Y$ such that 
$\nu_t(U) \le 2 \e$ for any $t$. This is  
 possible because $\nu_t \to \nu_{Y_0}$ and $\nu_{Y_0}(\alpha(\gamma)) \le \e$. 
Then Proposition~\ref{prop:hybrid} shows that if  $\eta \ll 1$ is fixed, then for every   
 small enough $t$, for  $y\in Y_t\setminus U$ we get 
$$
\frac{1}{\log|t|^{-1}} \sigma(\gamma_t, y)
\ge \log \norm{\gamma}_\na - c(Y) - \eta,$$
whereas for $y\in Y_t \cap U$, Lemma~\ref{lem:sup2} implies that 
$$
\frac{1}{\log|t|^{-1}} \sigma(\gamma_t, y)
\ge - \log \norm{\gamma}_\na -\eta.
$$
Combining these two estimates we infer that 
 $$  \unsur{\log\abs{t }\inv} \int_{Y_t} \sigma(\gamma_t, y) d\nu_t(y) \geq
(1-2\e) (\log \norm{\gamma}_\na - c(Y) - \eta) - 2\e (\log \norm{\gamma}_\na +\eta ) 
,$$
therefore since $\eta$ is arbitrary, 
$$ \liminf_{t\cv0} 
 \unsur{\log\abs{t }\inv} \int_{Y_t}  \sigma(\gamma_t, y) d\nu_t(y) \geq 
 \log \norm{\gamma}_\na - (1-2\e)c(Y) - 2\e \log \norm{\gamma}_\na .$$
Using this  inequality and~\eqref{eq:666} we finally obtain  
 \begin{equation}\label{eq:limsupinf}
 \limsup_{t\cv0} 
 \abs{ \unsur{\log\abs{t }\inv} \int_{Y_t} \sigma(\gamma_t, y) d\nu_t(y) -  \log\norm{\gamma}_{\sf na}}   \leq 2c(Y) + 4\e
  \log \norm{\gamma}_{\sf na}.
  \end{equation}
  
  To conclude the argument we   integrate  this estimate  with respect to $g$. 
  Fix an  integer $n$ so large that 
 $$\frac{2c(Y)}{n} <\e \text{ and } \abs{\unsur{n} \int \log \norm{\gamma}_{\sf na} d\mu_\na^n(\gamma) - \chi_{\sf na}} <\e.$$ 
 We observe that  the  Furstenberg formula for the Lyapunov exponent iterated $n$ times and read  
  in the model $Y$ 
expresses as 
$$
\chi(t) =  \unsur{n} \int_{Y_t \times \Gamma}   {\sigma(\gamma_t, y)}  d\nu_t(y)  d\mu_t^n(\gamma), 
$$
 so we can write
 \begin{align*}
  \abs{ \unsur{\log\abs{t }\inv} \chi(t)   - \chi_{\sf na} }
&=    \abs{\unsur{n}     \int_{Y_t \times \Gamma_t}  \frac{\sigma(\gamma_t, y)}{{\log\abs{t }\inv}} d\nu_t(y)  d\mu_t^n(\gamma) - \chi_{\sf na} }\\
 &\leq    \abs{\unsur{n}     \int_{Y_t \times \Gamma_t} \frac{\sigma(\gamma_t, y)}{{\log\abs{t }\inv}} d\nu_t(y)  d\mu_t^n(\gamma)  - \unsur{n} \int \log \norm{\gamma}_{\sf na} d\mu_\na^n(\gamma) }+ \e \\
 &=:  \Delta(t) + \e~.
 \end{align*}
By the moment  assumption (A2), there exists 
  a finite subset $G' \subset G$ such that 
\begin{equation}\label{eq:G'}
  \int_{G\setminus G'} \length(g) d\sm^n (g) \le \e \text{ and }
 \frac1n \int_{\Gamma_\na \setminus \Gamma'_\na} \log\norm{\gamma}_\na d\mu_\na^n(\gamma) \le \e 
  \end{equation}
  where $\Gamma'_\na := \rho_\na(G')$. 
To bound the quantity $\Delta(t)$ 
we split the integrals according to the decomposition $G = G'\cup G\setminus G'$.
If $\Delta'(t)$ denotes the contribution coming from  $G'$, using \eqref{eq:G'} and
 Lemma~\ref{lem:sup2}, we get
\begin{align*}
\Delta(t) & \le \Delta'(t)
 + 2  \int_{\Gamma_\na \setminus \Gamma'_\na}\frac1n \log\norm{\gamma}_\na d\mu_\na^n(\gamma) + \frac{C}{\log \abs{ t}\inv }
   \int_{G\setminus G'} \length(g) d\sm^n (g) \\
   &  \le \Delta'(t) + (C+2) \e~,
\end{align*}
From~\eqref{eq:limsupinf} we have that 
\begin{align*}
\limsup_{t\to 0}\Delta'(t)  &= 
\limsup_{t\to 0}\abs{\unsur{n}     \int_{Y_t \times \Gamma'_t}  \frac{\sigma(\gamma_t, y)}{{\log\abs{t }\inv}} d\nu_t(y)  d\mu_t^n(\gamma)   -\int_{\Gamma' _\na}  \unsur{n} \log\norm{\gamma}_\na d\mu_\na^n(\gamma) } \\
&\leq 
 \unsur{n} \int_{G'}  \limsup_{t\to 0} \abs{ \int_{Y_t}  \frac{\sigma(\rho_t(g), y)}{{\log\abs{t }\inv}} d\nu_t(y) 
 -  \log\norm{\rho_\na(g)} } d\sm^n(g)
 \\
& \leq 
 \frac{2c(Y)}n + \frac{ 4\e}{n}   \int_{\Gamma'}  \log \norm{\gamma}_\na d\mu_\na^n(\gamma)\leq \e + 4 \e (\chi_\na + \e)~,
\end{align*} 
where in the second line we use $\mu_t = (\rho_t)_*\sm$ and $\mu_\na = (\rho_\na)_*\sm$.
Finally we conclude that 
$$\limsup_{t\cv0}  \abs{ \unsur{\log\abs{t }\inv} \chi(t)  - \chi_{\sf na} } \leq  \e + (C+2)\e + \e + 4 \e (\chi_\na + \e)$$
Since this estimate  makes no reference to the  model $Y$  and  $\e$ is arbitrary, the theorem follows.

 \begin{proof}[Proof of Proposition \ref{prop:hybrid}]
 We first need a version of Proposition \ref{prop:KAK} for $\SL(2, \mm)$. Indeed 
since $\mm$ is neither a field, nor complete this proposition cannot be applied directly.   Let us explain how to adapt the argument 
to this concrete setting. We first introduce  some notation: for $r>0$ denote by $\mathcal{O}_r$ (resp. $\mm_r$)
the ring of holomorphic functions in $D(0, r)$ (resp. of holomorphic functions in $D(0, r)\setminus {0}$ admitting a meromorphic 
extension at the origin). 
 
 \begin{lem}\label{lem:KAK holomorphe}
 For every $\gamma\in \SL(2, \mm)$ there exist $r>0$,  
 $m,n \in \SL(2,\mathcal{O}_r)$ and $a \in \SL(2, \mm_r)$  diagonal such that
 $\gamma= m\cdot a \cdot n$ and $\norm{a}_\na  = \norm{\gamma}_\na$.
 \end{lem}

\begin{proof}
 Observe first that a meromorphic family of matrices $(\gamma_t)$ in $\SL(2, \mm)$ 
 extends holomorphically at the origin if and only if for any triple $\set{a,b,c}$ of distinct points in $\pu(\cc)$, 
 then as  $t\to 0$, 
 there exists distinct $a',b',c'$ such that 
 $\gamma_t(a)\to a'$, $\gamma_t(b)\to b'$, and $\gamma_t(c)\to c'$.
 
 Let now $\gamma\in \SL(2, \mm)$ and assume that $\gamma\notin\SL(2, \mathcal{O})$. Then on $X = \dd\times \PP^1_\C$, 
 $\gamma$ contracts $X_0\setminus\set{\rep(\gamma_X)}$ to $\set{\att(\gamma_X)}$. Pick $m\in  \SL(2, \mathcal{O})$ such 
 that $m\inv(\att(\gamma_X)) = \infty$ and $n\in  \SL(2, \mathcal{O})$ such 
 that $n(\rep(\gamma_X)) = \infty$. Then $\gamma' = m\inv\gamma n\inv$ maps $X_0\setminus\set{0}$ to $\infty$, which implies that for small $t$, $\gamma_t'$ is loxodromic with an attracting fixed point $\att(\gamma_t)$ 
 close to $\infty$ and a repelling fixed point $\rep(\gamma_t)$ close to 0. Thus there exists 
 $r>0$ and   $h\in \SL(2, \mathcal{O}_r)$ such that 
 $h_t(\att(\gamma_t)) =\infty$, $h_t(\rep(\gamma_t)) = 0$ and $h_t(1) =1$. Then $a= h\gamma'h^{-1}$ fixes 0 and $\infty$, so it is diagonal.
 
 By the first  observation, $t\mapsto h_t$ extends holomorphically
 at the origin, that is $h\in \SL(2, \mathcal{O}_r)$. So the desired decomposition is $\gamma = (h\inv m) a (n h)$. The equality 
  $\norm{a}_\na  = \norm{\gamma}_\na$ follows easily.
\end{proof}

We are now ready to  prove  Proposition \ref{prop:hybrid}. We start by working on  $X$. Pick a sequence of points $(x_{t_j})$ converging to the central fiber, and  consider the 
 quantities $\sigma(\gamma_{t_j}, x_{t_j}) = \norm{\gamma_{t_j}X_{t_j}}$. Extract so that $(x_{t_j})$ converges and drop the 
index $j$ for notational simplicity. 
 If $m\in \SL(2, \mathcal{O})$ then
  for every $Z\in \mathcal O^2$,  $\norm{ m_t Z_t}\asymp \norm{Z_t}$ so by the 
  previous lemma we can assume that $\gamma$ is diagonal,  
 $\gamma_t  =\mathrm{diag}( \lambda_t, \lambda_t\inv)$. If $(x_{t})$ does not converge to $[0:1]$, then $\norm{\gamma_{t}X_{t}} \asymp \norm{\gamma_{t}}$ so the desired estimate holds, therefore the interesting case is when 
 $(x_{t})$ converges to $[0:1]$. In this case a lift of norm 1 of $x_t$ will be  of the form $X_{t} = (\xi_t, \eta_t)$ with $\abs{\eta_t} = 1$, so 
 $\sigma (\gamma_t, x_t) = \max(\abs{\la_t\xi_t}, \abs{\la_t}\inv)$. From this formula we infer that if for some $l>0$, 
   $\abs{\xi_t} \geq \abs{t}^l$ when $t\cv 0$ then  
 $$\liminf_{t\to 0} \unsur{\log\abs{t}\inv} \sigma (\gamma_t, x_t) \geq  \abs{\la}_\na -l   = \norm{\gamma}_\na -l.  $$

We rely on the following  elementary geometric fact. 
 \begin{lem}\label{lem:geometric}
Let $\pi:M\cv \dd^2$ be  a composition of $N$ blow-ups above the vertical fiber 
$\set{0}\times \dd$ in the unit bidisk, and denote $M_0 = \pi\inv(\set{0}\times \dd)$. 
Then if $\ell>N$, the open set $$\pi\inv\lrpar{\set{(t,x)\in \dd^*\times \dd, \ \abs{x}<\abs{t}^\ell}}$$
clusters at a unique point of $M_0$. 
 \end{lem}
 \begin{proof}
 The open set $\set{(t,x)\in \dd^*\times \dd, \ \abs{x}<\abs{t}^\ell}$ is the union of the curves $\set{ x= c t^\ell}$ in 
 $\dd^*\times \dd$,  where $c$ ranges over $\abs{c}<1$. These curves 
get separated after exactly $\ell$ blow-ups.  
 \end{proof}
 
In order to  conclude the proof, pick a model $\pi: Y\to X$ 
 and a sequence $(y_{t_j})$
 as in the statement of the proposition.  Extract so that $(y_{t_j})$ converges.  
 Let $N$ be the number of 
 blow-ups required to obtain $Y$. 
 We put $x_{t_j} = \pi(y_{t_j})$ and do the analysis of the first part of the proof.
Then, Lemma \ref{lem:geometric} applied to $l = N+1$  provides a point $\alpha = \alpha(\gamma)$ in the central fiber 
$Y_0$ such that if $(y_{t_j})$ does not converge to $\alpha$, then 
$$\liminf_{j\cv\infty}  \unsur{\log\abs{t_j}\inv} \sigma (\gamma_{t_j}, y_{t_j}) \geq  \norm{\gamma}_\na -(N+1).$$
  The result follows. 
  \end{proof}

%%%%%%%%%%%%%%%%%%%%%%%%%%%%%%%%%%%%%%

 \section{Degenerations: elementary representations}\label{sec:elem}

 In this section we complete the proof of Theorem \ref{Thm:lyapunov na} by addressing the case of elementary representations. 
 Let as before   $G$ be a finitely generated group endowed with some probability measure  $\sm$  satisfying 
\begin{itemize}
 \item[(A2${}^+$)] there   exists $\delta>0$ such that $\int (\mathrm{length}(g))^{1+\delta} d\sm(g)<\infty$, 
 \end{itemize}
and let $\rho :G\to   \SL(2, \mm)$ be any meromorphic family of representations. 
  
With notation as in \S\ref{subs:atomic},  
Viewing  $\mm$ as a subring of $\Laurent$    we denote by $\rho_\na$
 the corresponding non-Archimedean representation  $G\to   \SL(2, \Laurent)$ and $\mu_\na = (\rho_\na)_* \sm$, which 
 satisfies the moment condition 
   \begin{itemize}
 \item[(B2${}^+$)] there   exists $\delta>0$ such that $\int \log \|\gamma\|^{1+\delta} d\mu(\gamma_\na)<\infty$
 \end{itemize}
in $\SL(2, \Laurent)$.  In particular  the non-Archimedean Lyapunov exponent 
$\chi_\na =  \chi(\mu_\na)$ is well-defined.  Likewise for  $t\in \D^*$ we let $\mu_t = (\rho_t)_* \mu$ and 
$\chi(t) =     \chi(\rho_t(G), \mu_t)$.

\begin{thm} \label{thm:degeneration elementary}
 Let $(G, \sm)$ be a finitely generated group endowed with a probability measure satisfying 
  (A2${}^+$), and let 
 $\rho: G\to \SL(2, \mm)$ be such that $ \rho_\na(G) \subset \SL(2,\Laurent)$ is elementary. 
 Then 
 \begin{equation} \label{eq:lyapunov elementary}
\chi(t)   =    \big({\log\abs{t}\inv} \big) \chi_{\sf na}   + O(1) \text{ as } t\cv 0.   \end{equation}
  If in addition  $\mu$ is symmetric,  then   $\chi_{\sf na} = 0$. 
    \end{thm}
 
 Under mild assumptions, the error term can be understood more precisely, see \S \ref{subs:continuity elementary}

Put $\Gamma_\na = \rho_\na(G)$. 
 According to the discussion in  \S\ref{subs:elementary}, 
   if $\Gamma_\na \leq  \SL(2, \Laurent)$ is elementary then it is either non-proximal or non-strongly irreducible, so  
   there are three possibilities:
   \begin{enumerate}
   \item  $\Gamma_\na $ has potential good reduction; 
   \item  $\Gamma_\na $ is conjugate to a subgroup of the affine group $\set{z\mapsto az+b, a\in \Laurent^\times, b\in  \Laurent}$;
   \item  $\Gamma_\na $ is conjugate to a subgroup of the group of transformations fixing   $\set{0, \infty}$, that is, 
   $\set{z\mapsto \lambda z^{\pm1}, \la\in \Laurent^\times}$. 
   \end{enumerate}

Note that if we are not in case (1), then the projection of $\Gamma_\na$ in $\PGL(2, \Laurent)$ is not purely elliptic, and it follows
from the analysis of \S\ref{subs:elementary} that the conjugacy in (2) and (3) lies in $\SL(2, \Laurent)$ (i.e. no  
 field extension is  required). 

In the remaining part of this section we split the proof of Theorem \ref{thm:degeneration elementary} according to these three cases.

\subsection{Potential good reduction}
In  case (1),    $\Gamma_\na $  is conjugate in $\SL(2, \C(\!(t^{1/2})\!))$ to a representation fixing the Gau{\ss} point. 
Lifting to  a branched 2-cover (which amounts to making the change of variables $t= u^2$), we can assume that 
the conjugacy lies in  $\SL(2, \Laurent)$, that is 
 there exists $\alpha\in \SL(2, \Laurent)$ such that for every $\gamma\in \Gamma_\na$, 
 $\norm{\alpha\inv\gamma\alpha}\leq 1$. 
 
 Observe first that $\mm$ is dense in $\Laurent$ so that there exists a sequence $\alpha_n \in \SL(2,\mm)$ 
 such that $\norm{\alpha - \alpha_n} \to0$. From the continuity of the matrix product, and the ultrametric property, for any sufficiently large  
integer $n$ we get
 $\norm{\alpha_n\inv\rho_{\na}(s)\alpha_n}\leq 1$ for all $s$ in a fixed finite set of generators of $G$. 
In particular, we have
  $\norm{\alpha_n\inv\rho_{\na}(g)\alpha_n}\leq 1$ for all $g \in G$ so that we may suppose that our original conjugacy $\alpha$ 
  belongs to $\SL(2,\mm)$. 
 
  Since the Lyapunov exponent is insensitive to conjugacy, by  replacing $\rho$ by 
  $\alpha\inv\rho(\cdot)\alpha$ we can assume that $\rho$ extends holomorphically at the origin. 
  For every $t\neq 0$, by sub-additivity we have  
  the bound $0\leq \chi(t) \leq \int \log \norm{\rho_t(g)} d\sm(g)$. Therefore applying 
  Lemma \ref{lem:uniform disk} and the moment condition  we infer that $\chi(t) = O(1)$ as $t\to 0$. 
On the other hand, since  $\Gamma_\na$ has good reduction, 
  $\chi_\na$ vanishes, and we are done.

\subsection{Affine representations}
Let $({k},\abs{\cdot})$ be any complete valued field  
and  consider any  subgroup  $\Gamma$ of $\SLk$ endowed with a measure $\mu$, such that the projection 
 of $\Gamma$ in $\PGLk$ lies in the affine group  $\mathrm{Aff}(k)$.   An element $\gamma\in \Gamma$ can be written in matrix form 
  as 
  \begin{equation}\label{eq:matrix}
  \gamma = \lrpar{ \begin{matrix} \alpha & \beta \\ 0 &\alpha\inv\end{matrix}} \ ,
\end{equation}   corresponding to the M\"obius transformation $\gamma(z) = az+b$, with $a= \alpha^2$ and $b = \beta\alpha$. 
  Thus its norm 
  is  
\begin{equation}
\label{eq:norm affine}
\norm{\gamma} =
\max\lrpar{ \abs{\alpha}, \abs{\beta}, \abs{\alpha\inv} }=  \max\lrpar{\abs{a}^{1/2}, \abs{a}^{ -1/2}, \abs{ba^{-1/2}}  }~.
\end{equation}

\begin{prop}\label{prop:formula affine}
Let $(k,\abs{\cdot})$ be a complete valued field and  let  $\mu$ be a measure with countable support in $\SLk$, contained in the 
 affine group, and satisfying  (B2${}^+$).
 Then with notation as above we have 
  \begin{equation}\label{eq:formula affine}
\chi(\mu) = 
\abs{\int \log |\alpha (\gamma)| \, d\mu(\gamma)} = \frac12 \abs{\int \log |a(\gamma)|\, d\mu(\gamma)}
\end{equation}
In particular if $\mu$ is symmetric, $\chi(\mu) = 0$. 
 \end{prop}
 \begin{proof}
 For any $\omega  = (\gamma_n) \in \Omega$, we write  $\gamma_n(z) = a_n(\omega) z+b_n(\omega)$
so that 
 $$\ell_n(\omega) = \gamma_n \cdots \gamma_1 (z)    = A_n(\omega) z + B_n(\omega) = a_n\cdots a_1 z + \sum_{j=1}^{n-1}  a_n\cdots a_{j+1} b_j +  b_n
 ~.$$
By the law of large numbers (or equivalently the Birkhoff  ergodic theorem) we have that 
\begin{equation}\label{eq:large numbers}
\frac1n \log| A_n| \to  \lambda :=
 \int \log |a| d\mu~ \text{ a.s.}
 \end{equation}
Fix   $\varepsilon >0$. For a.e. $\omega$, $ e^{n(\lambda - \varepsilon)} \le |A_n(\omega) |\le e^{n(\lambda + \varepsilon)}$ for large $n$. 
The moment condition (B2${}^+$) and  
Chebyshev's inequality yield  $\mu \{  \abs{b}> e^{\e j} \} \le C j^{-1-\delta}$, so that 
by the  Borel-Cantelli lemma we get that
$b_n(\omega) \le e^{\varepsilon n}$ a.s.  for   large $n$.

At this point we split the proof into two cases according to the sign of $\lambda= \int \log |a| d\mu$. 
Write $$\gamma_n\circ \cdots \circ \gamma_1 (z) = a_1\cdots a_n \lrpar{z+  \sum_{j=1}^{n} \frac{b_j}{a_1\cdots a_j}}
 = A_n\lrpar{z+ \sum_{j=1}^n b_jA_j\inv}.$$ If $\lambda>0$ we infer  from \eqref{eq:large numbers} that 
 a.s. the partial sums of the  series $\sum_{j\geq 1} b_j A_j^{-1}$ are    bounded, from which 
it follows that  $\abs{B_n} = O(\abs{A_n})$.  Therefore 
$$\lim_{n\to \infty}  \unsur{n} \log \abs{A_n} = \frac{\lambda}{2} \,\text{ and } \limsup_{n\to \infty} 
\lrpar{\frac1{2n} \log |B_n| - \frac1{2n} \log |A_n| }\leq \frac{\lambda}{2}~. $$ 
By~\eqref{eq:norm affine}, we have   
\begin{equation}\label{eq:norm affine2}
\frac1n \log \| \ell_n(\omega)\| = \max \left\{ \frac1{2n} \log |A_n| , - \frac1{2n} \log |A_n|, \frac1{n} \log |B_n| - \frac1{2n} \log |A_n|\right\} ~,
\end{equation} so we conclude  that  $\chi(\mu) =    {\lambda}/{2}$. 

On the other hand, if $\lambda\leq 0$, then since almost surely for large $j$,
$\abs{b_j}\leq e^{\e j}$ and $e^{(\lambda-\e) j}\leq \abs{A_j}\leq e^{(\lambda+\e) j}$ 
 we deduce that 
$$\abs{ \sum_{j=1}^n b_jA_j\inv} = O \lrpar{e^{(-\lambda +2\e)n}} \text{  hence } \abs{B_n} = O\lrpar{ e^{3\e n}}.$$
Thus from  \eqref{eq:norm affine2}   
we get that for large $n$, 
$$ -\frac{\lambda}{2} -\e \leq \frac1n \log \| \ell_n(\omega)\| \leq -\frac{\lambda}{2} + 4\e ~,$$  and 
$\chi(\mu) = -   {\lambda}/{2}$, as required. 
 \end{proof}
    
 \begin{proof}[Proof of Theorem \ref{thm:degeneration elementary} in the affine case]
Under the assumptions of the theorem, assume that  the projection  of $\Gamma_\na$ into $\PGL(2,\Laurent)$ 
lies in the affine  group.  For any $g\in G$, 
 we use the same notation as above,
  writing $\alpha(\rho(g))\in \mm$ for the upper diagonal term of $\rho(g)$, and $\beta(\rho(g))\in \mm$ for its upper right term.  
We get corresponding coefficients $\alpha(\rho_t(g))\in \cc^*$  (for fixed $t\neq 0$)
and $\alpha(\rho_\na(g))\in \Laurent^*$.

By Lemma~\ref{lem:uniform disk} we have that 
$\alpha(\rho_t(g)) = t^{-\log \abs{\alpha(\rho_\na(g))}} \widetilde\alpha(\rho_t(g))$, 
where $ t\mapsto  \widetilde\alpha(\rho_t(g))$ is   holomorphic   in $\dd$,
and $\abs{\log\norm{\widetilde \alpha }_{L^\infty(\overline{\D}(0, 1/2))} }\leq C(\rho) \length(g)$.
 Hence
\begin{equation}\label{eq:affine error}
 \int \log|\alpha(\rho_t(g))| \, d\sm(g) = (\log|t|^{-1}) \, \int \log \abs{\alpha(\rho_\na(g))} \, d\sm  +  \mathcal{E}(t)
\end{equation}
 where
\begin{equation}\label{eq:affine error2}
|\mathcal{E}(t)| =  \left| \int \log \abs{\widetilde\alpha(\rho_t(g))} \, d\sm \right|
 \le C \int \length(g) d\sm <+\infty~.
\end{equation}
Therefore applying the formula of  Proposition~\ref{prop:formula affine}   to $k = \cc$ and $k= \Laurent$
 we infer the desired estimate \eqref{eq:lyapunov elementary} in the affine case. 
 \end{proof}

  \subsection{Representations fixing $\set{0, \infty}$}
  Back to the general setting, 
  consider now a subgroup $\Gamma\leq \SLk$ whose projection in $\PGL(2,k)$ fixes $\set{0, \infty}$. Then every matrix
  in $\Gamma$ is     of the form 
  $$\text{either }\gamma = \begin{pmatrix} \alpha & 0 \\ 0 &\alpha\inv
  \end{pmatrix} \text{ or  } 
\gamma = \begin{pmatrix} 0 & -\alpha \\ \alpha \inv &0
  \end{pmatrix}~,$$
  and $\| \gamma \| = \max \{|\alpha|,  |\alpha|^{-1}\}$. As M\"obius transformations, we have
  $\gamma(z) = a z$ or $-a/z$ with $a = \alpha^2$, and $\| \gamma\| =  \max \{|a|^{1/2},  |a|^{-1/2}\}$.

  \begin{prop}\label{prop:formula 0infty}
Let $(k,\abs{\cdot})$ be a complete valued field and  let  $\mu$ be a measure with countable support in $\SLk$,  satisfying 
the moment condition (B2). Suppose that any element in the support of $\mu$ leaves  the pair $\{0,\infty\}$ invariant and at least one element
permutes $0$ and $\infty$. 

Then  $\chi(\mu) = 0$.
 \end{prop}

The last case of Theorem \ref{thm:degeneration elementary} immediately follows, since in this case  
we have that $\chi(t) \equiv 0 =  \chi_\na $. 

\begin{proof} 
It is more convenient here to use  probabilistic language. 
We denote by $\sE(\cdot)$
the expectation of a random variable. 

\smallskip

In terms of M\"obius transformations, we are considering a random composition of 
maps of the form 
$\gamma_j(z) = \la_j z^{\e_j}$ where $\lambda_j\in k^\times$ and $\e_j \in \{\pm1\}$ are iid random variables.

Write $\ell_n(\omega) = (\gamma_n \circ \cdots \circ \gamma_1) = \Lambda_n z^{\mathcal{E}_n}$, and let $X_n = \log |\Lambda_n|$
and  $x_n = \log|\lambda_n|$. A simple computation shows that 
$\mathcal{E}_n = \prod_{i=1}^{n} \e_i$
 and 
\begin{equation}\label{eq:Xn}
X_n = x_n + \e_n x_{n-1} + \e_n \e_{n-1} x_{n-2} + \ldots + \e_n\cdots \e_1 x_1
\end{equation}
Note that   $(x_n)$  is a sequence of iid real random variables with $\sE(\abs{x_1})<\infty$. 
Kingman's theorem implies that the sequence   $(X_n/n)$ converges a.s. We have to show that its limit is 0. 

Let $(n_l)_{l\ge0}$ be the increasing sequence of random times where $\e_{n_l} = -1$, that is, $(n_l)$ is
 defined by  $n_0=0$ and $n_{l+1} = \min\set{j> n_l, \ \e_j = -1}$.
Since the $\e_n$ are iid and $\mu$ 
 does not give full mass to the affine group,  $(n_{l+1} - n_l)_{l\geq 0}$ is a sequence of  iid
 random variables with a  geometric distribution of non-zero parameter $p>0$ 
 (which is the probability that $\gamma$ is not affine). 
 
 For $q\geq 1$ put  $Y_q = \sum_{j=n_{q-1}}^{n_{q}-1} x_j$ with the convention that $x_0=0$. Observe that $(Y_q)$ forms a sequence of iid random variables
 with finite first moment, and such that 
 $$\sE( {Y_1}) =  \sE({x_1})+ \sum_{j=1}^{\infty} j(1-p)^j p  \sE({x_1}) = \unsur{p}  \sE( {x_1})~.$$ 
 It follows from \eqref{eq:Xn} that for every $l\geq 0$, $X_{n_l-1} = \sum_{j=1}^l (-1)^{l-j} Y_j$. 
 
 Finally, let $(Z_{l})_{l\geq 1} =( Y_{2l-1}-Y_{2l})_{l\geq 1} $ 
which is a  sequence of iid random variables with $\sE( \abs{Z_1})<\infty$ and $\sE( {Z_1})=0$. Up to sign we have 
that 
$$ X_{n_{2l}-1} = \pm \sum_{j=1}^l Z_j ,$$
 thus from the strong law of large numbers we infer  that $\unsur{l} X_{n_{2l}-1} \to 0$ a.s. 
 as $l\to \infty$,
  hence since $l\le n_l$ the same holds for $\unsur{n_{2l}-1} X_{n_{2l}-1}$. The proof is  complete.
\end{proof}

 \subsection{Continuity of the error term}\label{subs:continuity elementary}
  If $\sm$ is finitely supported, the proof of Theorem \ref{thm:degeneration elementary} actually yields  
 a finer  estimate   in \eqref{eq:lyapunov elementary} of the form
 $$\chi(t)   =    \big({\log\abs{t}\inv} \big) \chi_{\sf na}   + C + o(1) \text{ as } t\cv 0.   $$
 Indeed:
 \begin{itemize}
 \item if $\Gamma_\na$ has potential good reduction, the proof reduces the situation to that of a holomorphic family of representations, in which case the result follows from the Furstenberg theory when $\rho_0$ is non elementary in 
 $\SLC$
 and from Bocker-Viana \cite{bocker viana} when $\rho_0$ is elementary (the finiteness assumption on $\sm$ is used here);
 \item if $\Gamma_\na$ is affine  we have to show that the error $\mathcal{E}(t)$  in \eqref{eq:affine error} 
 admits a limit when $t\to0$, which by virtue of \eqref{eq:affine error2} and Lemma \ref{lem:uniform disk} 
 follows from the dominated convergence theorem;
 \item finally in the case of representations fixing $\set{0, \infty}$ there is nothing to prove because 
 $\chi(t)\equiv 0$. 
  \end{itemize}
 
  %%%%%%%%%%%%%%%%%%%%%%%%%%%%%%%%%%%%%%
  
\section{Degenerations: the hybrid approach}\label{sec:hybrid}

We propose an alternative approach to the analysis of the blow-up of the Lyapunov exponent, which is based on the hybrid space constructed by Berkovich and used 
by Boucksom and Jonsson in~\cite{boucksom jonsson} and by the first author in~\cite{degeneration}. The introduction of this space allows   us
to make sense of the convergence of measures $\nu_t \to \nu_{\na}$ and leads to a proof of Theorem~\ref{Thm:hybrid}.

%%%%%%%%%%%%%%%%%%%%%%%%%%%%%%%%%%%%%%

\subsection{The hybrid space}\label{sec:hybrid detail}
We start by briefly recalling the definition the hybrid space, referring to~\cite{boucksom jonsson,degeneration} for more details. 

Let  $\cA$ be the subring of $\Laurent$  consisting of those series $f$ such that  $\|f\|_{\hyb}  < + \infty$, where
$$\|f\|_{\hyb}:= \sum_{n= - \infty}^{+\infty} |a_n|_{\hyb}\, e^{-n} ~,  \text{ and } \begin{cases}
   |a|_{\hyb} = \max\{ |a|, 1\} \text{ if }a\in \C^* \\   |0|_{\hyb} =0 \end{cases}.$$  
  Observe that for any $f\in \cA$,  the sum has only finitely many negative terms and the series defining $f$ converges in 
 $\overline{\dd}^*_{1/e}$.
Endowed with the hybrid norm $\|\cdot\|_{\hyb}$, $\cA$  is a Banach ring, and its Berkovich spectrum 
$\cD:= M_{\mathrm{ber}}(\cA)$ is defined as usual to be 
the space of multiplicative semi-norms $|\cdot|$ on $\cA$ such that $|\cdot| \le \|\cdot\|_{\hyb}$, 
endowed with the topology of   pointwise convergence. 

It turns out that $\cD$ is naturally a closed disk. To see this, introduce  the map $\tau$
 from the closed disk of radius $1/e$ to $\cD$ by the formula:
\begin{equation} \label{eqn:def hyb}
\begin{cases}
|f(\tau(0))| = e^{-\ord_{t=0}(f)}; 
\\
|f(\tau(t))| = |f(t)|^{\frac{-1}{\log\mathopen|t\mathclose|}}  \text{ if } 0 < |t| \le 1/e.
\end{cases}
\end{equation} 
for any $f\in \cA$. One can show that this map is a homeomorphism, see e.g.~\cite[Prop. 1.1]{degeneration}.
A note on terminology: an element  $x \in \cD$ is a non-negative real valued function on $\cA$,  
nevertheless as already said it is customary 
to write $f \mapsto |f(x)| = |f|_x \in \R_+$ for the evaluation map. 

\smallskip

The hybrid affine line $\A^1_{\hyb}:= M_{\mathrm{ber}}(\cA[Z])$  
 is by definition the set of multiplicative semi-norms $|\cdot|$ on $\cA[Z]$
such that $|\cdot| \le \norm{\cdot}_{\hyb}$ on $\cA$. We endow it with the topology of the pointwise convergence which makes 
 it   locally compact.  
The restriction to $\cA$ of any semi-norm $x$ is a point $p_{\hyb}(x) \in \cD$, and 
the projection $p_{\hyb} : \A^1_{\hyb}\to \cD$ is a continuous surjective map. 
Given   $x\in \A^1_{\hyb}$, according to the value (zero or non-zero) of $ \tau\inv \circ p_{\hyb}(x)$,  the semi-norm $x$ will carry 
non-Archimedean or Archimedean information. 

It follows from the Gelfand-Mazur theorem  
(see e.g. the proof of \cite[Prop. 1.1]{degeneration})
that $M_{\mathrm{ber}}(\cc[Z]) \simeq \cc$, 
so  if $t\neq 0$, the fiber $p_{\hyb}\inv(\tau(t))$ is homeomorphic to $\cc$. Furthermore, there exists 
  a unique homeomorphism $\tilde \psi: \overline{\D}^*_{1/e}  \times  \C \to p_{\hyb}^{-1}(\tau (\overline{\D}^*_{1/e})) $ 
satisfying \begin{equation}\label{eq:identify}
|g(\tilde\psi(t,z))| = |g(t,z)|^{\frac{-1}{\log \mathopen| t \mathclose|}}
\end{equation}
for any $(t,z) \in  \overline{\D}^*_{1/e}  \times \C $ and any $g\in \cA[Z]$. By construction, we have 
$ p_{\hyb} \circ \tilde\psi=\tau \circ \pi$, where $\pi:  \overline{\D}^*_{1/e}  \times  \C   \to\overline{\D}^*_{1/e} $ 
is the first projection.
On the other hand, any semi-norm on $\Laurent[Z]$ can be restricted to $\cA[Z]$ which yields a canonical map 
$\tilde\psi_{\na}: \A^{1,\an}_{\Laurent} \to p_{\hyb}\inv(\tau(0))$. This map is a homeomorphism since
 the completion of $\cA$ with respect to
 the $t$-adic norm is the field $\Laurent$.

\smallskip

The hybrid space $(\A^1)^*_{\hyb}$ associated to the punctured affine line $(\A^1)^*$ is the Berkovich spectrum of $\cA[Z,Z^{-1}]$
and can be identified with an open subset of $\A^1_{\hyb}$ whose complement is the set $ g \mapsto |g(0)|$
with $|\cdot| \in \cD$. The latter set is the closure in $\A^1_{\hyb}$ of $\psi\lrpar{\overline{\D}^*_{1/e}  \times \{0\} }$. 

\smallskip

The hybrid projective line is constructed as the union of two copies of $\A^1_{\hyb}$ patched in the usual way. 
Specifically, the natural inclusions $ \cA[z_1] \to \cA[Z,Z^{-1}]$ and  $\cA[z_2] \to \cA[Z,Z^{-1}]$  sending $z_1$ to $Z$, and  $z_2$ to $Z^{-1}$
yield two open embeddings $\imath_1, \imath_2\colon(\A^1)^*_{\hyb} \to \A^1_{\hyb}$, and 
$\phyb$ is defined to be 
 the union of $U_1:=M_{\mathrm{ber}}(\cA[z_1])$ and $U_2 :=M_{\mathrm{ber}}(\cA[z_2])$ glued together
using the identification
$\imath_1(x) = \imath_2(x)$ for any $x\in (\A^1)^*_{\hyb}$.

The inclusion $U_1 \subset \PP^1_{\hyb}$ yields an 
open and dense embedding of $\A^1_{\hyb}$ into $\PP^1_{\hyb}$, and the following proposition holds. 
\begin{prop}[\cite{degeneration}]
The hybrid space $\PP^1_{\hyb}$ is compact, and 
there exists a homeomorphism $\psi: \overline{\D}^*_{1/e}  \times \PP^1_\C   \to p_{\hyb}^{-1}(\tau (\overline{\D}^*_{1/e})) $
whose restriction to $\overline{\D}^*_{1/e}  \times \C $ is equal to $\tilde \psi $. 
Likewise, there is a canonical homeomorphism $\psi_\na: \PP^{1,\an}_{\Laurent} \to p_{\hyb}\inv(\tau(0))$
whose restriction to $\A^{1,\an}_{\Laurent}$ is equal to $\tilde \psi_{\na}$.
\end{prop}
\begin{rem}
In other words,  there exists 
 a topology on  the disjoint union $\lrpar{\overline{\D}^*_{1/e}  \times \PP^1_\C } \bigsqcup  \PP^{1,\an}_{\Laurent}$ such that the map 
 defined by $\psi$ on $\overline{\D}^*_{1/e}  \times \PP^1 $ and $\psi_\na$ on $\PP^{1,\an}_{\Laurent}$ is a 
 homeomorphism onto $\phyb$. 
\end{rem}

The group  $\SL(2,\mm)$ is   contained in $\SL(2,\Laurent)$ so it admits a natural action on $\PP^{1,\an}_{\Laurent}$ preserving the analytic structure on this space. 
It also acts by biholomorphisms on $ \D^*\times \PP^1_\C $ commuting with the second projection. 
The next proposition shows that these two actions fit  together nicely in the hybrid space. 
To ease notation we   write $\psi_t(z) = \psi(z,t)$.

\begin{prop}\label{prop:action hybrid}
The group $\SL(2,\mm)$ admits a unique action by homeomorphisms on the hybrid space $\PP^1_{\hyb}$ which is compatible with its natural action on $\PP^{1,\an}_{\Laurent}$, and such that  \begin{equation}\label{eq:321}\psi_t (\gamma_t \cdot z ) = \gamma_t \cdot \psi_t(z)
\end{equation}
 for any $\gamma\in \SL(2,\mm)$, and any $(t,z )\in \overline{\D}^*_{1/e}  \times\PP^1_\C  $. In particular, 
 for all  $\gamma\in \SL(2,\mm)$  and   $x\in \phyb$,
 we have
$p_{\hyb} (\gamma \cdot x) = p_{\hyb} (x)$.
\end{prop}

\begin{proof}
We define an action of $\SL(2,\mm)$ on $\PP^1_{\hyb}$
by setting $\gamma \cdot x = \gamma_{\na} \cdot x$ when $x\in \PP^{1,\an}_{\Laurent}$, and
such that~\eqref{eq:321} holds true. It is only necessary to check that this action is continuous which will
 follow  from the very definition
of the hybrid space. 

Recall that $\PP^1_{\hyb}$ is the union of two copies $U_1$ and $U_2$ of $\A^1_{\hyb}$. We pick $\gamma =  \lrpar{\begin{smallmatrix} a & b \\ c & d \end{smallmatrix}} \in \SL(2,\mm)$ 
and look at the action of $\gamma$ in the first chart
in  $U_1= M_{\mathrm{ber}}(\cA[z])$. 
Observe that for any  $f\in \cA[z]$,    $f(\frac{az+b}{cz+d})$ is the quotient of some $\tilde{f}\in \cA[z]$ by a polynomial of the form $(cz+d)^N\in \mm[z]$ for some integer $N$. It follows that
\[
|f(\gamma \cdot x)| = \left|f\left(\frac{az+b}{cz+d}\right) (x)\right|
\]
depends continuously on $x$ on the open set $U:= \{x\in U_1, \,  |(cz+d)|_x\neq 0\}$, so that $\gamma$ defines a continuous map from $U$ to $U_1$. 

If now  $f\in \cA[z^{-1}]$,  then  $f(\frac{az+b}{cz+d})$ is the quotient of an element $g\in \cA[z]$ by a polynomial of the form $(az+b)^N\in \mm[z]$ for some $N$, and
we conclude similarly that $\gamma$ defines a continuous map from $U':= \{x\in U_1, \,  |(az+b)|_x\neq 0\}$ to $U_2$. 

Since $ad-bc =1$, the two open sets $U$ and $U'$ cover $U_1$, which completes the proof. 
\end{proof}

Recall from \eqref{eq:sigma1} the definition of the cocycle $\sigma$. 

\begin{prop}\label{prop:hybrid sigma}
For any $\gamma\in \SL(2,\mm)$, the function defined by 
\[\sigma_{\hyb}(\gamma,x) :=
 \begin{cases}
\sigma(\gamma_\na,\psi_\na\inv(x)) &  \text{ if } x \in \PP^{1,\an}_{\Laurent}
\\
\frac{\sigma(\gamma_t,z)}{\log\mathopen|t\mathclose|^{-1}}& \text{ if } p_{\hyb}(x) \neq 0 \text{ and } \psi^{-1}(x) = (z,t)
\end{cases}
\]
is continuous on $\PP^1_{\hyb}$.
\end{prop}
\begin{proof}
It is enough to check that the restriction of $\sigma_{\hyb}(\gamma,\cdot)$ to one of the two defining charts of the hybrid projective line is continuous, say on $U = M_{\mathrm{ber}}(\cA[Z])$.
Since $\sigma_{\hyb}(g,\cdot)$ is continuous on $\PP^{1,\an}_{\Laurent}$ and on 
$\psi (\overline{\D}^*_{1/e}  \times \PP^1_\C)$ separately ,
we  only need to check that $\sigma_{\hyb}(g,x_i) \to \sigma_{\hyb}(g,x)$ for any net of points $x_i = \psi(t_i, z_i)$, $(t_i, z_i)\in
 {\D}^*_{1/e}  \times \C  $ indexed by some inductive set $I$ and such that $x_i \to x \in \PP^{1,\an}_{\Laurent}$ along $I$. Note that this implies $t_i \to 0$, and $|g(x_i)| \to |g(x)|$ for any $g\in \cA[Z]$, i.e. 
\[
|g_{t_i}(z_i)|^{\frac{-1}{\log\mathopen| t_i\mathclose|}} \to |g(x)|~.
\]
By switching the chart we are working in and extracting a subfamily if necessary, we
may also suppose that $|Z(x)| \le 1$ and $|z_i| \le 1$ for all $I\in I$.
Let $a,b,c$ and $d\in \mm$ be the coefficients of $g$. By definition, we have 
\[
\sigma_{\hyb}(g,x_i)
=
\frac{\log \max\{ |a(t_i) z_i +b(t_i)|, |c(t_i) z_i + d(t_i)|\}}{\log\mathopen|t_i\mathclose|^{-1}}
~.\]
On the other hand for any $h\in \cA[Z]$ we have  $\lim_i |h(x_i)| = |h(x)|$. 
Since $\mm \subset \cA$,~\eqref{eq:identify} yields
\[
\lim_i |(a Z + b)(x_i)| =\lim_i |a(t_i)(z_i) + b(t_i)|^{\frac{-1}{\log\mathopen| t_i\mathclose|}} = |(aZ+b)(x)|
\]
which implies $\sigma_{\hyb}(g,x_i) \to \sigma_{\hyb}(g,x)$ as required.
\end{proof}

\subsection{Convergence of measures in the hybrid space: proof of Theorem~\ref{Thm:hybrid}}
Consider   a representation $\rho: G \to \SL(2,\mm)$ such that $\rho_{\na}$ is non-elementary, 
and $\sm$   a measure on $G$ satisfying (A1). 

By Theorem~\ref{thm:positive}, we may consider the unique stationary measure $\nu_{\na}$  associated to $\rho_{\na}$: this is a probability measure supported in $\PP^{1,\an}_{\Laurent} = p_{\hyb}^{-1}(\tau(0)) \subset \phyb$. By Lemma~\ref{lem:specialization NE}, for any small enough $t\neq0$  the representation $\rho_t$ is   non-elementary, and we 
denote by $\nu_t$ the image in $ \set{t} \times \PP^1_\C $ of the unique stationary measure  associated to $\rho_t$ under the natural inclusion
$\PP^1_\C \subset \set{t} \times \PP^1_\C $.
We shall see that any limit point of $\nu_t$ is a stationary measure on $\PP^{1,\an}_{\Laurent}$ hence equal to $\nu_{\na}$ so that $\nu_t \to \nu_{\na}$ (here for simplicity we drop the mention to the embeddings $\psi$ and 
$\psi_\na$). Since the hybrid space is not metrizable, some care needs to be taken when arguing in this way, and we thus proceed as follows.

Consider the set of probability measures $M = \{ \psi_*(\nu_t), \,  0 < |t| \le 1/e\}$ in $\phyb$. Let $\overline{M}$ be the closure of $M$ in the space of probability measures, for the weak-$\star$ topology associated to the hybrid topology. 
Since $\PP^{1}_{\hyb}$ is   compact, so does  $\overline{M}$. 

Let us prove that $\overline{M} \setminus M = \{(\psi_\na)_* \nu_{\na}\}$.

We claim that $\overline{M} \setminus M\neq \emptyset$. Inded 
for any $\delta \in (0, 1/e)$, define  $M_\delta =\{ \psi_*(\nu_t), \,  0 < |t| \le \delta\}$. 
This forms an increasing family of subsets of $M$. Observe that any measure in $M_\delta$ has its support included in $p_{\hyb}^{-1}(\tau(\overline{\D}^*_\delta))$. 
The intersection $\bigcap_{\delta>0} \overline{M_\delta}$ is non-empty as an intersection of compact sets, and it is included in  $\overline{M} \setminus M$ since
any measure in this intersection has its support included in $p_{\hyb}^{-1}(\tau(0))$. This proves our claim. This 
 also proves that any measure  in $\overline{M} \setminus M$
is supported on $p_{\hyb}^{-1}(\tau(0)) =  \PP^{1,\an}_{\Laurent}$. 

Pick now any   measure $\nu \in \overline{M} \setminus M$.  Let  $\varphi_{\na} : \PP^{1,\an}_{\Laurent} \to \R$ be  an arbitrary continuous function. Since $\phyb$ is compact, it is a normal topological space, so the  Tietze-Urysohn
 extension lemma applies. We can thus find a continuous function $\phi: \phyb \to \R$ whose restriction to $\PP^{1,\an}_{\Laurent}$ is equal to $\varphi_{\na}\circ \psi_\na\inv$. 
Let us introduce the   convolution operator acting 
on continuous functions of $\phyb$ by setting
\[
\sm \ast \phi (x) = \int_G  \phi( \rho(g)^{-1} \cdot x) d\sm(g)~.
\]
Observe that for any $0 < |t| \le 1/e$, we have 
\begin{align*}
\int_{\PP^1(\cc)} (\sm \ast \phi) d(\psi_*(\nu_t)) &= 
\int_{\PP^1}   \int_G   \phi( \rho(g)^{-1} \cdot x)  d(\psi_*(\nu_t)) \; d \sm(g)
\\ &
= \int_{\PP^1}   \int_G       \phi( \rho(g) ^{-1} \cdot \psi(z,t)) d\nu_t \; d \sm(g)
\\
& \mathop{=}\limits^{\eqref{eq:321}} \int_{\PP^1}   \int_G (\phi \circ \psi_t) ( \rho_t(g)^{-1} \cdot z) d\nu_t
\\ &= \int_{\PP^1}   (\phi \circ \psi_t)  d((\rho_t)_*\sm)  * \nu_t) 
= \int_{\PP^1} \phi \, d(\psi_{*}(\nu_t))
\end{align*}
so by   definition of the weak-$\star$ topology we get 
that  $\int (\sm \ast \phi) d\nu=  \int \phi d\nu$. This implies
that  $(\rho_{\na*}(\sm)) \ast (\psi_\na^*\nu) = \psi_\na^*\nu$, hence $\nu = (\psi_\na)_*\nu_{\na}$
since $\rho_{\na}$ admits a unique stationary measure. 

\smallskip

Finally let us show that $\psi_*(\nu_t) \to(\psi_\na)_* \nu_{\na}$. We argue by contradiction, and pick $\e>0$, a continuous function $\phi$ on $\phyb$, a sequence $t_n\to0$ such that 
$\int \phi \, d(\psi_*(\nu_{t_n}))  \ge \int \phi\, d((\psi_\na)_*\nu_{\na}) + \e$. 
Since $\nu_{\na}$ belongs to the accumulation set $\bigcap_m \overline{\bigcup_{n\ge m}  \{\psi_*(\nu_{t_n})\}}$
of the sequence  $\psi_*(\nu_{t_n})$, it follows that the open set 
$\{ \nu, \, \int \phi \, d\nu < \int \phi\, d((\psi_\na)_*\nu_{\na}) + \e\}$ contains infinitely many measures of the form $\psi_*(\nu_{t_n})$, which is contradictory, thereby finishing the proof. \qed

\subsection{The hybrid approach to Theorem~\ref{Thm:lyapunov na} for non-elementary representations}

Let $\e>0$ be any positive small real number. By condition (A2) there exists  a finite subset $G'$ of $G$ such that 
$\int_{G\setminus G'} \length(g) d \sm(g) \le \e$. 
By Lemma \ref{lem:sup2} we   have
\begin{align*}
\left|
\int_{G\setminus G'} \! \int_{\PP^1_\C} \frac{\sigma(\rho_t(g),v)}{\log|t|^{-1}} d\sm(g)d\nu_t(v)
\right|
&\mathop{\le}\limits^{\eqref{eq:sup2}} 
\int_{G\setminus G'} 
\lrpar{ \log \|\rho(g)\|_{\na}   
+ \frac{C}{\log\mathopen|t\mathclose|^{-1}}
  \length(g) }d\sm(g)
\le 2C\e~,
\end{align*}
for   small enough $t$. Likewise from Lemma \ref{lem:expand vector 1} we infer that 
\[
\left|\int_{G\setminus G'} \! \int_{\PP^1_{\Laurent}} \sigma(\rho_\na(g), v) d\sm(g) \right|
\le 
\left|\int_{G\setminus G'}  \log \|\rho(g)\|_{\na}\right|
\le C\e~. 
\]
So using the above and Furstenberg's formula for the Lyapunov exponent we get 
\begin{align*}
\abs{\frac{\chi(t)}{\log\abs{t}\inv} -  \chi_{\na}}  
&=
\left|
\int_{G} \int_{\PP^1_\C} \frac{\sigma(\rho_t(g),v)}{\log|t|^{-1}} d\sm(g) d\nu_t(v)
- 
\int_{G}\int_{\PP^{1,\an}_{\Laurent}} \sigma(\rho_\na(g), v) d\sm(g)  d\nu_\na(v)
\right|
\\  
& \le 2 C \e + 
\left|
\int_{G'} \! \left( 
\int_{\PP^1_\C} \frac{\sigma(\rho_t(g),v)}{\log|t|^{-1}} d\nu_t(v)
- \int_{\PP^{1,\an}_{\Laurent}} \sigma(\rho_\na(g), v) d\nu_\na(v) 
\right) d\sm(g)
\right|.
\end{align*} 
Viewed in  the hybrid space, the difference of  integrals in the last line  rewrites as 
 $$
\int_\phyb \frac{\sigma_{\hyb}(\rho(g),x)}{\log|t|^{-1}}  d\lrpar{(\psi_t)_*\nu_t }(x) 
- \int_\phyb  {\sigma_{\hyb}(\rho(g),x)}  d \lrpar{(\psi_\na)_*\nu_\na}(x)
 \, ,
$$
so using the finiteness of   $G'$,  Proposition~\ref{prop:hybrid sigma} and 
 Theorem~\ref{Thm:hybrid}  we deduce that  
  $$\limsup_{t\to0} \abs{\log\abs{t}\inv \chi(t)  - \chi_{\na} }\le 2C \e,$$ and we conclude
by letting $\e\to0$.  \qed

%%%%%%%%%%%%%%%%%%%%%%%%%%%%%%%%%%%%%%

\end{document}